\documentclass[11pt]{amsart}
\usepackage{amssymb,latexsym,amsmath,comment}
\usepackage[mathscr]{eucal}
\usepackage{bbm}
\usepackage{youngtab}
\usepackage{xcolor}

\numberwithin{equation}{section}
\numberwithin{table}{section}
\numberwithin{figure}{section}

\def\Kahler{{K\"ahler}}
\def\naive{{na\"ive}}

\def\cf{cf.~}

\newcommand{\rb}[1]{\raisebox{1.5ex}[0pt]{#1}}

\newcommand{\hsp}[1]{{\hbox{\hspace{#1}}}}

\newcommand{\mystack}[2]{\ensuremath{ \substack{ \hbox{\tiny{${#1}$}} \\ \hbox{\tiny{${#2}$}} }} }

\newcounter{letcnt} 

\def\a{\alpha}  
\def\b{\beta}  
\def\d{\delta}

\def\z{\zeta}
\def\m{\mu}

\def\s{\sigma}

\def\w{\omega} 

\def\x{\xi}

\def\tAd{\mathrm{Ad}} 
 
\def\tAut{\mathrm{Aut}}
\def\cB{\mathcal B}

\def\fb{\mathfrak{b}} 
 \def\tbd{\mathrm{bd}}
\def\bC{\mathbb C} \def\cC{\mathcal C}

 \def\tcl{\mathrm{cl}}
\def\tcodim{\mathrm{codim}} 
 \def\sD{\mathscr{D}}

\def\td{\mathrm{d}}

 \def\tdim{\mathrm{dim}}
 \def\sE{\mathscr{E}}
 \def\ttE{\mathtt{E}}
\def\tEnd{\mathrm{End}} 

\def\be{\mathbf{e}} \def\fe{\mathfrak{e}} 
 \def\teven{\mathrm{even}}
\def\texp{\mathrm{exp}}
 
\def\sF{\mathscr{F}}
\def\ff{\mathfrak{f}}  

  \def\ttF{\mathtt{F}}
\def\tFlag{\mathrm{Flag}}

\def\cG{\mathcal G} 

\def\tGr{\mathrm{Gr}}
\def\fg{{\mathfrak{g}}} 

\def\fgl{\mathfrak{gl}}

 \def\ttH{\mathtt{H}}
 \def\sH{\mathscr{H}}
 
\def\tHom{\mathrm{Hom}}
\def\fh{\mathfrak{h}} 
\def\bh{\mathbf{h}}
\def\bi{\mathbf{i}} 
\def\tti{\mathtt{i}}

 \def\tIm{\mathrm{Im}}
 \def\tim{\mathrm{im}}

\def\fk{\mathfrak{k}} 
\def\bbk{\mathbbm{k}}
 
 \def\tker{\mathrm{ker}}
\def\cL{\mathcal L} 
 
\def\fl{\mathfrak{l}}

 \def\tmin{\mathrm{min}}
 
 \def\cN{\mathcal N}

\def\bn{\mathbf{n}}

\def\tNilp{\mathrm{Nilp}}

 \def\cO{\mathcal O}

 \def\todd{\mathrm{odd}}

\def\bP{\mathbb P} \def\cP{\mathcal P}

\def\fp{\mathfrak{p}} 
 
\def\tprim{\mathrm{prim}}

\def\tprin{\mathrm{prin}} 
\def\bQ{\mathbb Q} \def\cQ{\mathcal Q}

\def\bR{\mathbb R}

\def\trank{\mathrm{rank}}
 
\def\sS{\mathscr{S}}
  
\def\ttS{\mathtt{S}} 
\def\fs{\mathfrak{s}}
 
\def\tss{\mathrm{ss}}
\def\tSL{\mathrm{SL}} \def\tSO{\mathrm{SO}}
\def\tSp{\mathrm{Sp}} 
 
\def\tStab{\mathrm{Stab}} 

\def\tSU{\mathrm{SU}}
 \def\tspan{\mathrm{span}}
\def\fsl{\mathfrak{sl}} \def\fso{\mathfrak{so}} 
\def\fsp{\mathfrak{sp}} \def\fsu{\mathfrak{su}}

\def\ft{\mathfrak{t}}

\def\fu{\mathfrak{u}}

 \def\sV{\mathscr{V}}

\def\cW{\mathcal W} \def\sW{\mathscr{W}}

\def\by{\mathbf{y}}
   \def\bZ{\mathbb Z}
 \def\sZ{\mathscr{Z}}
\def\fz{\mathfrak{z}} 
 \def\bz{\mathbf{z}}
\def\half{\tfrac{1}{2}}

\def\one{\mathbbm{1}}
\def\tand{\quad\hbox{and}\quad}

\def\sb{{\hbox{\tiny{$\bullet$}}}}

\def\inj{\hookrightarrow}
\def\sur{\twoheadrightarrow}

\def\op{\oplus}
\def\ot{\otimes}

\def\tw{\hbox{\small $\bigwedge$}}

\def\wtL{{\Lambda_\mathrm{wt}}}

\def\rtL{{\Lambda_\mathrm{rt}}}

\newcounter{numcnt}
\newenvironment{numlist}{
   \begin{list}{{\small{(\arabic{numcnt})}}}
   {\usecounter{numcnt} 
    \setlength{\itemsep}{3pt}
    \setlength{\leftmargin}{25pt} 
    \setlength{\labelwidth}{20pt} 
    \setlength{\listparindent}{20pt} }
   }
   {\end{list}}
\newcounter{cnt}

\newcounter{acnt}
\newenvironment{a_list}{ 
  \begin{list}{{(\alph{acnt})}}
   {\usecounter{acnt} \setlength{\itemsep}{3pt}
    \setlength{\leftmargin}{25pt} 
    \setlength{\labelwidth}{20pt}
    \setlength{\listparindent}{20pt} }
   }
   {\end{list}}
\newenvironment{a_list_emph}{ 
  \begin{list}{{\emph{(\alph{acnt})}}}
   {\usecounter{acnt} \setlength{\itemsep}{3pt}
    \setlength{\leftmargin}{25pt} 
    \setlength{\labelwidth}{20pt}
    \setlength{\listparindent}{20pt} }
   }
   {\end{list}}

\newcounter{Acnt}

\newcounter{icnt}

\newenvironment{i_list}{ 
  \begin{list}{{(\roman{icnt})}}
   {\usecounter{icnt} 
    \setlength{\itemsep}{3pt} 
    \setlength{\parsep}{0pt} 
    \setlength{\topsep}{0pt} 
    \setlength{\leftmargin}{25pt} 
    \setlength{\labelwidth}{20pt}
    \setlength{\listparindent}{20pt} }
   }
   {\end{list}}
\newcounter{Icnt}

\newcounter{exam_cnt}

\newcounter{mccnt}

\newtheorem{corollary}[equation]{Corollary}
\newtheorem{lemma}[equation]{Lemma}
\newtheorem{proposition}[equation]{Proposition}
\newtheorem{theorem}[equation]{Theorem}

\theoremstyle{definition}

\newtheorem*{boldQ*}{Question}
\newtheorem*{boldP*}{Problem}

\theoremstyle{definition}

\theoremstyle{remark}
\newtheorem*{assume*}{Assume}
\newtheorem*{answer*}{Answer}
\newtheorem{claim}[equation]{Claim}
\newtheorem*{claim*}{Claim}

\newtheorem*{definition*}{Definition}
\newtheorem{example}[equation]{Example}
\newtheorem*{example*}{Example}
\newtheorem*{hint*}{Hint}
\newtheorem*{notation*}{Notation}
\newtheorem{remark}[equation]{Remark}
\newtheorem*{remark*}{Remark}
\newtheorem*{remarks*}{Remarks}
\newtheorem*{fact*}{Fact}
\newtheorem*{emphL*}{Lemma}

\newtheorem*{emphQ*}{Question}
\newtheorem*{emphA*}{Answer}

\def\clO{{\cO_\mathrm{cl}}}

\def\tNC{\mathcal{C}}

\def\bi{\hbox{\small{$\mathbf{i}$}}}

\def\CS{MR2532439}
\def\CKpmhs{MR664326}
\def\CKextn{MR0432925}
\def\CKSdeg{MR840721}
\def\CKScoh{MR870728}
\def\CoMc{MR1251060}
\def\DeligneII{MR0498551}

\def\FHW{MR2188135}
\def\GGK{MR2918237}
\def\GGKtcu{MR3115136}
\def\Knapp{MR1920389}

\def\Noel{MR1600330}

\def\Schmid{MR0382272}

\begin{document}
\title{Classification of horizontal $\tSL(2)$s}
\author[Robles]{Colleen Robles}
\email{robles@math.duke.edu}
\address{Mathematics Department, Duke University, Box 90320, Durham, NC 27708-0320} 
\date{\today}
\thanks{This work is supported by the National Science Foundation though the grants DMS-1309238, 1361120, and was undertaken while a member of the Institute for Advanced Study; I am grateful to the institute for a wonderful working environment and the Robert and Luisa Fernholz Foundation for financial support.
}
\begin{abstract}
We classify the horizontal $\tSL(2)$s and $\bR$--split polarized mixed Hodge structures on a Mumford--Tate domain.
\end{abstract}

\maketitle

\section{Introduction}

A variation of (pure, polarized) Hodge structure gives rise to a horizontal holomorphic mapping into a flag domain $D$; here \emph{horizontal} indicates that the image of the map satisfies a system of partial differential equations known as the \emph{infinitesimal period relation} (or \emph{Griffiths' transversality condition}).  Such maps arise as (lifts of) period mappings associated with families of polarized algebraic manifolds.  
The celebrated Nilpotent Orbit and $\tSL(2)$--Orbit Theorems of Schmid \cite{\Schmid} and Cattani--Kaplan--Schmid \cite{\CKSdeg}, describe the asymptotic behavior of a horizontal mapping, and play a fundamental r\^ole in the analysis of singularities of the period mapping (equivalently, degenerations of Hodge structure), \cf the work of Kato and Usui \cite{MR2465224}.  Two of the more striking applications of the Nilpotent and $\tSL(2)$--Orbit Theorems are: (i) Cattani, Deligne and Kaplan's \cite{MR1273413} proof of the algebraicity of Hodge loci, which provides some of the strongest evidence for the Hodge conjecture; and (ii) the proof of Deligne's conjectured isomorphism between the $L^2$ and intersection cohomologies of a polarized variation of Hodge structure with normal crossing singularities over a compact \Kahler~manifold (first proved by Zucker \cite{MR534758} in the case of a one-dimensional base, followed by Cattani and Kaplan's \cite{MR786917} proof in the case of dimension two and weight one, with the general case established, independently, by Cattani, Kaplan and Schmid \cite{MR870728}, and Kashiwara and Kawai \cite{MR804358}), and the resulting corollary that the intersection cohomology carries a pure Hodge structure.

As a consequence it became an important problem to describe the $\tSL(2)$s appearing in Schmid's Theorem, and to that end partial results were obtained by Cattani and Kaplan \cite{MR0432925, MR496761}, and Usui \cite{MR1245714}.  The main result (Theorem \ref{T:cSL2}) of the paper is a classification of those objects.  It is a corollary of Theorem \ref{T:cPMHS}, which classifies the $\bR$--split polarized mixed Hodge structures (PMHS), and the familiar equivalence 
\begin{equation}\label{E:equiv}
  \left\{ 
  \hbox{$\bR$--split PMHS on $D$}
  \right\} \ \longleftrightarrow \ 
  \left\{ \hbox{horizontal $\tSL(2)$s on $D$} \right\} \,,
\end{equation}
which follows from the work of Cattani, Deligne, Kaplan and Schmid \cite{MR664326, CattaniKaplanIHES, MR840721, DeligneApp}. These results are established for Mumford--Tate domains \cite{MR2918237}; the latter generalize period domains to include classifying spaces for Hodge structures with nongeneric Hodge tensors, i.e., the Mumford--Tate group of a generic Hodge structure in the domain need not be the full automorphism group 
\[
  \cG_\bR \ = \ \tAut(V_\bR,Q)\,.
\]
A Mumford--Tate domain $D$ is homogeneous with respect to a real, reductive Lie group 
\[
  G_\bR \ \subset \ \tAut(V_\bR,Q)\,.
\]
In particular, the classification describes (as $\bR$--split PMHS) the degenerations that may arise in a variation of Hodge structure subject to the constraint that (the connected identity component of) the Mumford--Tate group of the generic fibre lies in $G_\bR$.  There is a natural action of $G_\bR$ on both the horizontal $\tSL(2)$s and the $\bR$--split PMHS; the classification theorems enumerate these objects up to the action of $G_\bR$, which we assume to be connected.

Cattani has pointed out that the problem of classifying horizontal $\tSL(2)$--orbits in \emph{period domains} is essentially solved by the possible Hodge diamonds.  This is a consequence of: (i) the equivalence \eqref{E:equiv}; (ii) the classification of subalgebras $\fsl_2\bR \subset \fg_\bR$ by signed Young diagrams when $\fg_\bR$ is classical \cite[Chapter 9]{\CoMc}; and (iii) the fact that the signed Young diagram is determined by the Hodge diamond, \cf \cite{BPR}.  One subtlety to keep in mind here is that the Hodge diamonds suffice to classify the $\tSL(2)$--orbits up to the action of the full automorphism group $\cG_\bR$.  However, in the case of even weight, the Hodge diamonds do \emph{not} suffice to classify the orbits up to the action of the connected identity component $\cG_\bR^\circ$.  This is essentially due to the fact that the signed Young diagrams classify the $\fsl_2\bR \subset \tEnd(V_\bR,Q)$ up to the adjoint action of $\cG_\bR$; and some of these $\cG_\bR$--conjugacy classes decompose into distinct $\cG_\bR^\circ$--conjugacy classes which the signed Young diagram/Hodge diamond fails to distinguish.  (See \S\ref{S:egPMHS} for further discussion and examples.)  I assume throughout that $G_\bR$ is connected.  A second point to keep in mind is that the classification of subalgebras $\fsl_2\bR \subset \fg_\bR$ by signed Young diagrams requires that $\fg_\bR$ not only be classical, but also act by the ``standard representation.''  However, even when $\fg_\bR$ is classical, it may not be possible to realize $D$ as the Mumford--Tate domain for a Hodge representation\footnote{Defined in \S\ref{S:Hrep}.} $(G_\bR,V_\bR,\varphi,Q)$ with $V_\bR$ the standard representation.  This means that in general we will not be able to classify the horizontal $\tSL(2)$--orbits on $D$ by Hodge diamonds when $D$ cannot be realized as a period domain.


A second motivation behind Theorem \ref{T:cPMHS} is the problem to identify polarizable orbits.  Recall that the flag domain $D$ is an open $G_\bR$--orbit in the compact dual $\check D = G_\bC/P$.  In particular, the boundary $\tbd(D) \subset \check D$ is a union of $G_\bR$--orbits.  We say that one of these boundary orbits is \emph{polarizable} if it contains the limit of a nilpotent orbit, \cf\cite{MR3115136, KP2013} and \S\ref{S:rlpm}.\footnote{This notion of a ``polarized'' orbit is distinct from J.~Wolf's in \cite[Definition 9.1] {MR0251246}.  In Wolf's sense, the polarized orbits $\cO = G_\bR \cdot o$ in $\check D$ are those that realize the minimal CR--structure on the homogeneous manifold $G_\bR/\tStab_{G_\bR}(o)$, \cf\cite[Remark 5.5]{MR2668874}.}  We think of these as the ``Hodge theoretically accessible'' orbits.  Then the natural partial order on the $G_\bR$--orbits in $\tbd(D)$ allows one to address, from a Hodge theoretic perspective, the question ``what is the most/least singular variety to which a smooth projective variety can degenerate?''~\cite{GGR}.  Theorem \ref{T:cPMHS}(c) parameterizes the polarizable orbits (\S\ref{S:po}), and from that point of view generalizes \cite[Theorem 6.38]{KR1}.  The parameterization is surjective by definition, and is shown to be injective in the forthcoming \cite{KR2}.  As a corollary to Theorem \ref{T:cPMHS}, and the fact that all codimension--one orbits $\cO \subset \tbd(D)$ are polarized \cite{KP2013}, we obtain a precise count of the number of codimension--one orbits in $\tbd(D)$ (Proposition \ref{P:codim1}); in the case that $P$ is a maximal parabolic, this recovers \cite[Proposition 6.56]{KR1}.

The key observation in the proof of Theorem \ref{T:cPMHS} is that underlying every $\bR$--split PMHS is a Hodge--Tate degeneration (Theorem \ref{T:underHT}), and from the latter we may recover the original $\bR$--split PMHS.  Consequently, the \emph{sine qua non} of the paper is the classification of the Hodge--Tate degenerations (Theorem \ref{T:cHT}).  Theorem \ref{T:underHT} may be viewed as describing the branching of a $\fg_\bR$--Hodge representation under a Levi algebra $\fl_\bR \subset \fg_\bR$, \cf~Remark \ref{R:underHT1}.  Let $L_\bR \subset G_\bR$ be the connected Lie subgroup with Lie algebra $\fl_\bR$.  As a corollary to Theorem \ref{T:underHT}, Mal'cev's Theorem and a result of Cattani, Kaplan and Schmid we find that the (open) nilpotent cone $\cC \subset\fg_\bR$ underlying a nilpotent orbit is contained in an $\tAd(L^Y_\bR)$--orbit, where $L^Y_\bR \subset L_\bR$ is a connected, reductive Lie group (Corollary \ref{C:Ad-orb}).

Both the statements of the classification theorems and their proofs are couched in representation theory; the necessary background material is reviewed in \S\ref{S:rtb}.  Both Levi subalgebras, and their ``distinguished'' parabolic subalgebras, play a key r\^ole in the classification theorems.  This is not surprising as Bala and Carter's classification \cite{MR0417306, MR0417307} of the $\fsl_2\bC$'s in a complex semisimple $\fg_\bC$ is in terms of these pairs.  Indeed, Theorem \ref{T:cSL2} could be viewed as the analog the Bala--Carter classification for horizontal $\fsl_2\bR$'s, and from this perspective is related to both Vinberg's classification \cite{MR0506488} of nilpotent elements of graded Lie algebras, and No\"el's classification \cite{\Noel} of (not necessarily horizontal) $\fsl_2\bR \subset \fg_\bR$.  The pertinent Hodge--theoretic material is reviewed in \S\ref{S:htb}.

As I hope the examples presented here (most are concentrated in \S\S\ref{S:egHT} and \ref{S:egPMHS}) demonstrate, the classifications are computationally accessible: it is straightforward to describe the horizontal $\tSL(2)$s and the Deligne splittings of the associated $\bR$--split PMHS.  Here is one illustrative example.

\begin{example}
The exceptional simple Lie group $F_4$ of rank four admits a real form $G_\bR$ with maximal compact subalgebra $\fsp(2) \op \fsu(2)$.\footnote{This real form is commonly denoted by $\mathrm{F\,I}$ or $F_4(4)$.}  This real form admits a real Hodge representation $V_\bR$ with Hodge numbers $(6,14,6)$;\footnote{The highest weight of $V_\bC$ the fourth fundamental weight.} in particular, $G_\bR \subset \tSO(14,12)$.  The horizontal distribution is a holomorphic contact distribution\footnote{See \cite{KR1} for further discussion of the case that the horizontal distribution is contact.} on the associated domain $D$.  Theorem \ref{T:cSL2} identifies four horizontal $\tSL(2)$s.  The Hodge diamonds of the corresponding $\bR$--split PMHS are depicted below; see \S\ref{S:egPMHS} for further explanation of these diagrams and the notations $\fl_\bR^\tss$, $\sS'$, $\sZ$ and $\cO$ below.
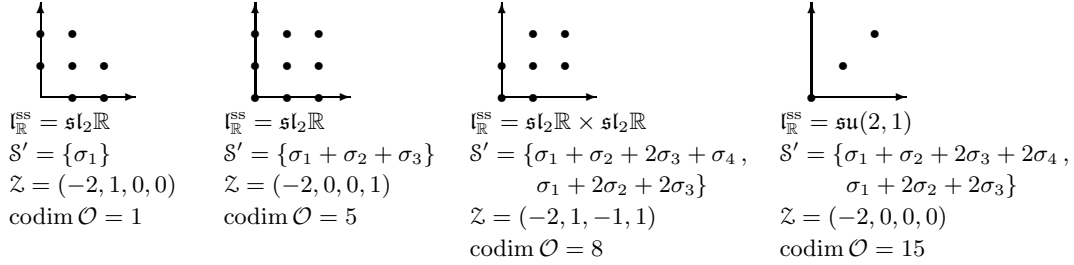
\begin{center}
\begin{footnotesize}
\setlength{\unitlength}{12pt}
\begin{picture}(5,9)(0,-5)
\put(0,0){\vector(1,0){3}} \put(0,0){\vector(0,1){3}}  \put(0,1){\circle*{0.25}} \put(0,2){\circle*{0.25}} \put(1,0){\circle*{0.25}} \put(1,1){\circle*{0.25}} \put(1,2){\circle*{0.25}} \put(2,0){\circle*{0.25}} \put(2,1){\circle*{0.25}}
\put(-1,-1){$\fl_\bR^\tss = \fsl_2\bR$}
\put(-1,-3){$\sZ = (-2,1,0,0)$}
\put(-1,-2){$\sS' = \{ \s_1\}$}
\put(-1,-4){$\tcodim\,\cO = 1$}
\end{picture}
\hspace{15pt}
\begin{picture}(6,9)(0,-5)
\put(0,0){\vector(1,0){3}} \put(0,0){\vector(0,1){3}}  \put(0,0){\circle*{0.25}} \put(0,1){\circle*{0.25}} \put(0,2){\circle*{0.25}} \put(1,0){\circle*{0.25}} \put(1,1){\circle*{0.25}} \put(1,2){\circle*{0.25}} \put(2,0){\circle*{0.25}} \put(2,1){\circle*{0.25}} \put(2,2){\circle*{0.25}}
\put(-1,-1){$\fl_\bR^\tss = \fsl_2\bR$}
\put(-1,-3){$\sZ = (-2,0,0,1)$}
\put(-1,-2){$\sS' = \{ \s_1+\s_2+\s_3\}$}
\put(-1,-4){$\tcodim\,\cO = 5$}
\end{picture}
\hspace{15pt}
\begin{picture}(8,9)(0,-5)
\put(0,0){\vector(1,0){3}} \put(0,0){\vector(0,1){3}}  \put(0,0){\circle*{0.25}} \put(0,1){\circle*{0.25}} \put(1,0){\circle*{0.25}} \put(1,1){\circle*{0.25}} \put(1,2){\circle*{0.25}} \put(2,1){\circle*{0.25}} \put(2,2){\circle*{0.25}}
\put(-1,-1){$\fl_\bR^\tss = \fsl_2\bR \times\fsl_2\bR$}
\put(-1,-4){$\sZ = (-2,1,-1,1)$}
\put(-1,-2){$\sS' = \{ \s_1+\s_2+2\s_3+\s_4 \,,$}
\put(1.1,-3){$\s_1+2\s_2+2\s_3\}$}
\put(-1,-5){$\tcodim\,\cO = 8$}
\end{picture}
\hspace{15pt}
\begin{picture}(8,9)(0,-5)
\put(0,0){\vector(1,0){3}} \put(0,0){\vector(0,1){3}}  \put(0,0){\circle*{0.25}} \put(1,1){\circle*{0.25}} \put(2,2){\circle*{0.25}}
\put(-1,-1){$\fl_\bR^\tss = \fsu(2,1)$}
\put(-1,-4){$\sZ = (-2,0,0,0)$}
\put(-1,-2){$\sS' = \{ \s_1+\s_2+2\s_3+2\s_4 \,,$}
\put(1.1,-3){$\s_1+2\s_2+2\s_3\}$}
\put(-1,-5){$\tcodim\,\cO = 15$}
\end{picture}
\end{footnotesize}
\end{center}
\end{example}

Finally, I wish to mention that an inductive argument based on Theorem \ref{T:cSL2} yields a classification of the commuting $\tSL(2)$s in Cattani, Kaplan and Schmid's several--variables $\tSL(2)$--Orbit Theorem, as will be demonstrated in the forthcoming work \cite{KR2} with Matt Kerr in which we will also establish the injectivity of the parameterization of the polarized orbits by Theorem \ref{T:cPMHS}.

\tableofcontents

\subsubsection*{Einstein summation convention}  When an index appears as both a subscript and a superscript in a formula, it is meant to be summed over.  For example, $z^iN_i$ denotes $\sum_i z^i N_i$.

\subsection*{Acknowledgements}  
Over the course of this work I have benefitted from conversations and correspondence with several colleagues; I would especially like to thank Eduardo Cattani, Mark Green, Phillip Griffiths, Tatsuki Hayama, Matt Kerr, William M.~McGovern and Greg Pearlstein.  I also thank the anonymous referee for ameliorating suggestions.

\section{Representation theory background} \label{S:rtb}

\subsection{Parabolic subgroups and subalgebras} \label{S:P}

Let $G_\bC$ be a connected, complex semisimple Lie group, and let $P \subset G_\bC$ be a parabolic subgroup.  Fix \emph{Cartan} and \emph{Borel subgroups} $H \subset B \subset P$.  Let $\fh \subset \fb \subset\fp\subset\fg$ be the associated Lie algebras.  The choice of Cartan determines a set of \emph{roots} $\Delta = \Delta(\fg,\fh) \subset \fh^*$.  Given a root $\a\in \Delta$, let $\fg^\a \subset \fg$ denote the \emph{root space}.  Given a subspace $\fs \subset \fg$, let 
$$
  \Delta(\fs) \ := \ \{ \a\in\Delta \ | \ \fg^\a \subset \fs \} \,.
$$ 
The choice of Borel determines \emph{positive roots} $\Delta^+ = \Delta(\fb) = \{ \a\in\Delta \ | \ \fg^\a \subset \fb \}$.  Let $\sS = \{\s_1,\ldots,\s_r\}$ denote the \emph{simple roots}, and set 
\begin{equation} \label{E:I}
  I \ = \ I(\fp) \ := \ \{ i \ | \ \fg^{-\s_i} \not\subset \fp \}\,.
\end{equation}
Note that the parabolic $\fp$ is maximal if and only if $I = \{\tti\}$; in this case we say that $\s_\tti$ is the \emph{simple root associated with the maximal parabolic $\fp$}.  Likewise, $\fp = \fb$ if and only if $I = \{1,\ldots,r\}$.

Every parabolic $P \subset G_\bC$ is $G_\bC$--conjugate to one containing $B$.  Thus, the conjugacy classes $\cP_I$ of parabolic subgroups are indexed by the subsets $I \subset \{ 1 , \ldots , r\}$.  Let $\cB = \cP_{\{1,\ldots,r\}}$ denote the conjugacy class of the Borel subgroups.  

\subsection{Grading elements and Levi subalgebras} \label{S:GE}

Given a choice of Cartan subalgebra $\fh \subset \fg_\bC$, let $\rtL \subset \fh^*$ denote the root lattice.  The set of \emph{grading elements} is the lattice $\tHom(\rtL,\bZ) \subset \fh$ taking integral values on roots.  As an element of the Cartan subalgebra a grading element $\ttE$ is necessarily semisimple.  Therefore, any $\fg_\bC$ module $V_\bC$ decomposes into a direct sum of $\ttE$--eigenspaces
\begin{equation} \label{E:VE}
  V_\bC \ = \ \bigoplus_{\ell \in \bQ} V^\ell
  \quad\hbox{where}\quad
  V^\ell \ = \ \{ v \in V_\bC \ | \ \ttE(v) = \ell v \} \,.\footnote{To see that the eigenvalues are necessarily rational, it suffices to observe that the eigenvalues are $\lambda(\ttE)$, where $\lambda \in \wtL \subset \fh^*$ is a weight of $V_\bC$, and to recall that the weights of $\fg_\bC$ are rational linear combinations of the roots.}
\end{equation}

When specialized to $V_\bC = \fg_\bC$, \eqref{E:VE} yields
\begin{subequations} \label{SE:grading}
\begin{equation}
  \fg \ = \ \bigoplus_{\ell\in\bZ} \fg^\ell
\end{equation}
where
\begin{equation} \label{E:grT}
  \fg^\ell \ := \ \{ \xi \in \fg \ | \ [\ttE,\xi] = \ell \xi \} \,.
\end{equation}
In terms of root spaces, we have
\begin{equation} \label{E:gr1} \renewcommand{\arraystretch}{1.3}
\begin{array}{rcl}
  \displaystyle \fg^\ell & = & 
  \displaystyle \bigoplus_{\a(\ttE)=\ell} \fg^\a \,,\quad \ell \not=0 \,,\\
  \displaystyle \fg^0 & = & 
  \displaystyle \fh \ \op \ \bigoplus_{\a(\ttE)=0} \fg^\a \,.
\end{array}
\end{equation}
\end{subequations}
The $\ttE$--eigenspace decomposition \eqref{SE:grading} is a \emph{graded Lie algebra decomposition} in the sense that 
\begin{equation}\label{E:gr}
  \left[ \fg^\ell  , \fg^m \right] \ \subset \ \fg^{\ell+m} \,,
\end{equation}
a straightforward consequence of the Jacobi identity.  It follows that 
\begin{equation} \label{E:p}
  \fp_\ttE \ = \ \fg^0 \,\op \,\fg^+ 
\end{equation}
is a Lie subalgebra of $\fg_\bC$; we call this the \emph{parabolic subalgebra determined by the grading element $\ttE$}.  

From \eqref{E:gr} we also see that $\fg^0$ is a Lie subalgebra of $\fg$ (in fact, reductive), and each $\fg^\ell$ is a $\fg^0$--module.  In general, by \emph{Levi subalgebra} we will mean any subalgebra $\fl_\bC \subset \fg_\bC$ that can be realized as the $0$--eigenspace $\fg^0$ of a grading element.  

\begin{remark}\label{R:levi}
By the second equation of \eqref{E:gr1} every Levi subalgebra contains a Cartan subalgebra of $\fg_\bC$.  Fix a Cartan subalgebra $\fh \subset \fg_\bC$ and recall that the Weyl group $\sW \subset \tAut(\fh^*)$ is generated by the reflections in the hyperplanes orthogonal to the roots $\a \in \Delta$.  Fix a choice of simple roots $\sS \subset \Delta \subset \fh^*$.  The Levi subalgebras containing $\fh$ are bijective correspondence with the subsets $\{ w\sS' \subset \Delta \ | \ w \in \sW \,,\ \sS' \subset \sS \}$: $w\sS'$ is a set of simple roots for the semisimple factor $\fl^\tss_\bC = [\fl_\bC,\fl_\bC]$ of the Levi subalgebra $\fl_\bC \supset \fh$.   In particular, there exist only finitely many Levi subalgebras containing $\fh$.  
\end{remark}

\begin{remark}\label{R:weyl}
Recall that the simple reflections $(i) \in \sW$ in the hyperplanes orthogonal to the simple roots $\s_i \in \sS$ form a minimal set of generators for the Weyl group.  Given a Levi subalgebra $\fl_\bC \supset \fh$, by replacing $\sS$ with $w\sS$ (the latter is also a set of simple roots for $\fh$), we may assume that the simple roots of $\fl_\bC^\tss$ are a subset $\sS'$ of the simple roots $\sS$ of $\fg_\bC$.  Then the Weyl group $\sW'$ of $\fl_\bC$ is generated by the simple reflections $(i) \in \sW$ with $\s_i \in \sS'$.  
\end{remark}

\noindent Given a real form $\fg_\bR$ of $\fg_\bC$, we will say that $\fl_\bR \subset \fg_\bR$ is a Levi subalgebra if the complexification $\fl_\bC = \fl_\bR \ot_\bR \bC$ is a Levi subalgebra of $\fg_\bC$; equivalently, a \emph{Levi subalgebra of the real form $\fg_\bR$} is the real form $\fl_\bR$ of a conjugation--stable Levi subalgebra $\fl_\bC \subset \fg_\bC$.  

Let $\{ \ttS^1 , \ldots , \ttS^r\}$ be the basis of $\fh$ dual to the simple roots $\{\s_1,\ldots,\s_r\}$.  Then any grading element $\ttE = n_i \ttS^i$ is an integral linear combination of the $\{\ttS^i\}$; if $\fp_\ttE$ contains the Borel $\fb \supset \fh$ determining the simple roots, then $n_i \ge 0$ for all $i$.  In this case, the index set \eqref{E:I} is 
\[
  I(\fp_\ttE) \ = \ \{ i \ | \ n_i > 0 \} \,,
\]
and the reductive Levi subalgebra $\fg^0 = \fg^0_\mathrm{ss} \op \fz$ has center $\fz = \tspan_\bC\{ \ttS^i \ | \ i \in I(\fp_\ttE) \}$ and semisimple subalgebra $\fg^0_\mathrm{ss} = [\fg^0,\fg^0]$.  A set of simple roots for $\fg^0_\mathrm{ss}$ is given by $\sS(\fg_0) = \{ \s_j \ | \ j \not \in I(\fp_\ttE) \}$.  I emphasize that the sets $\sS(\fg^0)$ and $I(\fp_\ttE)$ encode the same information which describes the $G_\bC$--conjugacy class $\cP_\ttE$ of the parabolic subgroup $P_\ttE$.

Two distinct grading elements may determine the same parabolic $\fp$.  For example, any positive multiple $n\ttS^\tti$ will determine the same (maximal) parabolic as $\ttS^\tti$.  However given a parabolic $\fp$, and a choice of Cartan and Borel subalgebras $\fh \subset \fb\subset\fp$, there is a canonical choice of grading element $\ttE =\ttE_\fp$ with $\fp_\ttE = \fp$ such that $\fg^{\pm1}$ generates the nilpotent $\fg^\pm$ as a subalgebra.  The \emph{grading element associated to} $\fp\supset\fb\supset \fh$ is
\begin{equation} \label{E:ttE}
  \ttE_\fp \ := \ \sum_{i \in I(\fp)} \ttS^i \,.
\end{equation}

For more detail on grading elements and parabolic subalgebras see \cite[\S2.2]{MR3217458} and the references therein.  

\subsection{Standard triples and TDS} \label{S:sl2trp}

Let $\fg$ be a Lie algebra defined over a field $\bbk = \bR,\bC$.  A \emph{standard triple} in $\fg$ is a set of three elements $\{ N^+ , Y , N\} \subset \fg$ such that 
$$
  [Y , N^+] \ = \ 2\,N^+ \,,\quad
  [N^+,N] \ = \ Y \tand
  [Y,N] \ = \ -2\,N \,.
$$
Note that $\{N^+,Y,N\}$ span a three--dimensional semisimple subalgebra (TDS) of $\fg$ isomorphic to $\fsl_2\bbk$.  We call $Y$ the \emph{neutral element}, $N$ the \emph{nilnegative element} and $N^+$ the \emph{nilpositive element}, respectively, of the standard triple.  The Jacobson--Morosov theorem asserts that every nilpotent $N \in \fg$ can be realized as the nilnegative of a standard triple.

\begin{example} \label{eg:stdtri_1}
The matrices 
\begin{equation} \label{E:stdtri_sl2}
  \bn^+ \ = \ \left(\begin{array}{cc} 0 & 1 \\ 0 & 0 \end{array}\right) \,,\quad 
  \by \ = \ \left(\begin{array}{cc} 1 & 0 \\ 0 & -1 \end{array}\right) \tand
  \bn \ = \ \left(\begin{array}{cc} 0 & 0 \\ 1 & 0 \end{array}\right)
\end{equation}
form a standard triple in $\fsl_2\bR$; while the matrices
\begin{equation} \label{E:stdtri_su11}
  \overline\be \ = \ \half \left(\begin{array}{cc} -\bi & 1 \\
            1 & \bi \end{array}\right) \,,\quad
  \bz \ = \ \left(\begin{array}{cc} 0 & -\bi \\ 
           \bi & 0 \end{array}\right) \tand
  \be = \half \left(\begin{array}{cc} \bi & 1 \\ 
            1 & -\bi \end{array}\right)
\end{equation}
form a standard triple in $\fsu(1,1)$.
\end{example}

\subsection{Jacobson--Morosov filtrations} \label{S:JMf}

Given a standard triple $\{N^+ , Y , N \} \subset \fg$ and a representation $\fg \inj \tEnd(V)$ of $\fg$, the theory of $\fsl_2\bbk$--representations implies that the eigenvalues $\ell$ of $Y$ are integers.  In the case that $V = \fg$, this implies that 
\begin{equation}\label{E:Y=GE}
\hbox{\emph{the neutral element $Y$ is a grading element.}}
\end{equation}

The \emph{Jacobson--Morosov filtration} (or \emph{weight filtration}) \emph{of $N$} is the unique filtration $W_\sb(N,V)$ of $V$ with the properties:
\begin{i_list}
\item
The filtration is increasing, $W_\ell(N,V) \subset W_{\ell+1}(N,V)$.
\item The nilpotent $N$ maps $W_\ell(N,V)$ into $W_{\ell-2}(N,V)$.
\item The induced map $N^\ell : \tGr_\ell W(N,V) \to \tGr_{-\ell}W(N,V)$ is isomorphism for all $\ell \ge 0$, where 
\[
  \tGr_k W(N,V) \ := \ W_k(N,V)/W_{k-1}(N,V) \,.
\]
\end{i_list}
If $V = \op V_\ell$ is the $Y$--eigenspace decomposition, $V_\ell = \{ v \in V \ | \ Y(v) = \ell v\}$, then
\begin{equation} \label{E:JMf}
  W_\ell(N,V) \ = \ \bigoplus_{m \le \ell} V_\ell \,.
\end{equation}
Note that 
\[
  W_0(N^+,\fg) \ = \ \fp_Y \,.
\]
Parabolic subalgebras of the form $W_0(N^+,\fg)$ are \emph{Jacobson--Morosov parabolics}.

\begin{remark}
  \emph{(a)}
Some parabolic subalgebras cannot be realized as Jacobson--Morosov parabolics, \cf Example \ref{eg:JMsl4}.  Similarly, not every grading element can be realized as the neutral element of a standard triple.
  
  \emph{(b)}
The neutral element $Y$ may not be a grading element $\ttE_\fp$ canonically associated with $\fp = W_0(N^+,\fg) \supset\fb\supset \fh$ by \eqref{E:ttE}.  Moreover, it is possible that there exist nilpotents $N_1$ and $N_2$ that are \emph{not} congruent under the action of $\tAd(G)$ on $\fg$ (equivalently, $Y_1$ and $Y_2$ are not congruent), but with $W_0(N_1,\fg) = W_0(N_2,\fg)$.  For an illustration of this, consider Example \ref{eg:SL_1} where we have $W_0(N_{[3,1]},\fg_\bC) = W_0(N_{[2,1^2]})$, but $Y_{[3,1]} = 2 Y_{[2,1^2]} = 2(\ttS^1 + \ttS^3)$.
\end{remark}

\subsection{$\tAd(G_\bC)$--orbits in $\tNilp(\fg_\bC)$} \label{S:rtno_C}

Given any Lie algebra $\fg$, let $\tNilp(\fg)$ denote the set of nilpotent elements.  A \emph{nilpotent orbit} is an $\tAd(G)$--orbit in $\tNilp(\fg)$.\footnote{Here we have a conflict in the nomenclature: the term ``nilpotent orbit'' is used in both representation theory and Hodge theory to refer to two distinct, but related objects (see \S\ref{S:NC} for the second).  Which of the two meanings is intended should be clear from the context.}  In this section we will review some properties of nilpotent orbits in a complex semisimple Lie algebra $\fg_\bC$, including their classification by ``characteristic vectors'' (a.k.a.~ ``weighted Dynkin diagrams'');\footnote{In the case that $\fg_\bC$ is a classical Lie algebra, the nilpotent orbits may be classified by partitions (or Young diagrams), see Example \ref{eg:SL_1} and \cite{BPR}.} an excellent reference for the discussion that follows is \cite{\CoMc}.

Given a nilpotent $N \in \fg_\bC$, fix a standard triple $\{ N^+ , Y , N\}$.  We may choose a Cartan subalgebra $\fh \subset \fg_\bC$ and a set of simple roots $\sS = \{\s_1,\ldots,\s_r\} \subset \fh^*$ such that $Y \in \fh$ and $\s_i(Y) \ge 0$ for all $i$.  The \emph{(complex) characteristic vector} 
\[
  \s(Y) \ := \ (\s_1(Y),\ldots,\s_r(Y))
\]
is independent of our choices, and is an invariant of the nilpotent orbit $\cN = \tAd(G_\bC)\cdot N \subset \fg_\bC$ through $N$, so that 
\[
  \s(\cN) \ := \ \s(Y)
\]  
is well--defined.  For the \emph{trivial orbit} $\cN_\mathrm{triv}=\{ 0 \} \subset \tNilp(\fg_\bC)$ we have $\s(\cN_\mathrm{triv}) = (0,\ldots,0)$.  The nilpotent orbits are characterized by their characteristic vectors: the following is \cite[Theorem 8.3]{MR0047629_trans}, see also \cite[Lemma 5.1]{MR0114875}.

\begin{theorem}[Dynkin] \label{T:Dynkin}
The characteristic vector $\s(\cN)$ is a complete invariant of a nilpotent orbit; that is, $\s(\cN) = \s(\cN')$ if and only if $\cN = \cN'$.  Moreover, $0 \le \s_i(\cN) \le 2$.
\end{theorem}

\begin{example}[Nilpotent orbits in $\fg_\bC = \fsl_n\bC$] \label{eg:SL_1}
The $\tAd(G_\bC)$--orbits in $\tNilp(\fg_\bC)$ are indexed by partitions $\mathbf{d} = [d_i]$ of $n$ \cite[Chapter 3]{\CoMc}.  Given a partition, the corresponding characteristic vector is obtained as follows.  From a part $d_i$, we construct a set $(d_i) = \{ d_i-1 \,,\, d_i-3 \,,\ldots,\, 3-d_i \,,\, 1-d_i \}$.  Take the union of these sets, re-ordering into a non-increasing sequence $\cup_i (d_i) = \{h_1 \ge \ldots \ge h_n\}$.  Then the characteristic vector of the orbit $\cN_\mathbf{d}$ indexed by $\mathbf{d}$ is 
\[
  \s(\cN_\mathbf{d}) \ = \ ( h_1-h_2 \,,\, h_2 - h_3 \,,\ldots,\, h_{n-1}-h_n ) \,.
\]
For example, in the case that $n=4$ there are five nilpotent orbits, indexed by
\begin{eqnarray*}
  \s( \cN_{[4]}) \ = \ (2,2,2) \,, & \quad &
  \s( \cN_{[3,1]}) \ = \ (2,0,2) \,,\\
  \s( \cN_{[2^2]}) \ = \ (0,2,0) \,,& \quad &
  \s( \cN_{[2,1^2]}) \ = \ (1,0,1) \,,\\
  \s( \cN_{[1^4]}) \ = \ (0,0,0) \,.  
\end{eqnarray*}
\end{example}

The index set $I$ (\S\ref{S:P}) corresponding to the conjugacy class of the Jacobson--Morosov parabolic $W_0(N,\fg)$ is 
\[
  I \ = \ \{ i \ | \ \s_i(\cN) \not=0 \}\,.
\]
Equivalently, the simple roots of the reductive Levi factor are
\[
  \sS(\fg_0) \ = \ \{ \s_j \ | \ \s_j(\cN) = 0 \} \,.
\]

\begin{example}[Jacobson--Morosov parabolics in $\fg_\bC = \fsl_4\bC$]\label{eg:JMsl4}
The group $G_\bC$ contains $2^3-1 = 7$ conjugacy classes $\cP_I$ of parabolic subgroups, indexed by nonempty $I \subset \{1,2,3\}$.  From Example \ref{eg:SL_1} we see that only three of the conjugacy classes are Jacobson--Morosov: the corresponding index sets are $I = \{2\} \,,\ \{1,3\}\,,\ \{1,2,3\}$.
\end{example}

The neutral element $Y$ is \emph{even} if the $Y$--eigenvalues are all even.  From Theorem \ref{T:Dynkin} we see that the neutral element $Y$ is even if and only if $\s_i(\cN) \in \{0,2\}$ for all $i$. Equivalently, 
\begin{equation} \label{E:even}
\begin{array}{l}\hbox{\emph{the neutral element $Y$ is even if and only if it is twice the grading element \eqref{E:ttE}}}\\
\hbox{\emph{canonically associated with a choice of Cartan and Borel $\fh \subset \fb \subset W_0(N^+,\fg_\bC)$.}}
\end{array}
\end{equation}
When $Y$ is even we say that $W_0(N^+,\fg_\bC)$ is an \emph{even Jacobson--Morosov parabolic}. 

There is a unique Zariski open orbit $\cN_\mathrm{prin} \subset \tNilp(\fg_\bC)$ of dimension $\tdim\,\fg_\bC \,-\, \trank\,\fg_\bC$; this is the \emph{principal} (or \emph{regular}) \emph{nilpotent orbit}.  The orbit is represented by $N = \xi^1 + \cdots + \x^r$ with each simple root vector $\x^i \in \fg^{\s_i}$ nonzero.  In this case the characteristic vector is
\[
  \s(\cN_\tprin) \ = \ (2,2,\ldots,2) \,.
\]
In particular,
\begin{equation} \label{E:BEP}
\hbox{\emph{the Borel $B \subset G_\bC$ is an even Jacobson--Morosov parabolic.}}
\end{equation}

\subsection{Compact roots} \label{S:cpt_rts}

Let $G_\bR$ be a real semisimple Lie algebra.  Fix a Cartan decomposition $\fg_\bR = \fk_\bR \op \fk_\bR^\perp$.  There is a classification of nilpotent orbits in $\fg_\bR$ that is analogous to that of Theorem \ref{T:Dynkin} in the sense that the orbits are enumerated by characteristic vectors that are given by the roots of $\fk_\bC$.  This classification is reviewed in \S\ref{S:rtno_R}; in anticipation of that discussion we briefly recall the relationship between the roots of $\fg_\bC$ and the roots of $\fk_\bC$.  

Fix a Cartan subalgebra $\ft \subset \fk_\bR$.  Let $\fh$ be a Cartan subalgebra of $\fg_\bC$ containing $\ft \ot_\bR \bC$.  Given a choice of simple roots $\sS = \{ \s_1,\ldots,\s_r\} \subset \fh^*$ of $\fg_\bC$, let $\tilde \a$ denote the highest root, and set
\[
  \sS_\mathrm{ext} \ := \ \{ \sS \} \,\cup\, \{-\tilde\a\} \,.
\]
For a suitable choice\footnote{This means we may need to replace $\sS$ with its image $w\sS$ under an element $w \in \sW$ of the Weyl group.} of $\sS$ there exists a subset $\sS_\fk \subset \sS_\mathrm{ext}$ such that $\left.\sS_\fk\right|_{\ft \ot_\bR \bC}$ is a set of simple roots of $\fk_\bC$.  We will assume throughout that $\trank\,\fk_\bC = \trank\,\fg_\bC$,\footnote{This is the case when $G_\bR$ may be realized as a Mumford--Tate group \cite{\GGK}.} so that $\fh = \ft \ot_\bR \bC$ is a Cartan subalgebra of both $\fk_\bC$ and $\fg_\bC$.   There are two cases to consider:
\begin{a_list}
\item
If $\fg_\bR$ is of Hermitian symmetric type, then $\fk_\bR$ is reductive with a one--dimensional center and we may take $\sS_\fk \subset \sS$.  In this case, the center of $\fk_\bC$ is spanned by the grading element dual to the simple noncompact root $\{\a'\} = \sS\backslash\sS_\fk$.
\item
If $\fg_\bR$ is not of Hermitian symmetric type, then $\fk_\bR$ is semisimple and $-\tilde\a \in \sS_\fk$.
\end{a_list}
In both cases $\sS \backslash \sS_\fk$ consists of a single simple root $\a'$, which we will refer to as \emph{the noncompact simple root}.\footnote{The root $\a'$ corresponds to the painted node in the Vogan diagram of $\fg_\bR$, \cf\cite[\S VI.8]{\Knapp}.}

\begin{example}
The algebra $\fg_\bR = \fsu(p,q)$ is of Hermitian symmetric type.  In this case we have $\a' = \s_p$ and $\sS_\fk = \sS \backslash \{ \s_p \} \subset \sS$.
\end{example}

\begin{example}
For the algebra $\fg_\bR = \fso(2p,2q+1)$, we have $\a' = \s_p$.  This real form is of Hermitian symmetric type if and only if $p = 1$.
\end{example}

\begin{example}
The algebra $\fg_\bR=\fsp(r,\bR)$ is of Hermitian symmetric type; in this case $\a'= \s_r$.  The real forms $\fsp(p,r-p)$, with $p \ge 1$, are not of Hermitian symmetric type; in this case $\a' = \s_p$.
\end{example}

\subsection{$\tAd(G_\bR)$--orbits in $\tNilp(\fg_\bR)$} \label{S:rtno_R}

This section is a terse review of the classification of the nilpotent orbits in a real semisimple Lie algebra $\fg_\bR$ by the \emph{Djokovi{\'c}--Kostant--Sekiguchi correspondence}\footnote{The correspondence was conjectured by Kostant, and proved independently by Djokovi{\'c} \cite{MR891636} and Sekiguchi \cite{MR867991}.}
\begin{equation} \label{E:DKS}
  \left\{ 
  \hbox{nilpotent $\tAd(G_\bR)$--orbits in $\fg_\bR$}
  \right\} 
  \quad \stackrel{\mathrm{bij}}{\longleftrightarrow}\quad
  \left\{ 
  \hbox{nilpotent $\tAd(K_\bC)$--orbits in $\fk^\perp_\bC$}
  \right\} \,.  
\end{equation}
For details, consult \cite[\S9]{\CoMc} and the references therein.

The correspondence is realized through refinements of the standard triples of \S\ref{S:sl2trp}.  Let $G_\bR$ be a real semisimple Lie algebra.  Fix a Cartan decomposition $\fg_\bR = \fk_\bR \op \fk_\bR^\perp$, and let $\theta$ be the associated Cartan involution.  A \emph{Cayley triple} is a standard triple $\{ N^+ , Y , N\}$ of $\fg_\bR$ with the property that 
\begin{equation} \label{E:Ctrp}
  \theta(N) \ = \ -N^+ \,,\quad \theta(N^+) \ = \ 
  -N \tand \theta(Y) \ = \ -Y \,.
\end{equation}

\begin{remark} \label{R:cayley}
Every standard triple in $\fg_\bR$ is $G_\bR$--conjugate to a Cayley triple \cite[Theorem 9.4.1]{MR1251060}.
\end{remark}

\begin{example} \label{eg:stdtri_2}
Let $\{ N^+ , Y , N\}$ be a standard triple.  Then $\tspan_\bR\{N^+ , Y , N\}$ is isomorphic to $\fsl_2\bR$.  The standard triple is a Cayley triple with respect to the Cartan decomposition $\fsl_2\bR = \fk_\bR \op \fk_\bR^\perp$ given by $\fk_\bR = \tspan_\bR\{ N^+ - N\}$ and $\fk^\perp_\bR = \tspan_\bR\{ Y , N^++N\}$.
\end{example}

A \emph{Djokovi{\'c}--Kostant--Sekiguchi triple} (DKS--triple) is any standard triple in $\fg_\bC$ of the form $\{ \overline\sE,\sZ,\sE\}$ with the property that $\sZ \in \fk_\bC$ and $\overline\sE ,\, \sE \in \fk^\perp_\bC$.  The \emph{Cayley transform} of a Cayley triple $\{ N^+ , Y , N \}$ is the DKS--triple
\begin{eqnarray} 
  \nonumber
  \overline\sE & = & \half \,( N^+ + N - \bi\,Y) \,,\\
  \label{E:DKStrp}
  \sZ  & = & \bi\,(N - N^+ ) \,,\\
  \nonumber
  \sE  & = & \half \, (N^+ + N + \bi\,Y) \,.
\end{eqnarray}
Note that 
\begin{equation}\label{E:ZvY}
  \left\{ \overline\sE , \sZ , \sE \right\} \ = \ 
  \tAd_\varrho \left\{ N^+ , Y , N \right\} \,,
\end{equation}
where the element $\varrho \in G_\bC$ is defined by
\begin{equation} \label{E:bd}
  \varrho \ = \ \texp\,\bi\tfrac{\pi}{4} \left( N^++N \right) 
  \quad 
  \Big( = \ \texp\,\bi\tfrac{\pi}{4} \left( \sE + \overline\sE \right) \,\Big)\,.
\end{equation}
The Djokovi{\'c}--Kostant--Sekiguchi correspondence \eqref{E:DKS} identifies the $\tAd(G_\bR)$--orbit of $N$ with the $\tAd(K_\bC)$--orbit of $\sE = \tAd_\varrho(N)$.  

\begin{example} \label{eg:stdtri_3}
Identify \eqref{E:stdtri_sl2} as a Cayley triple with respect to the Cartan decomposition of Example \ref{eg:stdtri_2}.  Then \eqref{E:stdtri_su11} is the Cayley transform of \eqref{E:stdtri_sl2}.
\end{example}

In summary, to distinguish the $\tAd(G_\bR)$--orbits in $\fg_\bR$ it suffices to distinguish the $\tAd(K_\bC)$--orbits in $\fk_\bC^\perp$.  Let $\sS_\fk = \{\gamma_1,\ldots,\gamma_s\} \subset \fh^*$ denote the simple roots of $\fk_\bC$ (\S\ref{S:cpt_rts}).  We may conjugate $\sZ$ by $\tAd(K_\bC)$ so that $\sZ \subset \fh$ and $\gamma_i(\sZ) \ge 0$.  The vector
\[
  \gamma(\sZ) \ := \ ( \gamma_1(\sZ) , \ldots, \gamma_s(\sZ) )
\]
it is an invariant of the nilpotent orbit so that 
\[
  \gamma(\cN) \ := \ \gamma(\sZ)
\]
is well--defined.  However, in the case that $\fg_\bR$ is of Hermitian symmetric type, it is not a complete invariant (two distinct orbits $\cN'\not=\cN$ may have $\gamma(\cN) = \gamma(\cN')$); we have lost information on the component of $\sZ$ lying in the center.  Recall the noncompact simple root $\a' \in \sS\backslash \sS_\fk$ (\S\ref{S:cpt_rts}).  The integer $\a'(\sZ)$ is also an invariant of the of nilpotent orbit, so that 
\[
  \a'(\cN) \ := \ \a'(\sZ)
\]
is also well--defined.  The pair $( \gamma(\sZ)\,;\,\a'(\sZ))$ is a complete invariant of the orbit, which we shall refer to as the  \emph{(compact) characteristic vector} of the orbit $\cN = \tAd(G_\bR)\cdot N$ (or the orbit $\tAd(K_\bC)\cdot \sE$).  (In the case that $\fg_\bR$ is not Hermitian symmetric, the simple roots $\sS_\fk$ span $\fh^*$ so that $\a'(\sZ)$ is determined by $\gamma(\cN)$.)  The following may be found in \cite[\S9.5]{\CoMc}.

\begin{theorem}\label{T:ccv}
The compact characteristic vector $( \gamma(\sZ)\,;\,\a'(\sZ))$ is a complete invariant of the orbit $\tAd(G_\bR)\cdot N \subset \tNilp(\fg_\bR)$.
\end{theorem}

\section{Hodge theory background} \label{S:htb}

\subsection{Hodge representations and Mumford--Tate domains} \label{S:Hrep}

Let $G_\bR$ be a non--compact, reductive, real algebraic group with maximal compact subgroup $K_\bR$ of equal rank
\[
  \trank\,\fg_\bC \ = \ \trank\,\fk_\bC \,.
\]
A \emph{(real) Hodge representation} (of weight $n$) of $G_\bR$ is defined in \cite{\GGK} and consists of:
\begin{i_list}
\item 
a finite dimensional vector space $V_\bR$ defined over $\bR$, a nondegenerate $(-1)^n$--symmetric bilinear form $Q : V_\bR \times V_\bR \to \bR$, and a homomorphism of real algebraic groups
\[
  \rho : G_\bR \to \tAut(V_\bR,Q) \,;
\]
\item
a nonconstant homomorphism of real algebraic groups
\[
  \varphi : S^1 \to G_\bR
\]
such that $\rho \circ\varphi$ defines a $Q$--polarized (pure, real) Hodge structure of weight $n$ on $V_\bR$.  The latter condition means that 
\begin{equation} \label{E:VH}
  V^{p,q} \ = \ 
  \{ v \in V_\bC \ | \ \rho\circ\varphi(z)v = z^{p-q} v \,\ \forall \ z\in S^1 \}
\end{equation}
defines a Hodge decomposition $V_\bC = \op_{p+q=n}V^{p,q}$ and $Q( \varphi(\bi) v , \bar v) > 0$ for all $0\not= v \in V_\bC$.
\end{i_list}
We always assume that the induced representation $\td \rho : \fg_\bR \to \tEnd(V_\bR,Q)$ is faithful, and will often refer to $\varphi$ as a ``circle''.  The Hodge representation is properly denoted $(V_\bR,Q,\rho,\varphi)$, but will sometimes be indicated by $V_\bR$ alone.  Additionally, we will often suppress $\rho$, and view the circle $\varphi$ as acting directly on $V_\bC$; it is from this perspective that we will refer to $\varphi$ as the Hodge structure on $V_\bR$, and generally write $N \in \tEnd(V_\bR)$ in place of $\td\rho(N) \in \tEnd(V_\bR)$.

Associated to the Hodge representation is the \emph{Hodge flag} 
\begin{equation}\label{E:Fp}
  F^p \ = \ \bigoplus_{r\ge p} V^{r,\sb} \,.
\end{equation}
The \emph{Hodge numbers} are the dimensions $\mathbf{f} = (f^p = \tdim_\bC F^p)$.  The Hodge flag is a point in the $Q$--isotropic flag variety $\tFlag^Q_\mathbf{f}(V_\bC)$.  The $G_\bR$--orbit $D = G_\bR\cdot F^\sb$ is the \emph{Mumford--Tate domain} of the Hodge representation; it is an open subset of the \emph{compact dual} $\check D = G_\bC \cdot F^\sb$.  When $G_\bR = \tAut(V_\bR,Q)$, $D$ is a \emph{period domain}.

As homogeneous manifolds
\[
  \check D \ = \ G_\bC/P \tand 
  D \ = \ G_\bR/K_\bR^0
\]
where $P = \tStab_{G_\bC}{F^\sb}$ is a parabolic subgroup of $G_\bC$ and $K_\bR^0 = G_\bR \cap P$ is compact.  We say that \emph{the Hodge representation $(V_\bR,Q,\rho,\varphi)$ realizes the homogeneous manifold $G_\bR/K_\bR^0$ as a Mumford--Tate domain}.  Such a realization is not unique.  For example, given $(V_\bR,Q,\rho,\varphi)$, there is an induced bilinear form $Q_\fg$ on $\fg_\bR \subset \tEnd(V_\bR,Q)$ that is nondegenerate and symmetric, and $(\fg_\bR,Q_\fg,\tAd,\varphi)$ is a weight zero Hodge representation that also realizes $G_\bR/K_\bR^0$ as a Mumford--Tate domain.  (See \S\ref{S:HS-CD} for further discussion of this induced representation.)   These two realizations are isomorphic as Mumford--Tate domains.  A key consequence of this is that
\begin{equation} \label{E:useg}
\begin{array}{l}
\emph{For the purposes of studying $G_\bR/K_\bR^0$ as a Mumford--Tate domain $D$,}\\ \emph{we may work with either the Hodge representation $(V_\bR,Q,\rho,\varphi)$ or} \\ \emph{the induced Hodge representation $(\fg_\bR,Q_\fg,\tAd,\varphi)$.}
\end{array}
\end{equation}  
What we have in mind is the case that $V_\bR$ carries an effective Hodge structure of weight $n \ge0$; for example, $V_\bR = H^n(X,\bR)$, where $X$ is a smooth projective variety.  It is helpful to work with the induced, weight zero, Hodge representation on $\fg_\bR$ because the latter is closely related to the geometry and representation theory associated with the flag domain $D \subset \check D$.  

\begin{remark}[A notational liberty] \label{R:not}
The Hodge flag $F^\sb$ and the circle $\varphi$ are equivalent: given one, the second is determined, \cf~\cite{\GGK}.  So we may identify $\varphi$ with the point $F^\sb \in D$.  This will be especially convenient when we wish to down play our choice of Hodge representation $(V_\bR,\rho)$ that gives $D \simeq G_\bR/K_\bR^0$ the structure of a Mumford--Tate domain.
\end{remark}

\subsubsection{Hodge structures and Cartan decompositions} \label{S:HS-CD}

Given a Hodge representation $(V_\bR,Q,\rho,\varphi)$ the induced Hodge structure on $\fg_\bC$ is
\begin{subequations} \label{SE:indHD}
\begin{equation} 
  \fg_\bC \ = \ \bigoplus \fg^{p,-p} \,,
\end{equation}
where
\begin{equation}
\renewcommand{\arraystretch}{1.3}
\begin{array}{rcl}
  \fg^{p,-p} & = & 
  \{ \x \in \fg_\bC \ | \ \x(V^{r,s}) \subset V^{r+p,s-p}
  \ \forall \ r,s \} \\
  & = & 
  \{ \x \in \fg_\bC \ | \ \tAd_{\varphi(z)}\x = z^{2p}\xi \ \forall \ 
  z \in S^1 \} \,.
\end{array}
\end{equation}
\end{subequations}
The decomposition is a grading of the Lie algebra in the sense that 
\[
  [ \fg^{p,-p}\,,\, \fg^{q,-q}] \ \subset \ [\fg^{p+q,-p-q}] \,.
\]
This implies that 
\begin{subequations}\label{SE:kCkperpC}
\begin{equation}
  \fk_\bC \ := \ \bigoplus_{p \ \teven} \fg^{p,-p} 
\end{equation}
is a subalgebra of $\fg_\bC$, and 
\begin{equation}
  \fk_\bC^\perp \ := \ \bigoplus_{p \ \todd} \fg^{p,-p} 
\end{equation}
\end{subequations}
is a $\fk_\bC$--submodule.  Moreover, $\overline{\fg^{p,-p}} = \fg^{-p,p}$ implies that both $\fk_\bC$ and $\fk_\bC^\perp$ are defined over $\bR$, so that 
\begin{equation} \label{E:CD}
  \fg_\bR \ = \ \fk_\bR \, \op \, \fk_\bR^\perp
\end{equation}
where $\fk_\bR = \fg_\bR \cap \fk_\bC$ and $\fk_\bR^\perp = \fg_\bR \cap \fk_\bC^\perp$.  The following is well-known; see, for example, \cite{\CKextn, \GGK}.

\begin{lemma}\label{L:CD}
The Weyl operator $\varphi(\bi)$ is a Cartan involution with Cartan decomposition \eqref{E:CD}.
\end{lemma}

\begin{remark} 
The projection $D = G_\bR/K_\bR^0 \to G_\bR/K_\bR$ may be viewed as the map taking the Hodge decomposition \eqref{SE:indHD} to the Cartan decomposition \eqref{E:CD}.
\end{remark}

\begin{proof}
In the case that $\fg_\bC$ is simple, $Q_\fg$ is necessarily a negative multiple of the Killing form.  This is because a simple complex Lie algebra admits a unique $\tAd(G_\bC)$--invariant symmetric bilinear form, the Killing form, up to scale.  So the induced polarization is necessarily a constant multiple of the Killing form.  The facts that: $Q_\fg$ is positive definite on the subalgebra $\fk_\bR$ and negative definite $\fk_\bR^\perp$ imply that \eqref{E:CD} is a Cartan decomposition of $\fg_\bR$ and $Q_\fg$ is a negative multiple of the Killing form.

More generally, as a reductive algebra $\fg_\bC$ decomposes as the direct sum $\fz \op \fg_\bC^\tss$ of its center and the semisimple factor $\fg_\bC^\tss = [\fg_\bC , \fg_\bC]$.  Note that $\fz \subset \fg^{0,0}$, so that the polarization $Q_\fg$ is positive definite on the real form $\fz \cap \fg_\bR \subset \fk_\bR$.  As above, the restriction of $Q_\fg$ to any simple factor of $\fg_\bC^\tss$ will be a negative multiple of the Killing form (the multiple may vary from one simple factor to the next) and \eqref{E:CD} is a Cartan decomposition.  
\end{proof}

\begin{remark}[A reasonable assumption on $Q_\fg$]
From the argument establishing Lemma \ref{L:CD} we see that there is no essential loss of generality in assuming that the induced polarization $Q_\fg$ on $\fg_\bR$ is minus the Killing form.
\end{remark}  

Given a maximal compact Lie subgroup $K_\bR \subset G_\bR$, let $\theta : \fg_\bR \to \fg_\bR$ be the corresponding Cartan involution.  A point $\varphi \in \check D$ is a \emph{$K$--Matsuki point} if the Lie algebra $\fp$ of the stabilizer $\tStab_{G_\bC}(\varphi)$ contains a conjugation and $\theta$--stable Cartan subalgebra $\fh$ of $\fg_\bC$.  As discussed in \cite[\S4.3]{MR2188135}, 
\begin{equation}\label{E:mat}
  \hbox{\emph{any two $K$--Matsuki points in $D$ are $K_\bR$--conjugate.}}
\end{equation}

\noindent From Lemma \ref{L:CD} we obtain

\begin{corollary} \label{C:matpt}
The circle $\varphi \in D$ is a Matsuki point with respect to the maximal compact subgroup $K_\bR$ determined by \eqref{E:CD}.
\end{corollary}

\subsubsection{Hodge structures and grading elements} \label{S:HS-GE}

As illustrated in \cite[\S2.3]{MR3217458}, grading elements (\S\ref{S:GE}) are essentially infinitesimal Hodge structures.  Briefly, given a circle $\varphi : S^1 \to G_\bR$, we may assume that the image $\tim\,\varphi$ is contained in a compact maximal torus $T \subset G_\bR$ and that the complexification $\fh = \ft \ot_\bR \bC$ of the Lie algebra $\ft$ of $T$ is a Cartan subalgebra of $\fg_\bC$.  Then the (rescaled) derivative 
\begin{equation} \label{E:Ephi}
  \ttE_\varphi \ := \ \frac{1}{4\pi\bi} \varphi'(1)
\end{equation}
is a grading element.  The relationship between the $\ttE_\varphi$--eigenspace decomposition \eqref{E:VE} and the Hodge decomposition \eqref{E:VH} is 
\[
  V^{(p-q)/2} \ = \ V^{p,q}\,.
\]
In the case that $V_\bC = \fg_\bC$, we have 
\begin{equation} \label{E:p-pp}
  \fg^p \ = \ \fg^{p,-p} \,.
\end{equation}
As a consequence, the Lie algebra $\fp_\varphi$ of the stabilizer $P_\varphi = \tStab_{G_\bC}(\varphi)$ is the parabolic \eqref{E:p} associated with the grading element $\ttE_\varphi$.  

Observe that the holomorphic tangent space $T_\varphi D = \fg_\bC/\fp_\varphi$ is naturally identified with $\op_{p>0} \fg^{-p,p}$.  The \emph{horizontal sub-bundle} $T^hD \subset TD$ is the $G_\bR$--homogeneous sub-bundle with fibre $T^h_\varphi D\simeq \fg^{-1,1}$.  A holomorphic map $f : M \to D$ is \emph{horizontal} if $f_*TM \subset T^hD$.  

The horizontal sub-bundle is bracket--generating if and only if $\ttE_\varphi$ is the grading element $\ttE_{\fp_\varphi}$ associated with $\fp_\varphi$ by \eqref{E:ttE}.  One may always reduce to the case that the infinitesimal period relation is bracket--generating, \cf\cite[\S3.3]{MR3217458}, and so we will 
\begin{equation} \label{E:bg}
\begin{array}{l}
\hbox{\emph{Assume that the horizontal sub-bundle is}}\\
\hbox{\emph{bracket--generating; equivalently, $\ttE_\varphi = \ttE_{\fp_\varphi}$.}}
\end{array}
\end{equation}
This assumption has the very significant consequence that 
\begin{equation} \label{E:G_R}
\hbox{\emph{The compact dual $\check D = G_\bC/P$ determines the real form $G_\bR$.}}
\end{equation}
This may be seen as follows.  The choice of compact dual is equivalent to a choice of conjugacy class $\cP$ of parabolic subgroups $P \subset G_\bC$.  Modulo the action of $G_\bC$, the conjugacy class determines the grading element $\ttE$ by \eqref{E:ttE}.  It then follows from \eqref{SE:kCkperpC} and \eqref{E:p-pp} that the $\ttE$--eigenspace decomposition \eqref{SE:grading} of $\fg_\bC$ determines the complexified Cartan decomposition $\fg_\bC = \fk_\bC \op \fk^\perp_\bC$.  If $\fg_\bC$ is simple, then $\fk_\bC$ uniquely determines $\fg_\bR$, \cf~\S\ref{S:excp_gR}.  More generally, if $\fg_\bC$ is semisimple then each simple ideal $\fg'_\bC \subset \fg_\bC$ is a sub-Hodge structure; again the grading element/infinitesimal Hodge structure determines a complexified Cartan decomposition, and the corresponding $\fk'_\bC$ determines $\fg_\bR'$.  Finally, in the general case that $\fg_\bC = \fz_\bC \op \fg_\bC^\tss$ is reductive, the fact that the center $\fz_\bC$ is contained in $\fg^{0,0} \subset \fk_\bC$ forces $Z_\bR$ to be a compact torus $S^1 \times \cdots \times S^1$.

\subsubsection{Levi subalgebras and sub--Hodge structures} \label{S:subHS}

A \emph{(real) sub--Hodge structure} of a Hodge representation $(V_\bR, Q, \rho, \varphi)$ is given by a real subspace $U_\bR \subset V_\bR$ that is preserved under the action of $\varphi(z)$ for all $z \in S^1$.  In this case, we will say that the subspace $U_\bR$ is \emph{$\varphi$--stable}.  The following lemma formalizes an observation made in the proof of \cite[Lemma V.23]{GGR}.

\begin{lemma} \label{L:levi1}
Consider a Hodge representation $(\fg_\bR,Q_\fg,\tAd,\varphi)$ of $G_\bR$ on the Lie algebra.  A Levi subalgebra $\fl_\bR \subset \fg_\bR$ carries a sub--Hodge structure if and only if the image $\varphi(S^1)$ lies in the (connected) Lie subgroup $L_\bR \subset G_\bR$ with Lie algebra $\fl_\bR$; equivalently, $\ttE_\varphi \in \fl_\bC$.
\end{lemma}

\begin{remark} \label{R:levi1}
A priori the condition that $\varphi(S^1) \subset L_\bR$ is stronger than the condition that $\fl_\bR$ carries a sub--Hodge structure: the former implies that $(\fg_\bR,Q_\fg,\left.\rho\right|_{L_\bR},\varphi)$ is a Hodge--representation of $L_\bR$.
\end{remark}

\begin{proof}
$(\Longleftarrow)$ If the image of $\varphi$ lies in $L_\bR$, then it is clear that $\varphi(z)$ preserves $\fl_\bR$ for all $z \in S^1$.  \smallskip

$(\Longrightarrow)$  Recall the (rescaled) derivative $\ttE_\varphi = \varphi'(1)/4\pi\bi$ of \eqref{E:Ephi}.  To show that the image of $\varphi$ lies in $L_\bR$, it suffices to show that $\ttE_\varphi \in \fl_\bC$.  Let $\fg_\bC = \op \fg^{p,q}$ be the Hodge decomposition.  Then $\fl_\bC = \op \fl^{p,q}$, where $\fl^{p,q} = \fl_\bC \cap \fg^{p,q}$.  As discussed in \S\ref{S:HS-GE}, these Hodge decompositions may be viewed as $\ttE_\varphi$--eigenspace decompositions for the grading element $\ttE_\varphi \in \fg_\bC$.  In particular, 
\begin{equation} \label{E:sl}
  \fl_\bC \ = \ \op\,\fl^a
\end{equation}
where $\fl^a = \fl_\bC \cap \fg^a$, and $\fg_\bC = \op\,\fg^a$ is given by \eqref{SE:grading}.  Moreover, \eqref{E:gr} implies that \eqref{E:sl} is a graded decomposition; that is $[\fl^a , \fl^b] \subset \fl^{a+b}$.

As a reductive Lie algebra $\fl_\bC = \fz_\bC \op \fl^\tss_\bC$, where $\fl^\tss_\bC = [\fl_\bC,\fl_\bC]$ is the semisimple factor, and $\fz_\bC \subset \fl^0$ is the center of $\fl_\bC$.  The graded decomposition of $\fl_\bC$ induces a graded decomposition 
\begin{equation}\label{E:s_ss}
  \fl^\tss_\bC \ = \ \op\,\fl_a^\tss
\end{equation}
by $\fl_a^\tss = \fl^\tss_\bC \cap \fl^a$.  There exists a grading element $\ttF \in \fl^\tss_\bC$ with the property that \eqref{E:s_ss} is the $\ttF$--eigenspace decomposition of $\fl^\tss_\bC$ \cite[Proposition 3.1.2]{\CS}.  Observe that $\ttE_\varphi - \ttF \in C_{\fg_\bC}(\fl_\bC)$ lies in the centralizer of $\fl_\bC$.  Because $\fl_\bC$ is a Levi subalgebra, this centralizer is equal to the center $\fz_\bC$.  Therefore, $\ttE_\varphi - \ttF \in \fl_\bC$.  Since $\ttF \in \fl_\bC$, this implies $\ttE_\varphi \in \fl_\bC$.
\end{proof}



\subsection{Polarized mixed Hodge structures} \label{S:NC}

Let $(V_\bR,Q)$ be a Hodge representation of $G_\bR$ and let $D \subset \check D$ be the corresponding Mumford--Tate domain.  A ($m$--variable) \emph{nilpotent orbit} on $D$ consists of a pair $(F^\sb; N_1,\ldots,N_m )$ such that $F^\sb \in \check D$, the $N_i \in \fg_\bR$ commute and $N_iF^p \subset F^{p-1}$, and the holomorphic map $\psi : \bC^m \to \check D$ defined by
\begin{equation}\label{E:htno}
  \psi(z^1,\ldots,z^m) \ = \ \texp( z^i N_i ) F^\sb 
\end{equation}
has the property that $\psi(z) \in D$ for $\tIm(z^i) \gg 0$.  The associated (open) \emph{nilpotent cone} is
\begin{equation} \label{E:s}
  \tNC\ = \ \{ t^i N_i \ | \ t^i > 0 \} \,.
\end{equation}

A \emph{polarized mixed Hodge structure} on $D$ is given by a pair $(F^\sb,N)$ such that $F^\sb \in \check D$, $N \in \fg_\bR$ and $N(F^p) \subset F^{p-1}$, $(F^\sb,W_\sb(N,V_\bR))$ is a mixed Hodge structure, and the Hodge structure on 
\[
  \tGr_k(W_\sb(N,V_\bR))_\tprim \ := \
  \tker\{ N^k : \tGr_k(W_\sb(N,V_\bR)) \to \tGr_{-k}(W_\sb(N,V_\bR))\}
\] 
is polarized by $Q(\cdot , N^k \cdot)$, for all $k\ge0$.  The notions of nilpotent orbit and polarized mixed Hodge structure are closely related.  The following well-known results are due to Cattani, Kaplan and Schmid \cite{\CKpmhs, MR1042802, \CKSdeg, \CKScoh, \Schmid}.

\begin{theorem}[Cattani, Kaplan, Schmid] \label{T:cks}
Let $D \subset \check D$ be a Mumford--Tate domain (and compact dual) for a Hodge representation $V_\bR$ of $G_\bR$.
\begin{a_list_emph}
\item
A pair $(F^\sb;N)$ forms a one--variable nilpotent orbit if and only if it forms a polarized mixed Hodge structure.  
\item
The weight filtration $W_\sb(N,V_\bR)$ does not depend on the choice of $N \in \tNC$. \emph{Let $W_\sb(\tNC,V_\bR)$ denote this common weight filtration.} 
\item
Fix $F^\sb\in\check D$ and commuting nilpotent elements $\{ N_1 , \ldots , N_m\}\subset \fg_\bR$ with the properties that: \emph{(i)} $N_i F^p \subset F^{p-1}$ for every $i$; and \emph{(ii)} the filtration $W_\sb(N,V_\bR)$ does not depend on the choice of $N \in \tNC$, where the latter is given by \eqref{E:s}.  Then $(F^\sb;N)$ is a polarized mixed Hodge structure for \emph{some} $N \in \tNC$, if and only if $(F^\sb;N_1,\ldots,N_m)$ is an $m$--variable nilpotent orbit.
\end{a_list_emph}
\end{theorem}

\noindent In a mild abuse of nomenclature, given a nilpotent orbit $(F^\sb;N_1,\ldots,N_m)$ we will sometimes refer to $(F^\sb,W_\sb(\tNC,V_\bR))$ as a polarized mixed Hodge structure (especially when we wish to emphasize the weight filtration $W_\sb(\tNC,V_\bR)$ over the nilpotents $N\in\tNC$).

The \emph{Deligne splitting} \cite{\CKSdeg, \DeligneII} 
\begin{subequations} \label{SE:deligne}
\begin{equation}
  V_\bC \ = \ \bigoplus I^{p,q}
\end{equation}
of a mixed Hodge structure $(F^\sb,W_\sb)$ on $V_\bR$ is given by
\begin{equation}
  I^{p,q} \ := \ F^p \,\cap\, W_{p+q} \, \cap \, 
  \Big( \overline{F^q} \,\cap\,W_{p+q} \,+\, 
         \sum_{j\ge1} \overline{F^{q-j}} \,\cap\, W_{p+q-j-1} \Big) \,.
\end{equation}
\end{subequations}
It is the unique bigrading of $V_\bC$ with the properties that 
\begin{equation} \label{E:FW}
  F^p \ = \ \bigoplus_{r \ge p} I^{r,\sb} \tand
  W_\ell \ = \ \bigoplus_{p+q \le \ell} I^{p,q} \,,
\end{equation}
and
\[
  \overline{I^{p,q}} \ = \ I^{q,p} \quad\hbox{mod} \quad
  \bigoplus_{r<q,s<p} I^{r,s} \,.
\]

Any mixed Hodge structure $(F^\sb , W_\sb)$ on $V$ induces a mixed Hodge structure $(F^\sb_\fg , W_\sb^\fg)$ on $\fg$ by 
\begin{eqnarray*}
  F^p_\fg & = & \{ \xi \in \fg_\bC \ | \ \xi(F^r) \subset F^{p+r} \ \forall \ r \} \\
  W^\fg_\ell & = & \{ \xi \in \fg_\bR \ | \ \xi(W_m) \subset W_{m+\ell} \ \forall \ m \} \,.
\end{eqnarray*}
The elements of $F^r_\fg \cap W^\fg_{2r}\cap\fg_\bR$ are the \emph{$(r,r)$--morphisms} of the mixed Hodge structure $(F^\sb , W_\sb)$.  Alternatively, if $\fg_\bC = \op I^{p,q}_\fg$ denotes the corresponding Deligne splitting
\[
  I^{p,q}_\fg \ = \ \{ \xi \in \fg_\bC \ | \ \x(I^{r,s}) \subset I^{p+r,q+s} \ \forall \ r,s \} \,,
\]
then the elements of $I^{r,r}_\fg\cap \fg_\bR$ are the $(r,r)$--morphisms.  

When $\overline{I^{p,q}} = I^{q,p}$ we say the mixed Hodge structure is \emph{$\bR$--split}.  When an $\bR$--split mixed Hodge structure $(F^\sb,W_\sb(\tNC,V_\bR))$ arises from a nilpotent orbit $(F^\sb;N_1,\ldots,N_m)$, we will say that the nilpotent orbit is $\bR$--split.  

\begin{remark} \label{R:Rsplit}
If $(F^\sb,N)$ is $\bR$--split, then so is the induced $(F^\sb_\fg,N)$.
\end{remark}

\noindent Observe that
\[
  L^{-1,-1}_\fg \ := \ \bigoplus_{p,q>0} I^{-p,-q}_\fg
\]
is a subalgebra of $\fg_\bC$ and is defined over $\bR$.    The following well-known results are due to Cattani, Deligne, Kaplan and Schmid \cite{\CKpmhs, \CKSdeg, \DeligneII}.

\begin{theorem}[Deligne, Cattani, Kaplan, Schmid] \label{T:dcks}
Let $D\subset \check D$ be a Mumford--Tate domain (and compact dual) for a weight $n$ Hodge representation of $G_\bR$ on $V_\bR$.
\begin{a_list_emph}
\item 
If $(F^\sb;N)$ is an $\bR$--split polarized mixed Hodge structure, then $\psi(z) = e^{zN}F^\sb \in D$ for all $\tIm(z) > 0$ and $\psi$ is a horizontal, $\tSL_2\bR$--equivariant embedding of the upper--half plane.
\item
Given a mixed Hodge structure $(F^\sb,W_\sb)$ on $V_\bR$, there exists a unique $\d \in L^{-1,-1}_{\fg,\bR}$ such that 
\[
  e^{-2\bi\d}\cdot F^p \ = \ \bigoplus_{s\ge p} I^{\sb,s} \,.
\]
The element $\d$ is real, commutes with all morphisms of $(F^\sb,W_\sb)$ and, given
\begin{equation}\label{E:tildeF}
  \tilde F^\sb \ := \ e^{-\bi\d} \cdot F^\sb, 
\end{equation}
$(\tilde F^\sb,W_\sb)$ is an $\bR$--split mixed Hodge structure.  \emph{(From $L^{-1,-1}_\fg \subset W_{-2}^\fg$ we see that $\d$ preserves the filtration $W_\sb$ and acts trivially on $\tGr_\ell(W_\sb)$.  It follows that both $F^\sb$ and $\tilde F^\sb$ determine the same filtrations on $\tGr_\ell(W_\sb)$.)}  Moreover, every morphism of $(F^\sb,W_\sb)$ commutes with $\d$, so that the morphisms of $(F^\sb,W_\sb)$ are precisely those of $(\tilde F^\sb,W_\sb)$ that commute with $\d$.
\item
In the case that $W_\sb = W_\sb(N,V_\bR)[-n]$, the two nilpotent orbits $\psi(z) = e^{zN} F^\sb$ and $\tilde\psi(z) = e^{zN}\tilde F^\sb$ agree to first order at $z=\infty$, and that limit flag is
\begin{equation} \label{E:Finfty}
  F^p_\infty \ := \ \lim_{\tIm(z) \to \infty} e^{z N} F^p \ = \ 
  \bigoplus_{s\le  n-p} I^{\sb,s} \,.
\end{equation}
\end{a_list_emph}
\end{theorem}

\subsection{Reduced limit period mapping} \label{S:rlpm}

Given commuting $N_1,\ldots,N_m \in \tNilp(\fg_\bR)$ defining a cone \eqref{E:s}, the \emph{boundary component} $B(\tNC)$ is the set of nilpotent orbits $(F^\sb;N_1,\ldots,N_m)$ modulo reparametrization.  That is, we say two elements $F^\sb_1$ and $F^\sb_2$ of 
\[
  \tilde B(\tNC) \ := \ \{ F^\sb \in \check D \ | \ (F^\sb;N_1,\ldots,N_m) 
  \hbox{ is a nilpotent orbit} \}
\]
are \emph{equivalent} if $F^\sb_1 = \exp(z^iN_i) F^\sb_2$ for some $z=(z^i) \in \bC^m$; then
\[
  B(\tNC) \ := \ \tilde B(\tNC)/\sim \,.
\]
In the case that $m=1$, we write $B(\tNC) = B(N)$ and $\tilde B(\tNC) = \tilde B(N)$.

The \emph{reduced limit period mapping} $\Phi_\infty : \tilde B(N) \to \tcl(D)$ defined by 
\begin{equation} \label{E:Phi_infty}
  \Phi_\infty(F^\sb,N) \ := \ \lim_{\tIm(z)\to\infty} e^{zN} \cdot F^\sb
\end{equation}
descends to a well--defined map on $B(N)$; see \cite[Appendix to Lecture 10]{\GGKtcu} and \cite[\S5]{KP2013} for details.\footnote{In \cite{KP2013}, $\Phi_\infty$ is called the \emph{\naive~ limit map}.}  More generally, as observed in \cite[Remark 5.6]{KP2013}, the reduced limit period mapping is well--defined on $B(\tNC)$; that is, \eqref{E:Phi_infty} does not depend on our choice of $N \in \tNC$.  This may be seen as follows.  First, by Theorem \ref{T:cks}(b), the weight filtration $W_\sb(\tNC,V_\bR)$ does not depend on our choice of $N \in \tNC$.  Let $(\tilde F^\sb,W_\sb(\tNC,V_\bR))$ be the $\bR$--split mixed Hodge structure given by Theorem \ref{T:dcks}(b), and let $V_\bC = \op \tilde I^{p,q}$ be the corresponding Deligne splitting \eqref{SE:deligne}.  Then Theorem \ref{T:dcks}(c) and \eqref{E:Phi_infty} assert that 
\[
  \Phi_\infty(F^\sb,N) \ = \ \Phi_\infty(\tilde F^\sb,N) \ = \ \tilde F^\sb_\infty
  \tand
  \tilde F^p_\infty \ = \ \bigoplus_{s \le n-p} \tilde I^{\sb,s}
\]
is independent of $N \in \tNC$.\footnote{See \cite{MR3133298, HayPearl} for more general convergence results.}

\section{Hodge--Tate degenerations} \label{S:HT}

The main results of this section are: (i) underlying every $\bR$--split polarized mixed Hodge structure $(F^\sb,N)$ is a Hodge--Tate polarized mixed Hodge structure $(\sF^\sb_\fl,N)$ on a Levi subalgebra $\fl_\bR \subset \fg_\bR$ (Theorem \ref{T:underHT}); and (ii) the classification of the Hodge--Tate degenerations (Theorem \ref{T:cHT}).  Corollary to these results we will: (a) see that the nilpotent cone $\tNC \subset \fg_\bR$ underlying a nilpotent orbit is contained in a $\tAd(L_\bR^Y)$--orbit, where $L_\bR^Y$ is a connected Lie subgroup of $G_\bR$ with reductive Lie algebra $\fl_\bR^Y \subset \fl_\bR$ (Corollary \ref{C:Ad-orb}); and (b) obtain the classification theorems of \S\ref{S:CTs}.

\subsection{Definition}  \label{S:dfnHT}

Let $(V_\bR,Q,\rho,\varphi)$ be a Hodge representation of $G_\bR$, and let $D \subset \check D = G_\bC/P$ be the associated Mumford--Tate domain and compact dual.  We say that $D$ \emph{admits a Hodge--Tate degeneration} if there exists a nilpotent orbit $(F^\sb;N_1,\ldots,N_m)$ with nilpotent cone $\tNC$ such that the Deligne splitting \eqref{SE:deligne} of $(F^\sb,W_\sb(\tNC,V_\bR))$ satisfies
\[
  I^{p,q} = 0 \quad\hbox{for all} \quad p\not=q \,.
\]  
In this case we say that the nilpotent orbit $(F^\sb;N_1,\ldots,N_m)$ is a \emph{Hodge--Tate degeneration}.  

We recall some properties of Hodge--Tate degenerations in

\begin{proposition}\label{P:HT}
Let $V_\bR$ admit the structure of a Hodge representation of $G_\bR$, and let $(F^\sb;N_1,\ldots,N_m)$ be a nilpotent orbit on the associated Mumford--Tate domain $D \subset \check D$.
\begin{a_list_emph}
\item 
If $(F^\sb;N_1,\ldots,N_m)$ is Hodge--Tate, then so is the induced nilpotent orbit $(F^\sb_\fg;N_1,\ldots,N_m)$ on $\fg_\bR$.
\item
Suppose that $G_\bR$ is semisimple.  Then $(F^\sb;N_1,\ldots,N_m)$ is Hodge--Tate if and only if $(F^\sb_\fg;N_1,\ldots,N_m)$ is Hodge--Tate.
\item 
If $(F^\sb_\fg;N_1,\ldots,N_m)$ is Hodge--Tate, then the nilpotent orbit $(F;N_1,\ldots,N_m)$ is a ``maximal'' degeneration of Hodge structure in the sense that $\Phi_\infty(F^\sb,N)$ lies in the unique closed $G_\bR$--orbit $\clO\subset \check D$, for any $N \in \tNC$.
\end{a_list_emph}
\end{proposition}

\begin{proof}
Part (b) is \cite[Proposition I.9]{GGR}, and Part (c) is \cite[Corollary 4.3]{KP2013} or \cite[Proposition I.15]{GGR}.  In general, $G_\bR$ is reductive and Proposition \ref{P:HT}(a) follows from the arguments establishing Proposition \ref{P:HT}(b).  
\end{proof}

\begin{remark}
If $G_\bR$ is not semisimple, then the converse to Proposition \ref{P:HT}(a) need not hold: it is possible for a non--Hodge--Tate $(F^\sb,N)$ to induce a Hodge--Tate $(F^\sb_\fg,N)$.  Indeed, this is precisely the case in Theorem \ref{T:underHT}, where the nilpotent orbit $(F^\sb;N_1,\ldots,N_m)$ on the Hodge representation $(V_\bR,Q,\rho,\varphi)$ of the reductive $L_\bR$ will in general fail to be Hodge--Tate, while the induced $(\sF^\sb_\fl;N_1,\ldots,N_m)$ is always Hodge--Tate, cf. Remark \ref{R:underHT1}.
\end{remark}

While the Hodge--Tate degenerations are ``maximal'' in the sense of Proposition \ref{P:HT}(b), the associated representation theory is relatively simple as we will see in the classification of Theorem \ref{T:cHT}.

\subsection{The underlying Hodge--Tate degeneration} \label{S:underHT}

In a suitably interpreted sense all degenerations are induced from a degeneration of Hodge--Tate type.\footnote{Some care must be taken with this statement, as it is not necessarily the case that the underlying degeneration arises algebro--geometrically: this is a statement about the orbit structure and representation theory associated with the $\tSL(2)$--orbit approximating an arbitrary degeneration, which may or may not arise algebro--geometrically.}  The results of this section for $\tdim_\bR\tNC= 1$ first appeared in \cite{GGR}.  Let 
\[
  \sH^m \ := \ \{ z = (z^i) \in \bC^m \ | \ \tIm(z^i) > 0 \} \,.
\]

\begin{theorem} \label{T:underHT}
Let $(V_\bR,Q,\rho,\varphi)$ be a Hodge representation of a semisimple Lie group $G_\bR$, and let $D$ be the associated Mumford--Tate domain.  Suppose that $(F^\sb;N_1,\ldots,N_m)$ is a $\bR$--split nilpotent orbit.  
\begin{a_list_emph}
\item
Let $\fg_\bC = \op I^{p,q}_\fg$ be the associated Deligne splitting, \cf\eqref{E:useg} and \eqref{SE:deligne}, and set
\begin{equation} \label{E:dfnL}
  \fl_\bC \ := \ \bigoplus_p I^{p,p}_\fg \,.
\end{equation}
Then $\fl_\bC$ is a Levi subalgebra of $\fg_\bC$ defined over $\bR$ with real form $\fl_\bR = \fl_\bC \cap \fg_\bR$ and $N_i \in \fl_\bR$.  Let $L_\bR \subset G_\bR$ be the connected Lie subgroup with Levi algebra $\fl_\bR$.
\item 
Given $z \in \sH^m$, let $\varphi_z : S^1 \to G_\bR$ denote the Hodge structure on $V_\bR$ parameterized by $\exp(z^i N_i)\cdot F^\sb \in D$.  Then the circle $\varphi_z$ is contained in $L_\bR$ for all $z \in \sH^m$; that is, $\tim\,\varphi_z \subset L_\bR$.  Equivalently, $(V_\bR,Q,\left.\rho\right|_{L_\bR},\varphi_z)$ is a Hodge representation of $L_\bR$; let $\sD$ denote the associated Mumford--Tate domain.
\item
The induced nilpotent orbit $(\sF^\sb_\fl;{N_1},\ldots,{N_m})$ on $\sD$ is a Hodge--Tate degeneration.
\end{a_list_emph}
\end{theorem}

\begin{remark}  \label{R:underHT}
An immediate and important consequence of Theorem \ref{T:underHT}(b) is that any nilpotent orbit on $\sD$ induces a nilpotent orbit on $D$; so we may think of the nilpotent orbit $(\sF^\sb_\fl;{N_1},\ldots,{N_m})$ as ``the Hodge--Tate degeneration underlying the nilpotent orbit $(F^\sb;N_1,\ldots,N_m)$.''  From this perspective, Theorem \ref{T:underHT} asserts that the essential structure/relationship is between the $\{N_1,\ldots,N_m\}$ and the Levi subalgebra $\fl$; the remaining structure on $\fg = \fl \op \fl^\perp$,\footnote{This $\fl$--module decomposition of $\fg$ exists because $\fl$ is reductive.} that is the Hodge structure on $\fl^\perp$, is induced from the $\fl$--module structure on $\fl^\perp$.\footnote{This sort of idea goes back to Bala and Carter's classification \cite{MR0417306, MR0417307} of nilpotent orbits $\cN \subset \fg_\bC$, where the idea is to look at minimal Levi subalgebras $\fl$ containing a fixed $N \in \cN$, and to classify the pairs $(N,\fl)$.  (In fact, the idea goes back farther to Dynkin \cite{MR0047629_trans}, who looked at minimal reductive subalgebras containing $N$, but this approach does not seem to work as well.)}
\end{remark}

\begin{remark} \label{R:underHT1}
Each 
\[
  \sV_\ell \ = \ \bigoplus_{p-q=\ell} I^{p,q}_\fg
\]
is a $\fl_\bC$--module, and $\sV_\ell + \sV_{-\ell}$ naturally has the structure of a Hodge representation of $L_\bR$.  In particular, $V = \op_{\ell\ge0} \sV_\ell$ is a coarse branching of $V$ as an $L_\bR$--Hodge representation.  (``Coarse'' because the $\sV_\ell$ need not be irreducible.) 
\end{remark}

\begin{proof}
The fact that the nilpotent orbit is $\bR$--split implies $\fl_\bC$ is a conjugation--stable subalgebra of $\fg_\bC$ and 
\[
  N_i \ \in \ I^{-1,-1}_{\fg,\bR} \ \subset \ \fl_\bR \,.
\]
As the zero eigenspace for the grading element $\ttE - \overline \ttE$, the subalgebra $\fl_\bC$ is necessarily a Levi subalgebra.  This establishes Theorem \ref{T:underHT}(a).  

Let $\cC$ be the nilpotent cone \eqref{E:s} underlying the nilpotent orbit.  Observe that the polarized mixed Hodge structure $(F^\sb_\fg,W_\sb(\tNC,\fg_\bR))$ on $\fg_\bR$ induces a polarized mixed Hodge sub-structure $(\sF^\sb_\fl , \sW_\sb(\tNC,\fl_\bR))$ on $\fl_\bR$ by 
\begin{equation}\label{E:sHT2}
  \sF^p_\fl \ := \ F^p \,\cap\, \fl_\bC 
  \ = \ \bigoplus_{q\ge p} I^{q,q}_\fg \tand
  \sW_\ell(\tNC,\fl_\bR) \ := \ W_\ell(\tNC,\fg_\bR) \,\cap\, \fl_\bR 
  \ = \ \bigoplus_{q \le \ell} I^{q,q}_{\fg,\bR} \,.
\end{equation}
Theorem \ref{T:cks}(c) implies that the Hodge flag $\exp(z^iN_i) \cdot \sF^\sb_\fl$ defines a Hodge structure on $\fl_\bR$; equivalently, $\fl_\bR$ is a sub--Hodge structure of $(\fg_\bR,\varphi_z)$.  Theorem \ref{T:underHT}(b) now follows from Lemma \ref{L:levi1}.

Finally, \eqref{E:dfnL} and \eqref{E:sHT2} yield Theorem \ref{T:underHT}(c).
\end{proof}

\begin{remark}[Mumford--Tate domain for the Hodge structures $\left.\varphi_z\right|_{\fl_\bR}$] \label{R:HSonS}
The Mumford--Tate domain $\sD$ for the Hodge structures $\varphi_z$ on $\fl_\bR$ may be viewed as a subset of $D$, the Mumford--Tate domain for the Hodge structure $\varphi$ on $\fg_\bR$ (or $V_\bR$).  Let $L_\bC \subset G_\bC$ be the connected Lie subgroup with Lie algebra $\fl_\bC$, and set $\check \sD = L_\bC \cdot F^\sb$.  Then $\sD \simeq \check \sD \,\cap\, D$. 
\end{remark}


\begin{corollary} \label{C:Ad-orb}
Given an $\bR$--split nilpotent orbit \eqref{E:htno} on a Mumford--Tate domain $D = G_\bR/K_\bR^0$ with nilpotent cone $\cC$ as in \eqref{E:s}, let $\fl_\bR$ be the Levi subalgebra \eqref{E:dfnL}.  Let $Y \in \fl_\bR$ be the grading element defined by 
\[
  \left. Y \right|_{I^{p,p}_\fg} \ = \ 2p \,,
\]
and let $L_\bR^Y$ denote the connected subgroup of $L_\bR$ stabilizing $Y$ under the adjoint action.  Then the Lie algebra $\fl_\bR^Y = \{ \xi \in \fl_\bR \ | \ [\xi,Y]=0\}$ is Levi and $\cC \subset \tNilp(\fl_\bR)$ is contained in an $\tAd(L_\bR^Y)$--orbit.
\end{corollary}

\begin{proof}[Proof of Corollary \ref{C:Ad-orb}]
Recall that $Y$ is a grading element, \cf\eqref{E:Y=GE}; it then follows from the definition (\S\ref{S:GE}) that 
\[
  \fl_\bR^Y \ = \ I^{0,0}_{\fg,\bR}
\]
is Levi.  

Cattani and Kaplan \cite[(3.3)]{MR664326} proved that the Jacobson--Morosov filtration $W_\sb(N')$ is independent of our choice of $N' \in \cC$; we denote this weight filtration by $W_\sb(\cC)$.  Let 
\[
  \cW_\cC \ = \ \{ N' \in I^{-1,-1}_{\fg,\bR} \ | \ W(N') = W(\cC) \}\,.
\]
Of course, $\cC \subset \cW_\cC$.  It suffices to show that $\cW_\cC$ is a disjoint union of open $L^Y_\bR$--orbits in $I^{-1,-1}_{\fg,\bR}$; for it then follows from the connectedness of $\cC$ that the cone is contained in an $\tAd(L^Y_\bR)$--orbit.

First observe that \S\ref{S:JMf}(iii) implies that for each $N' \in \cW_\cC$, there exists a unique $N'+_ \in I^{1,1}_{\fg,\bR}$ such that $\{ N'_+ , Y , N' \}$ is a standard triple.  Define
\[
  P^\ell_{N'} \ := \ \{ \xi \in \fl_\bR \ | \ [Y,\xi] = \ell\,\xi \,,\ 
  [N_+' , \xi] = 0 \}
  \quad \hbox{for all} \quad \ell \ge 0 \,.
\]  
This is the vector space of ``highest weight vectors" in the ``isotypic component of weight $\ell$" for the action of $\fsl_\bR2 = \tspan_\bR\{ N'_+ , Y , N'\}$ on $\fl_\bR$.  It is a basic result of $\fsl(2)$--representation theory that 
\[
  \fl_\bR \ = \ \bigoplus_{\mystack{\ell\ge0}{0 \le a \le \ell}} 
  (N')^a P^\ell{N'} \,.
\]
From this, and $(N')^{\ell+1}(P^\ell_{N'}) = 0$, we may deduce that $\tAd(L_\bR^Y)$--orbit of $N'$ is open in $I^{-1,-1}_{\fg,\bR}$.
\end{proof}

\subsection{Classification of Hodge--Tate degenerations} \label{S:cHT}

In \cite[Lemma V.7 and Theorem V.15]{GGR} it is shown that a \emph{period domain} parameterizing weight $n$ polarized Hodge structures admits a Hodge--Tate degeneration if and only if the Hodge numbers satisfy 
\begin{equation}\label{E:pdht}
  h^{n,0} \,\le\, h^{n-1,1} \,\le \cdots \le \, h^{n-m,m} \,,
\end{equation}
with $m$ defined by $n \in \{ 2m,2m+1\}$.  In the more general setting of Mumford--Tate domains, \eqref{E:pdht} is a necessary, but not sufficient, condition for the existence of a Hodge--Tate degeneration \cite[Lemma V.7 and Remark V.16]{GGR}.  Here we extend the classification to arbitrary Mumford--Tate domains $D$ with the property that the IPR is bracket--generating (we can always reduce to this case \cite[\S3.3]{MR3217458}).

\begin{theorem} \label{T:cHT}
Suppose that $V_\bR$ is a Hodge representation of a real semisimple algebraic group $G_\bR$.  Let $D \subset G_\bC/P$ be the associated Mumford--Tate domain and compact dual, and assume that the infinitesimal period relation is bracket--generating.  
Then $D$ admits a Hodge--Tate degeneration $(F^\sb,N)$ if and only if there exists a standard triple $\{ N^+ , Y , N\} \subset \fg_\bR$ such that the following two conditions hold:
\begin{a_list_emph}
\item
The neutral element $Y$ is even, and $\fp=W_0(N^+,\fg_\bC)$.  In this case, $\half Y$ is the grading element \eqref{E:ttE} associated with $\fp$ and $F^p_\fg=W_{-2p}(N^+,\fg_\bC)$.
\item 
The compact characteristic vector \emph{(\S\ref{S:rtno_R})} of the nilpotent orbit $\cN = \tAd(G_\bR)\cdot N$ satisfies the following conditions: $\gamma_i(\cN) \equiv 0$ mod $4$, for all $i$; and for the noncompact simple root, $\a'(\cN)$ is even and $\a'(\cN)/2$ is odd.
\end{a_list_emph}
If it exists, then the orbit $\cN$ is unique.  That is, given a second Hodge--Tate nilpotent orbit $(\tilde F^\sb,\tilde N)$, it is the case that $\tilde N \in \cN$.
\end{theorem}

\noindent 
The necessity of Theorem \ref{T:cHT}(a) was observed in \cite{GGR}.  It implies that the Lie algebra $\fp = F^0_\fg$ of the stabilizer $P=\tStab_{G_\bC}(F^\sb)$ is an even Jacobson--Morosov parabolic.  As illustrated by the examples at the end of this section this constrains the (conjugacy classes of the) parabolics $P$, and therefore the compact duals, that may arise.



\begin{proof}[Proof of Theorem \ref{T:cHT}]

$(\Longrightarrow)$  Suppose that there exists a Hodge--Tate nilpotent orbit $(F^\sb,N)$.  Then the induced nilpotent orbit $(F^\sb_\fg,N)$ is also Hodge--Tate (Proposition \ref{P:HT}).  Thus the Lie algebra of the parabolic subgroup $P \subset G_\bC$ stabilizing the Hodge flag $F^\sb$ is 
\begin{equation} \label{E:1}
  \fp \ = \ F^0_\fg \ = \ \bigoplus_{p\ge0} I^{p,\sb}_\fg
  \ = \ \bigoplus_{p\ge0} I^{p,p}_\fg 
  \ = \ \bigoplus_{p+q \ge 0} I^{p,q}_\fg \,;
\end{equation}
here, the second equality is due to \eqref{E:FW}, and the last two follow from the hypothesis that $(F^\sb_\fg,N)$ is Hodge--Tate.  Without loss of generality, the polarized mixed Hodge structure $(F^\sb,W_\sb(N,V_\bR))$ is $\bR$--split; then the induced polarized mixed Hodge structure $(F^\sb_\fg,W_\sb(N,\fg_\bR))$ is also $\bR$--split.  Therefore, we may complete $N$ to a standard triple (\S\ref{S:sl2trp}) with 
\begin{equation}\label{E:NYN}
  N \in I^{-1,-1}_\fg\,,\quad
  Y \in I^{0,0}_\fg \tand N^+ \in I^{1,1}_\fg \,.
\end{equation}
It follows that 
\[
  \fp \ = \ W_0(N^+,\fg_\bC)
\]
is a Jacobson--Morosov parabolic subalgebra.  Moreover, the neutral element 
\begin{equation} \label{E:2}
  \hbox{$Y$ acts on $I^{p,p}_\fg$ by the the scalar $2p$,}
\end{equation} 
establishing the necessity of (a).  

Since the infinitesimal period relation is bracket--generating, the grading element \eqref{E:ttE} associated with $\fp$ necessarily acts on $I^{p,q}_\fg$ by the eigenvalue $p$.  Given this, from \eqref{E:1} and \eqref{E:2}  we see that
\begin{equation} \label{E:halfY}
  \hbox{$\half Y$ is the grading element \eqref{E:ttE} associated with $\fp$.}
\end{equation}
Let $\fsl_2\bR \subset \fg_\bR$ be the TDS spanned by the standard triple \eqref{E:NYN}, and let $\tSL_2\bR \subset G_\bR$ be the corresponding subgroup.  By Theorem \ref{T:dcks}(a), the map $z \mapsto \exp(zN) \cdot F^\sb$ is a holomorphic, $\tSL_2\bR$--equivariant, horizontal embedding of the upper-half plane into $D$.  Let $\sH \subset D$ denote the image.  Recall the element $\varrho$ of \eqref{E:bd} and the triple $\{\overline\sE,\sZ,\sE\}$ of \eqref{E:ZvY}.  Note that $\varrho$ lies in the image of $\tSL_2\bC$ and 
\begin{equation}\label{E:phi}
  \varphi \ := \ \varrho(F^\sb) \ \in \ \sH \ \subset \ D \,.
\end{equation}
Taken with \eqref{E:halfY}, this implies 
\begin{equation} \label{E:3}
\begin{array}{l}
  \hbox{the grading element \eqref{E:ttE} associated with the} \\
  \hbox{stabilizer $\tAd_\varrho(\fp)$ of $\varphi$ is 
  $\half \tAd_\varrho(Y) = \half \sZ$.}
\end{array}
\end{equation}  
Then the hypothesis that the infinitesimal period relation is bracket--generating implies 
\begin{equation} \label{E:5}
  \half\sZ \ = \ \ttE_\varphi\,,
\end{equation}
where the latter is the grading element \eqref{E:Ephi} associated with $\varphi$, \cf~\S\ref{S:HS-GE}.  Therefore, by \eqref{SE:kCkperpC}, Lemma \ref{L:CD} and \eqref{E:p-pp}, the $\half \sZ$--graded decomposition \eqref{SE:grading} of $\fg_\bC$ must satisfy
\begin{equation}\label{E:4}
  \fg^\teven \ = \ \fk_\bC \tand 
  \fg^\todd \ = \ \fk_\bC^\perp
\end{equation}
where $\fg_\bR = \fk_\bR \op \fk_\bR^\perp$ is the Cartan decomposition given by the Cartan involution $\varphi(\bi)$.  Observe that
\[
  \overline\sE \, \in \, \fg^1 \,\subset\, \fk^\perp_\bC\,,\quad
  \sZ \, \in \, \fg^0 \,\subset\,\fk_\bC \tand
  \sE \, \in \, \fg^{-1} \,\subset \,\fk^\perp_\bC\,.
\]
Since the Cartan involution acts on $\fg^1 \op\fg^{-1} = \fg^{1,-1}\op\fg^{-1,1}$ by the scalar $-1$, and on $\fg^0 = \fg^{0,0}$ by the scalar $1$, we see that $\{ N^+,Y,N\}$ is a Cayley triple (with respect to $\fk$); equivalently, $\{\overline\sE,\sZ,\sE\}$ is a DKS--triple.  Equation \eqref{E:4} implies that the compact characteristic vector $(\gamma(\sZ);\a'(\sZ))$ of the orbit $\cN$ satisfies Theorem \ref{T:cHT}(b), establishing necessity.

\smallskip

(Uniqueness)  At this point we may observe that if the $\tAd(G_\bR)$--orbit $\cN$ exists, then it is unique: the compact characteristic vector $(\gamma(\cN);\a'(\cN))$ is uniquely determined by \eqref{E:5} and \eqref{E:4}.  Uniqueness of the orbit $\cN$ then follows from Theorem \ref{T:ccv}. 

\smallskip

$(\Longleftarrow)$  Assume that conditions (a) and (b) hold.  Fix a Cartan decomposition $\fg_\bR = \fk_\bR \op \fk^\perp_\bR$ and a Cayley triple $\{N^+,Y,N\}$ (\S\ref{S:rtno_R}).  Set $F^p_\fg = W_{-2p}(N^+,\fg_\bC)$.  The expression \eqref{E:JMf} implies that $Y,N^+ \in W_0(N^+,\fg_\bR)$.  Therefore, $Y$ and $N^+$ stabilize $F^\sb_\fg$.  Given the hypothesis (a), this implies that the $\tSL_2\bC$--orbit of $F^\sb_\fg$ is a holomorphic, equivariant, horizontal embedding $\bP^1 \inj G_\bC/P$.  Arguing as above, the conditions of Theorem \ref{T:cHT}(b) imply \eqref{E:4}; equivalently, $\varphi = \varrho(F^\sb_\fg) \in D \cap \bP^1$.  This implies $D \cap \bP^1 = \sH$ and $z \mapsto \texp(zN)F^\sb$ is a nilpotent orbit.  Then Theorem \ref{T:cks}(a) ensures that $(F^\sb,N)$ is a polarized mixed Hodge structure.  Finally, from $F^p_\fg = W_{-2p}(N^+,\fg_\bC)$, and the fact that $Y$ is even, we see that $\ttE = \half Y$ splits $F^\sb_\fg$, while $Y$ splits $W_\sb(N,\fg)$; it follows that $(F^\sb,N)$ is Hodge--Tate.

\end{proof}

It will be helpful later for us to observe that 
\begin{equation}\label{E:Phi-im}
  \Phi_\infty(F^\sb_\fg,N) \ = \ \varrho(\varphi) 
  \ = \ \varrho^2(F^\sb_\fg) \,.
\end{equation}
The second equality is \eqref{E:phi}.  To see why the first equality holds, set $F^\sb_{\fg,\infty} = \Phi_\infty(F^\sb_\fg,N)$ and observe that \eqref{E:Finfty} implies 
\[
  F^p_{\fg,\infty} \ = \ \bigoplus_{s\ge p} I^{\sb,-s} 
  \ = \ \bigoplus_{s \ge p} I^{-s,-s} \,.
\]
At the same time
\[
  F^p_\fg \ = \ \bigoplus_{s\ge p} I^{s,s} \,.
\]
The assertion now follows from \eqref{E:halfY} and the easily verified 
\begin{equation} \nonumber 
  \tAd^2_\varrho Y \ = \ -Y \,.
\end{equation}

\begin{remark}[Cayley triples and Matsuki points] \label{R:CT-MP}
As we observed in the proof of $(\Longrightarrow)$ above, the standard triple $\{N^+,Y,N\}$ of Theorem \ref{T:cHT} is a Cayley triple with respect to the Cartan involution $\varphi(\bi)$ defined by \eqref{E:phi}, \cf~Lemma \ref{L:CD}.   This implies that \eqref{E:phi} is a Matsuki point (with respect to the Cartan involution $\varphi(\bi)$).
\end{remark}

\subsection{Distinguished grading elements} \label{S:dp}
It may be the case that a Hodge--Tate degeneration $(F^\sb_\fg,N)$ on $D$ is itself induced from a Hodge--Tate degeneration $(\sF^\sb_\fl,N)$ on a Mumford--Tate subdomain $\sD \subset D$.  More precisely, let $\varphi$ be the circle \eqref{E:phi} and suppose that $\fl_\bR \subset \fg_\bR$ is a $\varphi$--stable Levi subalgebra containing $N$.  Then $\varphi(S^1) \subset L_\bR$ by Lemma \ref{L:levi1}.  In this case, setting $\sF^\sb_\fl = F^\sb_\fg \cap \fl_\bC$ defines a Hodge--Tate degeneration $(\sF^\sb_\fl , N)$ on $\sD = L_\bR \cdot \varphi$.  (Here $\sD \subset D$ is the Mumford--Tate domain for the Hodge representation $(\fl_\bR,Q_\fl , \tAd,\varphi)$ of $L_\bR$, \cf~Remark \ref{R:HSonS}.)  A simple test of the neutral element $Y$ will determine whether or not $\fg_\bR$ is the minimal such Levi subalgebra (that is, whether or not there exists $\fl_\bR \subsetneq \fg_\bR$), \cf~Lemma \ref{L:dist}.

A grading element $Y \in \fg_\bC$ is \emph{distinguished} if $\half Y$ is the grading element \eqref{E:ttE} associated with the parabolic $\fp_Y$ and the $Y$--eigenspace decomposition $\fg_\bC = \op \fg_\ell$ satisfies $\tdim\,\fg_0 = \tdim\,\fg_2$.


\begin{theorem}[{Bala--Carter \cite{MR0417306}}] \label{T:BC}
A grading element $Y \in \fg_\bC$ is distinguished if and only if it can be realized as the neutral element of a standard triple $\{N^+,Y,N\}$ with the property that no proper Levi subalgebra $\fl_\bC \subsetneq \fg_\bC$ contains the standard triple.  \emph{(Equivalently, no proper Levi subalgebra contains $N$.)}
\end{theorem}


\begin{lemma}\label{L:dist}
Given a Hodge--Tate degeneration $(F^\sb_\fg,N)$ on $D$, let $\{N^+,Y,N\}$ be the standard triple of Theorem \ref{T:cHT}, and let $\varphi$ be given by \eqref{E:phi}.  The neutral element $Y$ is distinguished if and only if $\fg_\bR$ is the only $\varphi$--stable Levi subalgebra of $\fg_\bR$ containing $N^+$ (equivalently, $N$).
\end{lemma}

\begin{remark}
The hypothesis that $(F^\sb,N)$ is a Hodge--Tate degeneration on a Mumford--Tate domain is essential: there exist nilpotent $N \in \fg_\bR$ with the property that $\fg_\bR$ is the minimal $\varphi$--stable Levi subalgebra of $\fg_\bR$ containing $N$, but for which $Y$ is not even, let alone distinguished.  Such nilpotents are \emph{noticed} \cite{\Noel}.
\end{remark}

\begin{remark} \label{R:dist}
Let $\{\overline\sE,\sZ,\sE\} = \tAd_\varrho\{N^+,Y,N\}$ 
be the DKS--triple in the proof of Theorem \ref{T:cHT}.  Note that $Y$ is distinguished if and only if $\sZ$ is.  Moreover, $\fg_\bR$ is the minimal $\varphi$--stable Levi subalgebra of $\fg_\bR$ containing $N^+$ if and only if $\fg_\bC$ is the minimal conjugation and $\varphi$--stable Levi subalgebra containing $\overline\sE$.
\end{remark}

\begin{proof}
$(\Longrightarrow)$ If $Y$ is distinguished, then $\fg_\bC$ is the smallest Levi subalgebra containing $N^+$ by Theorem \ref{T:BC}.  \smallskip

$(\Longleftarrow)$ By Lemma \ref{L:levi1} and \eqref{E:5} a Levi subalgebra of $\fg_\bC$ is $\varphi$--stable if and only if it contains $\sZ$.  Suppose that $\fg_\bC$ is the only Levi subalgebra of $\fg$ that: (i) contains the DKS--triple $\{\overline\sE,\sZ,\sE\}$, and (ii) can be expressed as the centralizer of an element in $\bi \fk_\bR$.  Any such Levi subalgebra of $\fg_\bC$ is both conjugation and $\varphi(\bi)$--stable.  Then $\overline\sE$ is a \emph{noticed} nilpotent, in the terminology of \cite{\Noel}.  Whence \cite[Lemma 2.1.1]{\Noel} yields $\tdim\,\fg_0 \,\cap\,\fk_\bC \, = \, \tdim\,\fg_2\,\cap\,\fk_\bC^\perp$, where $\fg_\bC = \op \fg_\ell$ is the $\sZ$--eigenspace decomposition.  From \eqref{E:5} and \eqref{E:4}, we see that $\tdim\,\fg_0 = \tdim\,\fg_2$, and $\sZ$ is distinguished by definition.  The lemma now follows from Remark \ref{R:dist}.
\end{proof}

\subsection{Examples} \label{S:egHT}

In the following examples, given $G_\bC$, we apply Theorem \ref{T:cHT} to identify the compact duals $\check D = G_\bC/P$ with an open $G_\bR$--orbit admitting the structure of a Mumford--Tate domain with a Hodge--Tate degeneration.  Keep in mind that, since we are assuming that the infinitesimal period relation is bracket--generating, the compact dual determines the real form, \cf\eqref{E:G_R}.

\begin{example}[The symplectic group $\tSp_8\bC$] \label{eg:HT-Sp8}

Of the $2^4-1 = 15$ conjugacy classes of parabolic subgroups in $G_\bC$, only six are even Jacobson--Morosov; the indexing sets (\S\ref{S:P}) are $I = \{1,2,3,4\} \,,\ \{1,2,4\}\,,\ \{2,4\} \,,\ \{1,4\} \,,\ \{2\} \,,\ \{4\}$, \cf\cite{\CoMc} or \cite{BPR}.  Therefore, the pairs of compact duals $\check D$ with an open $G_\bR$--orbit $D$ admitting the structure of a Mumford--Tate domain with a Hodge--Tate degeneration are as listed in the table below.
\begin{small}
\[ \renewcommand{\arraystretch}{1.5}
\begin{array}{l|cccccc}
\check D & \tFlag^Q_{1,2,3,4}(\bC^8) &  \tFlag^Q_{1,2,4}(\bC^8) 
  & \tFlag^Q_{2,4}(\bC^8) & \tFlag^Q_{1,4}(\bC^8) & \tGr^Q(2,\bC^8) 
  & \tGr^Q(4,\bC^8) \\
\hline
\fg_\bR & \fsp(4,\bR) & \fsp(4,\bR) &  \fsp(4,\bR) & \fsp(4,\bR) & \fsp(2,2) 
        & \fsp(4,\bR) \\
\hline 
V_\bR & \bR^8 & \bR^8 & \bR^8 & \bR^8 & \tw^2\bR^8 & \bR^8  \\
\hline
\bh & (1,\ldots,1) & (1,1,2,2,1,1) & (2,2,2,2) & (1,3,3,1) & (1,8,9,8,1) 
    & (4,4) \\
\end{array}
\]
\end{small}

\noindent
The table also lists a Hodge representation $V_\bR$ realizing $D$ as a Mumford--Tate domain, and the corresponding Hodge numbers.  In all but one of these cases we have $\fg_\bR = \fsp(4,\bC)$ and $V_\bR = \bR^8$; this realizes the Mumford--Tate domain as a period domain.  In the case that $\check D = \tGr^Q(2,\bC^8)$, the standard representation $\bR^8$ does not admit the structure of a Hodge representation (because $\bC^8$ quaternionic, rather than real, with respect to the real form $\fg_\bR$). However, the second exterior power $\tw^2\bR^8$ does admit the structure of Hodge representation that realizes $\tGr^Q(2,\bC^8)$ as the compact dual of a Mumford--Tate domain.
\end{example}

\begin{example}[The orthogonal group $\tSO_9\bC$] \label{eg:HT-SO9}
Of the $2^4-1 = 15$ conjugacy classes of parabolic subgroups in $G_\bC$, only five are even Jacobson--Morosov; the indexing sets (\S\ref{S:P}) are $I = \{1,2,3,4\} \,,\ \{1,2,3\}\,,\ \{1,3\} \,,\ \{3\} \,,\ \{1\}$.  It follows that the pairs of compact duals $\check D$ with an open $G_\bR$--orbit $D$ admitting the structure of a Mumford--Tate domain with a Hodge--Tate degeneration are
\begin{small}
\[ \renewcommand{\arraystretch}{1.5}
\begin{array}{l|ccccc}
\check D & \tFlag^Q_{1,2,3,4}(\bC^8) &  \tFlag^Q_{1,2,3}(\bC^8) 
  & \tFlag^Q_{1,3}(\bC^8) & \tGr^Q(3,\bC^8)
  & \tGr^Q(1,\bC^8) = \cQ^6  \\
\hline
\fg_\bR & \fso(4,5) & \fso(4,5) & \fso(4,5) & \fso(6,3) & \fso(2,7) \\
\end{array}
\]
\end{small}

\noindent Here we may take $V_\bR = \bR^9$ in each case, and the Mumford--Tate domains are all period domains.
\end{example}

\begin{example}[The exceptional Lie group $G_2(\bC)$] \label{eg:HT-G2}

The complex Lie group $G_\bC = G_2(\bC)$ contains three conjugacy classes $\cP_I$ of parabolic subgroups; as discussed in \S\ref{S:P}, they are indexed by the nonempty subsets $I \subset \{1,2\}$.  Parabolics in two of the three may be realized as even Jacobson--Morosov parabolics: the Borel subgroups $\cB = \cP_{\{1,2\}}$ and the maximal parabolics $\cP_{2}$, \cf\cite[\S8.4]{\CoMc}.  (The parabolics in the third class $\cP_1$ may also be realized as Jacobson--Morosov parabolics, but not as \emph{even} Jacobson--Morosov parabolics.)   The complex Lie algebra $\fg_\bC$ admits a single noncompact real form $\fg_\bR$.  The maximal compact subalgebra is $\fk_\bR = \fsu(2) \op \fsu(2)$.  In both cases we may take $V_\bR$ to be the standard representation $\bR^7$.
\begin{a_list}
\item 
In the case of the Borel conjugacy class $\cB$, as discussed in \S\ref{S:rtno_C}, we have $\s(\cN) = (2,2)$.  From the tables of \cite[\S9.6]{\CoMc} we see that $\cN \cap \fg_\bR$ consists of a single $\tAd(G_\bR)$--orbit $\cN$ and $\gamma(\cN) = (4,8)$ and $\a'(\cN) = \s_2(\cN) = -10$.  It follows from Theorem \ref{T:cHT} that $D \subset G_\bC/B$ admits a Hodge--Tate degeneration.  The Hodge numbers are $\bh=(1,1,1,1,1,1,1)$.
\item
For $P \in \cP_2$ we have $\s(\cN) = (0,2)$.  From the tables of \cite[\S9.6]{\CoMc} we see that $\cN \cap \fg_\bR$ consists of two $\tAd(G_\bR)$--orbits.  One of these has characteristic vector $\gamma(\cN) = (0,4)$ and $\a'(\cN) = \s_2(\cN) = -2$.  Theorem \ref{T:cHT} implies $D \subset G_\bC/P$ has a Hodge--Tate degeneration.  The Hodge numbers are $\bh=(2,3,2)$.
\end{a_list}
More generally, the polarized $G_\bR$--orbits in a $G_2(\bC)$--homogeneous compact dual have been determined by Kerr and Pearlstein in \cite[\S6.1.3]{KP2013}.
\end{example}

\subsection{Constraints on the existence of Hodge--Tate degenerations}

In the case that the compact dual is the full flag variety $\check D = G_\bC/B$, that is $P = B$ is a Borel subgroup, we may be explicit about the real forms $G_\bR$ that yield a $G_\bR$--orbit $D \subset \check D$ admitting the structure of a Mumford--Tate domain with a Hodge--Tate degeneration.

\begin{proposition} \label{P:borel}
Let $G_\bC$ be a simple complex Lie group and consider the full flag variety $\check D = G_\bC/B$.  Given a real form $G_\bR$ of $G_\bC$ there exists a $G_\bR$--flag domain $D \subset \check D$ admitting the structure of a Mumford--Tate domain (with bracket--generating IPR) with a Hodge--Tate degeneration if and only if $\fg_\bR$ is one of the following:
\begin{eqnarray*}
  &
  \fsu(p,p) \,,\quad \fsu(p,p\pm1) \,,\quad
  \fsp(n,\bR) \,,
  & \\ &
  \fso(2p\pm1,2p)\,,\quad \fso(2p,2p) \,,\quad \fso(2p+2,2p)\,,
  & \\ &
  \mathrm{E\,II}\,,\quad\mathrm{E\,V}\,,\quad\mathrm{E\,VIII}\,,\quad
  \mathrm{F\,I}\,,\quad
  \mathrm{G}\,.
  &
\end{eqnarray*}
\end{proposition}

\begin{proof}
Hodge--Tate degenerations in full flag varieties are discussed in \cite[Remark V.12]{GGR}.  There it was observed that, if $G_\bC$ is classical (special linear, symplectic or orthogonal), then $\fg_\bR$ is necessarily one of the algebras listed above.  Additionally, for each of the symplectic and orthogonal algebras, a Mumford--Tate domain and Hodge--Tate degeneration are exhibited.

Now consider the special linear algebra $\fg_\bC = \fsl_n\bC$.  If the Mumford--Tate domain admits a Hodge--Tate degeneration, then the complex characteristic vector $\s(\cN)$ is necessarily of the form $(2,\ldots,2)$.  Moreover, \eqref{E:halfY} implies $(1,\ldots,1) = (\s_1(\ttE),\ldots,\s_r(\ttE))$, where $r = n-1$ and $\ttE = \ttE_\varphi$ is the grading element \eqref{E:ttE} associated with the Borel.  Therefore the simple roots $\s_i$ are all noncompact.  Whence the collection $\sS' = \{ \s_1 + \s_2 \,,\, \s_2 + \s_3 \,,\, \s_3+\s_4\,,\ldots,\, \s_{r-1}+\s_r\}$ forms a set of simple roots for $\fk_\bC$.  Attaching the non-compact $-\s_1$ completes $\sS'$ to a set of simple roots for $\fg_\bC$.  From this choice of simple roots we see that Theorem \ref{T:cHT}(b) holds; whence $D$ admits a Hodge--Tate degeneration.  To see that the real form is either $\fsu(p,p)$ or $\fsu(p\pm1,p)$ observe that $-\s_1$ is the unique noncompact simple root in the system $\sS' \cup \{-\s_1\}$.  In the Vogan diagram classification of real forms \cite[\S VI.10]{\Knapp}, this corresponds to painting either the $(p\pm1)$--st or $p$--th node in the Dynkin diagram.

In the case that $G_\bC$ is exceptional, the proposition follows from Theorem \ref{T:cHT} and the tables in \cite[\S9.6]{\CoMc}. 
\end{proof}

\section{Classification theorems} \label{S:CTs}

In this section we prove the two main results of the paper: the classifications of the $\bR$--split polarized mixed Hodge structures (Theorem \ref{T:cPMHS}), and of the horizontal $\tSL(2)$s (Theorem \ref{T:cSL2}).

\subsection{$\bR$--split polarized mixed Hodge structures} 

Let $(F^\sb,N)$ be an $\bR$--split polarized mixed Hodge structure on a Mumford--Tate domain $D$.  Given any $g \in G_\bR$, 
\[
  g \cdot (F^\sb,N) \ := \ ( g \cdot F^\sb , \tAd_gN )
\]
is also an $\bR$--split PMHS on $D$; let 
\[
  [F^\sb,N] \ := \ \{ g \cdot (F^\sb,N) \ | \ g \in G_\bR\}
\]
denote the corresponding $G_\bR$--conjugacy class, and let 
\begin{equation}\label{E:PsiD}
  \Psi_D \ := \ 
  \{ [F^\sb,N] \ | \ (F^\sb,N) \hbox{ is an $\bR$--split PMHS on } D \}
\end{equation}
denote the set of all such conjugacy classes.

Fix a point $\varphi \in D$.  Recall the grading element $\ttE_\varphi$ of \eqref{E:Ephi} and let $\ft \ni \bi\,\ttE_\varphi$ be a compact Cartan subalgebra of $\fg_\bR$.  Given a Levi subalgebra $\fl_\bC \supset \fh$, recall from \S\ref{S:GE} that $\fl_\bC = \fl_\bC^\tss \op \fz$ where $\fz$ is the center of $\fl_\bC$ and $\fl_\bC^\tss = [\fl_\bC , \fl_\bC]$ is semisimple; let $\pi_\fl^\tss : \fl_\bC \to \fl_\bC^\tss$ denote the projection.  Set 
\[
  \cL_{\varphi,\ft} \ := \ 
  \left\{ \begin{array}{c}
  \hbox{$\varphi$--stable Levi subalgebras $\fl_\bR \subset \fg_\bR$ such
        that $\ft \subset \fl_\bR$ and}\\
  \hbox{$2\,\pi^\tss_\fl(\ttE_\varphi)$ is a distinguished 
        semisimple element of $\fl_\bC^\tss$}
  \end{array}\right\} \,.
\]
(The condition, in the definition of $\cL_{\varphi,\ft}$, that $\fl_\bR$ be $\varphi$--stable is added for emphasis/clarity; it follows from $\bi\ttE_\varphi \in \ft \subset \fl_\bR$ which implies that the image of the circle is contained in $L_\bR$.)  In computations it is helpful to note that $2\,\pi^\tss_\fl(\ttE_\varphi)$ is a distinguished semisimple element of $\fl_\bC$  if and only if 
\begin{equation} \label{E:distL}
  \trank\,\fl_\bC^\tss \ +\  
  \#\{ \a \in \Delta(\fl) \ | \ \a(\ttE_\varphi)=0  \} \ = \ 
  \#\{ \a \in \Delta(\fl) \ | \ \a(\ttE_\varphi) = 1 \} \,.
\end{equation}

\begin{lemma}\label{L:DKS}
Given $\fl_\bR \in \cL_{\varphi,\ft}$, there exists a DKS triple $\{ \overline\sE , \sZ , \sE\} \subset \fl_\bC^\tss$ with neutral element $\sZ=2\,\pi^\tss_\fl(\ttE_\varphi)$.  
\end{lemma}

\noindent The lemma is proved in \S\ref{S:prf}.

Let $\fg_\bC = \op \fg^p$ be the $\ttE_\varphi$--eigenspace decomposition \eqref{SE:grading}.  Recall from \eqref{E:p-pp} that the Hodge filtration $F^\sb_{\varphi,\fg}$ of $\fg_\bC$ induced by $\varphi$ is given by $F^p_{\varphi,\fg} = \op_{q\ge p} \fg^q$.  The parabolic $\fp_\varphi = \fg^0 \op \fg^+$ is the Lie algebra of the stabilizer $P_\varphi \subset G_\bC$ of $\varphi$, and the $0$--eigenspace $\fg^0$ is a Levi subalgebra of $\fg_\bC$ (\S\ref{S:P}) containing the Cartan subalgebra $\fh = \ft\ot_\bR\bC$.  Let $\sW^0 \subset \sW \subset \tAut(\fh)$ denote the Weyl group of $\fg^0$ (Remark \ref{R:weyl}).  Then $\sW^0$ acts on $\cL_{\varphi,\ft}$.  Given $\fl_\bR \in \cL_{\varphi,\ft}$, let $[\fl_\bR]$ denote the $\sW^0$--conjugacy class, and let
\[
  \Lambda_{\varphi,\ft} \ := \ 
  \{ [\fl_\bR] \ | \ \fl_\bR \in \cL_{\varphi,\ft} \}
\]
be the corresponding set of $\sW^0$--conjugacy classes.\footnote{The r\^ole of $\sW^0$ here is anticipated by Cattani and Kaplan's \cite[Proposition 3.29]{MR496761}.}

Finally we note that \eqref{SE:kCkperpC} and \eqref{E:p-pp} imply $\fg^0$ has compact real form 
\[
  \fk^0_\bR \ := \ \fg^0 \cap \fg_\bR \ = \ \fp_\varphi \,\cap\, \fk_\bR\,;
\]
let $K^0_\bR = P_\varphi \cap K_\bR$ denote the corresponding Lie subgroup.  (Note that $K^0_\bR$ is the stabilizer of $\varphi \in D$ in $G_\bR$.)  Then
\begin{equation} \label{E:W0}
  \hbox{elements of $\sW^0$ admit representatives in $K_\bR^0$.}
\end{equation}

\begin{theorem} \label{T:cPMHS}
Let $V_\bR$ be a Hodge representation of $G_\bR$, and assume that the infinitesimal period relation on the associated Mumford--Tate domain $D = G_\bR/K^0_\bR$ is bracket--generating.  With the notation above, we have:
\begin{a_list_emph}
\item 
There is a bijection $\Psi_D \leftrightarrow \Lambda_{\varphi,\ft}$.  That is, up to the action of $G_\bR$, the $\bR$--split polarized mixed Hodge structures on $D$ are indexed by the $\sW^0$--conjugacy classes of $\cL_{\varphi,\ft}$. 
\item
Given $\fl_\bR \in \cL_{\varphi,\ft}$, let $\{ \overline\sE , \sZ , \sE\} \subset \fl_\bC^\tss$ be a DKS--triple with neutral element $\sZ=2\,\pi^\tss_\fl(\ttE_\varphi)$, 
\emph{\cf~Lemma \ref{L:DKS}}.  The element $[F^\sb,N] \in \Psi_D$ corresponding to $[\fl_\bR] \in \Lambda_{\varphi,\ft}$ is represented by $(F^\sb,N) = \varrho^{-1} \cdot (\varphi,\sE)$, where 
\[
  \varrho \ := \ \texp\,\bi\tfrac{\pi}{4} ( \sE + \overline\sE )
  \ \in \ L_\bC \,.
\]
\item
The image of the reduced limit period mapping is $\Phi_\infty(F^\sb,N) = \varrho(\varphi) = \varrho^2(F^\sb)$.   
\item
If $V_\bC = \op \, V^\m$ is the weight space decomposition (with respect to $\fh$), then the Deligne splitting $V_\bC = \op I^{p,q}$ induced by $(F^\sb,N)$ is given by 
\begin{equation} \label{E:Ipq}
  {\varrho}(I^{p,q}) \ = \ 
  \bigoplus_{\mystack{\m(\ttE_\varphi)=p}{\m(\sZ) = p+q}} V^\m \,.
\end{equation}
\item 
With respect to the Deligne splitting $\fg_\bC = \op\,I^{p,q}_\fg$ we have $\fl_\bC \subset \op I^{p,p}_\fg$.
\end{a_list_emph}
\end{theorem}

\noindent
%
Theorem \ref{T:cPMHS} is proved in \S\ref{S:prf}, and a number of examples are worked out in \S\ref{S:egPMHS}.  As will be discussed in \S\ref{S:po}, Theorem \ref{T:cPMHS}(c) yields a parameterization of the polarized orbits in $\tbd(D) \subset \check D$.

\subsection{Horizontal $\tSL(2)$s} \label{S:hSL2}

In this section we will show that Theorem \ref{T:cPMHS} yields a classification of the horizontal $\tSL(2)$s on $D$, up to the action of $G_\bR$.   

Let 
\begin{eqnarray*}
  \fsl_2\bR & = & \tspan_\bR\{ \bn^+ , \by , \bn \} \,,\\
  \fsl_2\bC & = & \tspan_\bC\{ \bn^+ , \by , \bn \}
  \ = \ \tspan_\bC\{ \overline\be , \bz , \be \}
\end{eqnarray*}
be the algebras defined by \eqref{E:stdtri_sl2} and \eqref{E:stdtri_su11}.  Given $\varphi \in D$, recall the grading element $\ttE_\varphi$ given by \eqref{E:Ephi}, and let $\fg_\bC = \op \fg^p$ be the corresponding eigenspace decomposition given by \eqref{SE:grading}.  The latter is also the Hodge decomposition by \eqref{E:p-pp}.  A \emph{horizontal $\tSL(2)$ at $\varphi$} is given by a representation $\upsilon : \tSL(2,\bC) \to G_\bC$ such that 
\begin{subequations}\label{SE:hTDS}
\begin{equation}
  \upsilon(\tSL(2,\bR)) \ \subset \ G_\bR 
\end{equation}
and 
\begin{equation}
  \upsilon_* \overline\be \ \in\  \fg^1 \,,\quad
  \upsilon_* \bz \ \in\  \fg^0 \,,\quad
  \upsilon_* \be \ \in\  \fg^{-1} \,.
\end{equation}
\end{subequations}
We will say that $\upsilon$ is a \emph{horizontal $\tSL(2)$} if it is horizontal at some $\varphi \in D$.  

\begin{remark}\label{R:hTDS}
Observe that \eqref{SE:hTDS} implies that $\upsilon_*\{ \overline\be , \bz , \be\}$ is a DKS--triple with respect to the maximal compact subgroup $K_\bR \subset G_\bR$ determined by the Cartan involution $\varphi(\bi)$; likewise $\upsilon_*\{ \bn^+ , \by , \bn\}$ is a Cayley triple.
\end{remark}

Note that $g \in G_\bR$ acts on the set of horizontal $\tSL(2)$s by $\upsilon \mapsto g \cdot \upsilon$.  Let 
\[
  \Upsilon_D \ := \ \{ [\upsilon] \ | \ \upsilon \hbox{ is a horizontal $\tSL(2)$}\}
\]
be the set of $G_\bR$--equivalence classes.

\begin{theorem}\label{T:cSL2}
With the notation and assumptions of Theorem \ref{T:cPMHS}, we have:
\begin{a_list_emph}
\item
There is a bijection $\Upsilon_D \leftrightarrow \Lambda_{\varphi,\ft}$.  That is, up to the action of $G_\bR$, the horizontal $\tSL(2)$s on $D$ are parameterized by the $\sW^0$--conjugacy classes of $\cL_{\varphi,\ft}$.
\item
Given $\fl_\bR \in \cL_{\varphi,\ft}$, let $\{ \overline\sE , \sZ , \sE\} \subset \fl_\bC^\tss$ be a DKS--triple with neutral element $\sZ=2\,\pi^\tss_\fl(\ttE_\varphi)$, 
\emph{\cf~Lemma \ref{L:DKS}}.   The equivalence class $[\upsilon] \in \Upsilon_D$ corresponding to $[\fl_\bR] \in \Lambda_{\varphi,\ft}$ is represented by the $\upsilon : \tSL_2\bC \to L_\bC^\tss$ given by 
\begin{equation}\label{E:hTDS}
  \upsilon_* \overline\be \ = \ \overline\sE \,,\quad
  \upsilon_* \bz \ = \ \sZ \,,\quad
  \upsilon_* \be \ = \ \sE \,.
\end{equation}
\end{a_list_emph}
\end{theorem}


\begin{proof}
The result will follow from Theorem \ref{T:cPMHS} and the $G_\bR$--equivariant bijection \eqref{E:equiv}.  This bijection is well--known, \cf\cite{\CKpmhs, \CKSdeg, \CKextn, \Schmid, MR1245714}; the following proof is given the sake of completeness.

Given $(F^\sb,N)$, the Deligne splitting $\fg_\bC = \op I^{p,q}_\fg$ defines a semisimple $Y \in \fg_\bR$ by $\left. Y \right|_{I^{p,q}_\fg} = (p+q) \one$.  There exists a unique $N^+ \in \fg_\bR$ completing the pair $\{ Y , N\}$ to a standard triple \cite[pp.~477]{\CKSdeg}.  As discussed in Remark \ref{R:CT-MP}, this standard triple is a Cayley triple with respect to the Cartan involution $\varphi(\bi)$ defined by \eqref{E:phi}.  The corresponding Cayley transform \eqref{E:DKStrp} defines a horizontal $\tSL(2)$ $\upsilon$ at $\varphi = \varrho(F^\sb)$ by \eqref{E:hTDS}.  This defines the map from $\bR$--split polarized mixed Hodge structures to horizontal $\tSL(2)$s.

Conversely, suppose that $\upsilon$ is a horizontal $\tSL(2)$ at $\varphi \in D$.  By Remark \ref{R:hTDS}, \eqref{E:hTDS} defines a DKS--triple $\{ \overline\sE,\sZ,\sE\}$.  Let $\{ N^+ , Y , N\}=\upsilon_*\{ \bn^+ , \by , \bn\}$ be the corresponding Cayley triple, which is defined by \eqref{E:ZvY} and \eqref{E:bd}.  Recalling that $\varrho$ 
is given by \eqref{E:bd}, define $F^\sb = \varrho^{-1}(\varphi)$.  Then $(F^\sb,N)$ is a nilpotent orbit.  Moreover, the Deligne splitting $V_\bC = \op I^{p,q}$ of the corresponding polarized mixed Hodge structure is as given by \eqref{E:I'} in Remark \ref{R:I'}, and is manifestly $\bR$--split.  This defines the map from horizontal TDS to $\bR$--split polarized mixed Hodge structures.  Moreover, this map is easily seen to be the inverse of the map defined in the previous paragraph.
\end{proof}

\subsection{Proof of Theorem \ref{T:cPMHS}} \label{S:prf}

We begin with the 

\begin{proof}[Proof of Lemma \ref{L:DKS}]
As a distinguished semisimple element $\sZ$ is the neutral element of a standard triple $\{ \sE^+ , \sZ , \sE\}$.

Let $\fl_\bC = \op \fl^p$ be the $\ttE_\varphi$--eigenspace decomposition.  We have $\fl^p = \fl_\bC \cap \fg^p$, $\fl_\bC \cap \fk_\bC = \fl^\teven$ and $\fl_\bC \cap \fk_\bC^\perp = \fl^\todd$.  Since $\sZ$ differs from $2\ttE_\varphi$ 
by an element of the center of $\fl_\bC$, we see that $\sZ$ acts on $\fl^p$ by the eigenvalue $2p$. 
From $[\sZ,\sE^+] = 2\sE^+$ and $[\sZ,\sE] = -2\sE$, we see that $\sE^+ \in \fl^{1} \subset \fk_\bC^\perp$ and $\sE \in \fl^{-1} \subset \fk_\bC^\perp$. 
It remains to show that $\sE^+$ and $\sE$ may be chosen so that $\sE^+ = \overline\sE$.  

Let $K' = L\cap K$ and $\fk' = \fl \cap \fk$.  It is a consequence of the Djokovi{\'c}--Kostant--Sekiguchi correspondence and Remark \ref{R:cayley} that  $\{ \sE^+ , \sZ , \sE\}$ is $K'_\bC$--conjugate to a DKS--triple $\{ \overline\sE' , \sZ' , \sE'\}$ in $\fl_\bC^\tss$, \cf\cite{MR867991}.  By construction $\sZ \in \bi \ft$, and $\ft$ is a Cartan subalgebra of $\fk_\bR'$.  Therefore $\sZ'\in\bi\fk_\bR'$ is $K_\bR'$--conjugate to an element of $\bi\ft$.  So, without loss of generality, $\sZ' \in \bi\ft$.

The claim will follow once we show that $\sZ$ and $\sZ'$ are conjugate under the Weyl group $\sW_{K'} \subset \tAut(\ft)$ of $\fk'_\bR$.  First, observe that $\sZ$ and $\sZ'$ (i) lie in the same Cartan $\fh$, and (ii) are (twice) the grading elements associated with parabolic subalgebras $\fp_{\sZ}$ and $\fp_{\sZ'}$ that are $K'_\bC$--conjugate; it follows that $\sZ$ and $\sZ'$ are conjugate under an element $w$ of the Weyl group of $\fl_\bC$.  Because $\fp_{\sZ}$ and $\fp_{\sZ'}$ are conjugate under $K'_\bC$, the element $w$ must preserve the set of compact roots $\Delta(\fk_\bC) \subset \Delta$, and is therefore an element of $\sW_{K'}$. 
\end{proof}

We now turn to the proof of Theorem \ref{T:cPMHS}.  To establish the bijection $\Lambda_{\varphi,\ft} \leftrightarrow \Psi_D$, first suppose we are given a Levi subalgebra $\fl_\bC\in \cL_{\varphi,\ft}$; the corresponding $[F^\sb,N]\in\Psi_D$ is obtained as follows.  Given the DKS--triple of Lemma \ref{L:DKS}, let $\{ N^+ , Y , N\} = \tAd_\varrho^{-1}\{\overline\sE,\sZ,\sE\} \subset \fl^\tss_\bR$ 
be the corresponding Cayley triple.  

From Lemma \ref{L:levi1} we see that $\varphi$ induces a sub--Hodge structure on the real form $\fl_\bR$.  Let $\sD = L_\bR \cdot \varphi \subset D$ denote the corresponding Mumford--Tate domain.

\begin{claim} \label{cl:1}
Define $\sF^\sb_\fl \in \check\sD$ by $\sF^{-p}_\fl = W_{2p}(N^+,\fl_\bC)$.  Then the pair $(\sF^\sb_\fl,N)$ defines a Hodge--Tate degeneration on $\sD$.
\end{claim}  

\begin{proof}
The fact that $\fl_\bC$ is $\varphi(\bi)$--stable implies
\[
  \fl_\bC \ = \ ( \fl_\bC \cap \fk_\bC ) \,\op\,
  ( \fl_\bC \cap \fk_\bC^\perp ) \tand
  \fl_\bC^\tss \ = \ ( \fl_\bC^\tss \cap \fk_\bC ) \,\op\,
  ( \fl_\bC^\tss \cap \fk_\bC^\perp ) \,.  
\]
are both Cartan decompositions.  Since $\fl_\bC \supset \fh$, we may identify the roots of $\fl_\bC$ with a subset of the roots of $\fg_\bC$, and under this identification the (non)compact roots of $\fl_\bC$ are (non)compact roots of $\fg_\bC$.  It follows from \eqref{SE:kCkperpC} and \eqref{E:p-pp} that: (i) $\a(\sZ) \equiv 0$ mod $4$ for all compact roots of $\fl_\bC^\tss$, and (ii) $\b(\sZ)$ is even and $\half\b(\sZ)$ is odd for all non-compact roots.  The claim now follows from Theorem \ref{T:cHT}.
\end{proof}

\noindent Remark \ref{R:underHT} and Claim \ref{cl:1} imply that $(\sF^\sb_\fl,N)$ induces a nilpotent orbit $(F^\sb_\fg,N)$ on $D$.  At this point Theorem \ref{T:cPMHS}(b,c) follows from \eqref{E:ZvY}, \eqref{E:bd}, \eqref{E:phi} and \eqref{E:Phi-im}. 

\smallskip 

The nilpotent orbit $(F^\sb,N)$ depends on both the Levi subalgebra $\fl_\bR$ and our choice of DKS--triple $\{\overline\sE,\sZ,\sE\}$.  Suppose $\{ \overline\sE{}' , \sZ , \sE'\}$ is a second DKS--triple, also containing $\sZ$ as the neutral element.  Then Rao's \cite[Theorem 9.4.6]{\CoMc} implies that the triples are conjugate under $G^0 \cap L_\bR^\tss$.  It is then straightforward to confirm that the nilpotent orbit $({}'F^\sb,N')$ associated with the second DKS--triple is $G^0 \cap L_\bR^\tss$--congruent to $(F^\sb,N)$.  Whence the two nilpotent orbits determine the same conjugacy class $[F^\sb,N] \in \Psi_D$, and we have a well--defined map $\cL_{\varphi,\ft} \to \Psi_D$.  Finally, \eqref{E:W0} and Theorem \ref{T:cPMHS}(b) imply the map descends to $\Lambda_{\varphi,\ft} \to \Psi_D$.

\medskip

To address the second half of the correspondence asserted in Theorem \ref{T:cPMHS}(a) suppose that $[F^\sb,N] \in \Psi_D$.  We normalize our choice of representative $(F^\sb,N)$ as follows.  Let $\fg_\bC = \op I^{p,q}_\fg$ be the associated Deligne splitting, and let $\tilde \fl_\bC = \op I^{p,p}_\fg$ be the Levi subalgebra of Theorem \ref{T:underHT}.  Since $(F^\sb,N)$ is $\bR$--split, $\tilde\fl_\bC$ is necessarily stable under conjugation.  Moreover, we may complete $N$ to standard triple $\{N^+,Y,N\} \subset \tilde\fl^\tss_\bR$ so that \eqref{E:NYN} holds.  Conjugating $(F^\sb,N)$ by an element $g\in G_\bR$ if necessary, we may assume that this is a Cayley triple.  Then $\tilde\fl_\bC$ is $\varphi(\bi)$--stable.  Let $\{ \overline\sE,\sZ,\sE\} \subset \tilde\fl_\bC^\tss$ be the Cayley transform \eqref{E:DKStrp} of the Cayley triple, and let $\tilde\varphi$ be as given by \eqref{E:phi}.  Then $\tilde\varphi$ is a $K$--Matsuki point of $D$ (Remark \ref{R:CT-MP}), and therefore $K_\bR$--conjugate to $\varphi$ by \eqref{E:mat}.  So, conjugating $(F^\sb,N)$ by an element $g\in K_\bR$, if necessary, we may assume that $\tilde \varphi = \varphi$.

Let $\fl_\bC \subset \tilde\fl_\bC$ be a minimal conjugation and $\varphi$--stable Levi subalgebra containing the DKS--triple.  (Such a Levi is not unique; however any two such are conjugate under the reductive centralizer $Z(\overline\sE,\sE)$ of $\overline\sE$ and $\sE$ in $K^0_\bR \cap L_\bR$, see the proof of \cite[Proposition 1.1.3]{\Noel}.)  Then $\sZ$ is a distinguished semisimple element of the semisimple factor $\fl_\bC^\tss$ by Lemma \ref{L:dist} and Remark \ref{R:dist}.  By construction, $L_\bR$ admits the Hodge representation $(\fl_\bR,Q_\fl,\tAd,\varphi)$.  Therefore, $\fl_\bR$ has a compact Cartan subalgebra $\tilde\ft \ni \bi \ttE_\varphi$.  Since both Cartans $\ft$ and $\tilde\ft$ contain $\bi \ttE_\varphi$, they are necessarily Cartan subalgebras of the compact $\fk^0_\bR$.  Therefore, up to conjugation by $g \in K_\bR^0$, we may assume that $\ft = \tilde\ft$.  Thus $\fl_\bR \in \cL_{\varphi,\ft}$.  At this point, the ambiguity in our choice of minimal $\fl_\bC$ (see the parenthetical remark above) is up to the action of the Weyl group of $Z(\overline\sE,\sE)$.  Since the latter is a subgroup of $\sW^0$, we have a well--defined map $\Psi_D \to \Lambda_{\varphi,\ft}$.  This completes the proof of Theorem \ref{T:cPMHS}(a).

\medskip

It remains to establish Theorem \ref{T:cPMHS}(d).  The induced Deligne splitting $V_\bC = \op I^{p,q}$ may be obtained as follows.  Let $V_\bC = \op\,V^\lambda$ be the weight space decomposition of $V_\bC$.  That is, $\lambda \in \fh^*$ and $v \in V^\lambda$ if and only if $\xi(v) = \lambda(\xi) v$ for all $\xi \in \fh$.  It is immediate from \eqref{E:Phi-im} that $I^{p,\sb}$ is the 
\[
  \ttE' \ := \ \tAd_\varrho^{-1}(\ttE_\varphi)
\]
eigenspace for the eigenvalue $p$.  That is, 
\begin{equation} \nonumber
  I^{p,\sb} \ = \ \bigoplus_{\m(\ttE')=p} {}'V^\m\,;
\end{equation}
here $V_\bC = \op\,{}'V^\m$ is the weight space decomposition with respect to the Cartan subalgebra $\fh'=\tAd_\varrho^{-1}(\fh)$.  On the other hand, by \eqref{E:FW}, 
\begin{equation} \nonumber
  \bigoplus_{p+q=\ell} I^{p,q} \ = \ 
  \bigoplus_{\m(Y)=\ell} {}'V^\m
\end{equation}
is the $\ell$--eigenspace for $Y = \varrho^{-1}(\sZ)$.  
Thus
\begin{equation} \label{E:'Ipq}
  I^{p,q} \ = \ \bigoplus_{\mystack{\m(\ttE')=p}{\m(Y) = p+q}} {}'V^\m \,.
\end{equation}
Applying $\tAd_{\varrho}$ to \eqref{E:'Ipq} yields \eqref{E:Ipq}, and completes the proof of Theorem \ref{T:cPMHS}.
\hfill\qed

\begin{remark}[Computing $\sZ$] \label{R:Z}
If we wish to compute the Deligne splitting \eqref{E:Ipq} it is necessary to determine $\sZ$.  As a reductive algebra, $\fl_\bC$ decomposes into a direct sum of its center and a semisimple factor
\[
  \fl_\bC \ = \ \fz_\bC \,\op\,\fl_\bC^\tss \,;
\]
the key is to recall (Theorem \ref{T:cPMHS}(b)) that 
\begin{equation} \label{E:Z}
\hbox{$\sZ$ is the image of $2\ttE_\varphi$ under the projection $\fl_\bC\to\fl^\tss_\bC$.} 
\end{equation}
Let $\sS' \subset \sS \subset \fh^*$ be a choice of simple roots for $\fl_\bC^\tss \subset \fg_\bC$.
We have 
\[
  \fz_\bC \ = \ \tspan_\bC\{\ttS^j \ | \s_j \not\in\sS'\} \,.
\]
Likewise, the Cartan subalgebra of the semisimple factor is
\[
  \fh\,\cap\,\fl^\tss_\bC \ = \ \tspan_\bC\{ \ttH^i \ | \ \s_i \in \sS' \}\,,
\]
where $\ttH^i \subset [\fg^{\s_i} \,,\, \fg^{-\s_i}] \subset \fh$ is defined by $\s_i(\ttH^i) = 2$ (no sum over $i$).  The sets $\{\ttS^i\}_{i=1}^r$ and $\{\ttH^i\}_{i=1}^r$ are the bases $\fh$ dual to the simple roots and fundamental weights, respectively.  In particular, if $C = (C^i_j)$ is the Cartan matrix, so that $\s_i = C_i^j \w_j$, then $\ttH^j = C^j_i \ttS^i$.  Moreover, $\{ \ttS^j \ | \ \s_j \not\in\sS'\}\,\cup\,\{\ttH^i \ | \ \s_i \in \sS'\}$ is a basis of $\fh$.  Therefore, we may write 
\[
  \ttE_\varphi \ = \ \sum_{\s_j \not\in\sS'} n_j \,\ttS^j 
  \ + \ \sum_{\s_i\in\sS'} m_i \,\ttH^i\,,
\]
and \eqref{E:Z} yields
\[
  \sZ \ = \ 2 \sum_{\s_i\in\sS'} m_i\,\ttH^i \ \in \  \fh\,\cap\,\fl_\bC^\tss \,.
\]
\end{remark}

\begin{remark}[The Deligne splitting] \label{R:I'}
By \eqref{E:Z} we have $\ttE_\varphi = \half \sZ + \z$ 
with $\z \in \fz_\bC$.  Indeed the discussion of Remark \ref{R:Z} yields
\[
  \z \ = \ \sum_{\s_j \not\in\sS'} n_j \,\ttS^j \ \in \ \fh \,\cap\,\fz_\bC \,.
\]
Since both $\ttE_\varphi$ and $\sZ$ are imaginary (i.e.,~they lie in $\bi\fg_\bR$), $\z$ is as well.  Observe that 
\[
  \ttE' \ = \ \half Y + \z \,,
\]
and from this we may conclude that 
\[
  \overline{\ttE'} \ = \ \half Y - \z \,.
\]
Since $\sZ , \z \in \fh$, we have $Y , \z \in \fh'$ so that $\overline{\ttE'} \in \fh'$.  It follows that the Deligne splitting \eqref{E:'Ipq} is alternatively given by 
\begin{equation} \label{E:I'}
  I^{p,q} \ = \ \bigoplus_{\mystack{\mu(\ttE') = p}{\mu(\overline{\ttE'})=q}} {}'V^\m
\end{equation}
where $V_\bC = \op\, {}'V^\m$ is the weight space decomposition with respect to $\fh' = \tAd_\varrho^{-1}(\fh)$.
\end{remark}

\subsection{Polarized orbits} \label{S:po}

Let $D \subset \check D$ be a Mumford--Tate domain.  We say that a $G_\bR$--orbit $\cO \subset \tcl(D)$ is \emph{polarized (relative to $D$)} if it contains the image $\Phi_\infty(F^\sb,N)$ of a point $F^\sb \in \tilde B(N)$ under the reduced limit period mapping \eqref{E:Phi_infty}.  We think of the polarized orbits as the ``Hodge--theoretically accessible'' orbits.

Let $(\tilde F^\sb,N)$ be the $\bR$--split polarized mixed Hodge structure \eqref{E:tildeF} associated with $(F^\sb,N)$.  From Theorem \ref{T:dcks}(c) and \eqref{E:Phi_infty}, we see that $\Phi_\infty(F^\sb,N) = \Phi_\infty(\tilde F^\sb,N)$.  So, \emph{for the purpose of studying polarized orbits, it suffices to consider $\bR$--split polarized mixed Hodge structures.}  From Theorem \ref{T:dcks}(a,c) one may also deduce that \emph{$\tilde F^\sb$ and $\tilde F^\sb_\infty$ lie in the same $G_\bR$--orbit in $\check D$.} (See the proof of \cite[Lemma 3.12]{MR840721}.)

Since $\Phi_\infty( g \cdot F^\sb , \tAd_gN ) = g \cdot \Phi_\infty(F^\sb,N)$, we see that any two $G_\bR$--congruent $\bR$--split polarized mixed Hodge structures parameterize the same $G_\bR$--orbit $\cO \subset \tbd(D)$.  We say that $\cO$ is \emph{the orbit polarized by $[F^\sb,N] \in \Psi_D$}.  Theorem \ref{T:cPMHS}(c) describes the image of the surjection 
\begin{equation} \label{E:sur}
  \Psi_D \ \sur \ \{\hbox{polarized } \cO \subset \tcl(D)\}\,.
\end{equation}

\begin{remark}
The parameterization \eqref{E:sur} of the polarized orbits generalizes a construction of \cite{KR1} which obtains polarized $G_\bR$--orbits $\cO \subset \tbd(D)$ from sets of strongly orthogonal noncompact roots.  In the case that $D$ is Hermitian symmetric, all the boundary orbits $\cO \subset \tbd(D)$ are polarizable; they are all parameterized by the \cite{KR1}--construction \cite[Theorem 3.2.1]{\FHW}; and the parameterization is essentially that given by the Harish--Chandra compactification of $D$.
\end{remark}

Polarized orbits have received much attention recently, \cf\cite{MR3115136, GGR, KP2013, KR1}.  One basic result is the following.

\begin{theorem}[{Kerr--Pearlstein \cite{KP2013}}] \label{T:KP}
The complexified normal space to $\cO \subset \check D$ at the point $\varrho(\varphi) = \Phi_\infty(F^\sb,N)$ is 
\begin{equation} \label{E:NO}
  N_{\varrho(\varphi)} \cO \ot \bC \ = \ \bigoplus_{p,q>0} I^{-p,-q}_\fg\,.
\end{equation}
In particular, the (real) codimension of the (polarized) $G_\bR$--orbit $\cO$ is
\begin{equation} \label{E:codim}
  \tcodim_{\check D}\,\cO \ = \ \tdim_\bC \bigoplus_{p,q>0} I^{p,q}_\fg \,.
\end{equation}
Moreover, the boundary $\tbd(D) \subset \check D$ contains codimension one $G_\bR$--orbits and they are all polarized.  In this case the normal space
\begin{equation}\label{E:codim1N}
  N_{\varrho(\varphi)}\cO \ = \ \fg^{-\a}_\bR
\end{equation}
is naturally identified with a real root space.
\end{theorem}

\noindent  Recall the set $I(\fp) = \{ i \ | \ \fg^{-\s_i} \not\subset \fp\} = \{ i \ | \ \s_i(\ttE) = 1 \}$ of \eqref{E:I}.  We will see that Theorems \ref{T:cPMHS} and \ref{T:KP} yield

\begin{proposition}\label{P:codim1}
The boundary $\tbd(D) \subset \check D$ contains exactly $|I(\fp)|$ codimension one $G_\bR$--orbits.  
\end{proposition}

\begin{remark}
In the case that $P$ is maximal (equivalently, $|I(\fp)| = 1$), Proposition \ref{P:codim1} was proven in \cite{KR1}.
\end{remark}

\begin{proof}
From \eqref{E:NYN}, \eqref{E:NO} and \eqref{E:codim1N} we see that $N$ spans $I^{-1,-1}_\fg = \fg^{-\a}$ when $\varrho(\varphi)$ lies in a codimension one $G_\bR$--orbit.  It follows that the Levi subalgebra $\fl_\bC$ of Theorem \ref{T:cPMHS} has rank one, and the semisimple factor is the $\fsl_2\bC \subset \fg_\bC$ with simple root $\{\a\}$.  In particular, $\a = \s_i$ for some $i \in I(\fp)$.  Whence, $\tbd(D)$ contains at most $|I(\fp)|$ codimension one orbits.

Since no two $\s_i$, with $i \in I(\fp)$, are congruent under the Weyl group $\sW^0$ of $\fg^0$, in order to see that equality holds we must show that every $i \in I(\fp)$ yields a codimension one orbit.  Let $\fl_\bC$ be the rank one Levi subalgebra with simple root $\s_i$.  Then $\fl_\bR = \fl_\bC \cap \fg_\bR \in \cL_{\varphi,\ft}$.  Moreover, in this case $\sZ = \ttH^i$, where $\{ \ttH^j \}_{j=1}^r$ is the basis of $\fh$ dual to the fundamental weights.  The fact that the $G_\bR$--orbit through $\varrho(\varphi)$ has codimension one is \cite[Lemma 6.52]{KR1}.
\end{proof}


\subsection{Examples} \label{S:egPMHS} 

Suppose that $D$ is a Mumford--Tate domain for a Hodge representation $(V_\bR,Q,\varphi)$ of $G_\bR$.  In the examples that follow we use Theorem \ref{T:cSL2} to enumerate the set $\Upsilon_D$ of horizontal $\tSL(2)$s on $D$ (modulo the action of $G_\bR$).  More precisely, given $[\upsilon] \in \Upsilon_D$, let $[\fl_\bR] \in \Lambda_{\ft,\varphi}$ be the corresponding conjugacy class under Theorem \ref{T:cSL2}; and let $[F^\sb,N] \in \Psi_D$ be the corresponding (conjugacy class of) nilpotent orbit under Theorem \ref{T:cPMHS}.  We will:
\begin{numlist}
\item  Identify a representative of $[\fl_\bR]$ by describing the simple roots $\sS'$ of $\fl_\bC$ as a subset of the roots $\Delta$ of $\fg_\bC$.  The Levis $\fl_\bR$ of $\cL_{\varphi,\ft}$ are identified as follows.  As discussed in Remark \ref{R:levi}, the Levi subalgebras $\fl_\bC$ of $\fg_\bC$ that contain $\fh = \ft \ot \bC$ are in bijection with the subsets $\{ w \sS_0 \ | \ w \in \sW \,,\ \sS_0 \subset \sS \}$.  This is a finite collection of subsets.  For each subset we consider the corresponding Levi $\fl$ and compute $\sZ = \pi^\tss_\fl(\ttE_\varphi)$, \cf~Remark \ref{R:Z}.  We then compute the $\sZ$--eigenspace decomposition of $\fl^\tss_\bC$ to determine whether or not $\sZ$ is a distinguished element of $\fl^\tss_\bC$, \cf\eqref{E:distL}.  
\item
Compute the codimension \eqref{E:codim} of the $G_\bR$--orbit $\cO$ polarized by $[F^\sb,N]$.
\item 
Determine the Deligne splitting \eqref{E:Ipq} of the $\bR$--split polarized mixed Hodge structure $(F^\sb,N)$.  The splittings will be depicted by pictures in the $pq$--plane that place a $\sb$ at the point $(p,q)$ if $I^{p,q}\not=0$.   When considering those pictures, keep in mind that $N \in I^{-1,-1}_\fg$, so that $N : I^{p,q} \to I^{p-1,q-1}$.
\end{numlist}
The Hodge diamond may fail to distinguish two distinct conjugacy classes in $\Upsilon_D$; see Remarks \ref{R:egD4P2} and \ref{R:egA2B}(a).

Throughout $(i) \in \sW$ will denote the simple reflection in the hyperplane $\s_i^\perp \subset \fh^*$.

\begin{example}[Period domain for $\bh = (1,3,1)$] \label{eg:B2P1}
We have $G_\bR = \tSO(3,2)^\circ$ and $\ttE_\varphi = \ttS^1$ so that $\sW^0 = \{ \one , (2) \}$.  In this case $D$ is Hermitian symmetric and all the $G_\bR$--orbits of $\tbd(D)$ are polarized.  Applying Theorem \ref{T:cSL2}, we find that $\Upsilon_D$ consists of two elements: 
\begin{center}
\begin{footnotesize}
\setlength{\unitlength}{10pt}
\begin{picture}(10,4)
\put(0,0){\vector(1,0){3}} \put(0,0){\vector(0,1){3}}  \put(0,1){\circle*{0.3}} \put(1,0){\circle*{0.3}} \put(1,1){\circle*{0.3}} \put(1,2){\circle*{0.3}} \put(2,1){\circle*{0.3}}
\put(4,2.4){$\sS' = \{ \s_1 \}$}
\put(4,1.2){$\sZ = 2\ttS^1 - \ttS^2$}
\put(4,0){$\tcodim\,\cO = 1$}
\end{picture}
\hsp{30pt}
\begin{picture}(10,4)
\put(0,0){\vector(1,0){3}} \put(0,0){\vector(0,1){3}}  \put(0,0){\circle*{0.3}} \put(1,1){\circle*{0.3}} \put(2,2){\circle*{0.3}}
\put(4,2.4){$\sS' = \{\s_1+\s_2\}$}
\put(4,1.2){$\sZ = 2\,\ttS^1$}
\put(4,0){$\tcodim\,\cO = 3$}
\end{picture}
\end{footnotesize}
\end{center}
In both cases $\fl_\bR^\tss = \fsu(1,1)$.
\end{example}

\begin{example}[Period domain for $\bh = (2,1,2)$] \label{eg:B2P2}
We have $G_\bR = \tSO(1,4)^\circ$ and $\ttE_\varphi = \ttS^2$ so that $\sW^0 = \{ \one , (1) \}$.  Applying Theorem \ref{T:cSL2}, we find that $\Upsilon_D$ consists of a single element:
\begin{center}
\begin{footnotesize}
\setlength{\unitlength}{10pt}
\begin{picture}(10,4.5)(0,-0.5)
\put(0,0){\vector(1,0){3}} \put(0,0){\vector(0,1){3}}
\put(0,0){\circle*{0.3}} \put(0,2){\circle*{0.3}} \put(1,1){\circle*{0.3}} \put(2,0){\circle*{0.3}} \put(2,2){\circle*{0.3}}
\put(4,3.1){$\fl_\bR^\tss = \fsu(1,1)$}
\put(4,1.9){$\sS' = \{\s_2\}$}
\put(4,0.7){$\sZ = -2\ttS^1 + 2\ttS^2$}
\put(4,-0.5){$\tcodim\,\cO = 1$}
\end{picture}
\end{footnotesize}
\end{center}
Moreover, from \cite[Lemma III.20]{GGR} we may deduce that the codimension--one polarized orbit $\cO$ is closed, giving an example of a closed orbit that is polarized, but not by a Hodge--Tate degeneration. 
\end{example}

\begin{example}[Period domain for $\bh = (1,1,1,1,1)$] \label{eg:B2B}
We have $G_\bR = \tSO(3,2)$ and $\ttE_\varphi = \ttS^1+\ttS^2$ so that $\sW^0 = \{ \one \}$.  Applying Theorem \ref{T:cSL2}, we find that $\Upsilon_D$ consists of three elements: 
\begin{center}
\begin{footnotesize}
\setlength{\unitlength}{10pt}
\begin{picture}(12,6)
\put(0,0){\vector(1,0){5}} \put(0,0){\vector(0,1){5}}  \put(0,3){\circle*{0.3}} \put(1,4){\circle*{0.3}} \put(2,2){\circle*{0.3}} \put(3,0){\circle*{0.3}} \put(4,1){\circle*{0.3}}
\put(6,4.0){$\fl_\bR^\tss = \fu(1,1)$}
\put(6,2.8){$\sS' = \{ \s_1 \}$}
\put(6,1.6){$\sZ = 2\ttS^1 - \ttS^2$}
\put(6,0.4){$\tcodim\,\cO = 1$}
\end{picture}
\hsp{20pt}
\begin{picture}(11,6)
\put(0,0){\vector(1,0){5}} \put(0,0){\vector(0,1){5}}  \put(0,4){\circle*{0.3}} \put(1,1){\circle*{0.3}} \put(2,2){\circle*{0.3}} \put(3,3){\circle*{0.3}} \put(4,0){\circle*{0.3}}
\put(5,4.0){$\fl_\bR^\tss = \fu(1,1)$}
\put(5,2.8){$\sS' = \{ \s_2 \}$}
\put(5,1.6){$\sZ = -2\ttS^1 + 2 \ttS^2$}
\put(5,0.4){$\tcodim\,\cO = 1$}
\end{picture}
\hsp{20pt}
\begin{picture}(13,6)
\put(0,0){\vector(1,0){5}} \put(0,0){\vector(0,1){5}}  \put(0,0){\circle*{0.3}} \put(1,1){\circle*{0.3}} \put(2,2){\circle*{0.3}} \put(3,3){\circle*{0.3}} \put(4,4){\circle*{0.3}}
\put(6,4.8){$\fl_\bR^\tss = \fg_\bR$}
\put(6,3.7){$\sS' = \{\s_1,\s_2\}$}
\put(6,2.6){$\sZ = 2\,\ttS^1 + 2\,\ttS^2$}
\put(6,1.5){$\tcodim\,\cO = 4$}
\end{picture}
\end{footnotesize}
\end{center}
\end{example}

\begin{example}[$G_\bR = \tSU(2,1)$ and $\check D = \tFlag_{1,2}\bC^3$] \label{eg:A2B}
We have $\ttE_\varphi = \ttS^1+\ttS^2$, and consider the Mumford--Tate domain $D \subset \check D$ for the Hodge representation $(\fg_\bR,Q_\fg,\tAd,\varphi)$.  This domain is well studied; indeed, it is known that $\tbd(D)$ contains three $G_\bR$--orbits, all of which are polarized, \cf\cite{\GGKtcu , KP2013}.  We have $\sW^0 = \{\one\}$.  Applying Theorem \ref{T:cPMHS}, we find that $\Upsilon_D$ consists of three elements; the corresponding data are: 
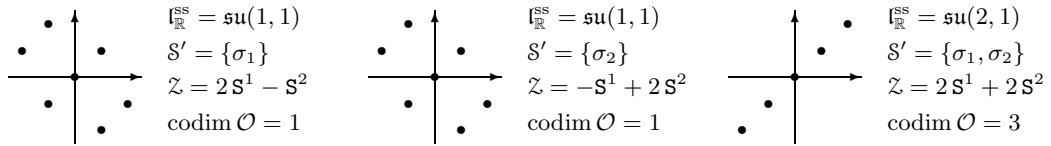
\begin{center}
\setlength{\unitlength}{10pt}
\begin{footnotesize}
\begin{picture}(11,6)(-3,-3)
\put(-2.5,0){\vector(1,0){5}} \put(0,-2.5){\vector(0,1){5}}
 \put(-2,1){\circle*{0.3}} \put(-1,-1){\circle*{0.3}} \put(-1,2){\circle*{0.3}} \put(0,0){\circle*{0.3}} \put(1,-2){\circle*{0.3}} \put(1,1){\circle*{0.3}} \put(2,-1){\circle*{0.3}}
\put(3.5,2){$\fl_\bR^\tss = \fsu(1,1)$}
\put(3.5,0.65){$\sS' = \{\s_1\}$}
\put(3.5,-0.65){$\sZ = 2\,\ttS^1 - \ttS^2$}
\put(3.5,-2){$\tcodim\,\cO = 1$}
\end{picture}
\hsp{20pt}
\begin{picture}(11,6)(-3,-3)
\put(-2.5,0){\vector(1,0){5}} \put(0,-2.5){\vector(0,1){5}}
 \put(-2,1){\circle*{0.3}} \put(-1,-1){\circle*{0.3}} \put(-1,2){\circle*{0.3}} \put(0,0){\circle*{0.3}} \put(1,-2){\circle*{0.3}} \put(1,1){\circle*{0.3}} \put(2,-1){\circle*{0.3}}
\put(3.5,2){$\fl_\bR^\tss = \fsu(1,1)$}
\put(3.5,0.65){$\sS' = \{\s_2\}$}
\put(3.5,-0.65){$\sZ = -\ttS^1 + 2\,\ttS^2$}
\put(3.5,-2){$\tcodim\,\cO = 1$}
\end{picture}
\hsp{20pt}
\begin{picture}(11,6)(-3,-3)
\put(-2.5,0){\vector(1,0){5}} \put(0,-2.5){\vector(0,1){5}}
\put(-2,-2){\circle*{0.3}} \put(-1,-1){\circle*{0.3}} \put(0,0){\circle*{0.3}} \put(1,1){\circle*{0.3}} \put(2,2){\circle*{0.3}}
\put(3.5,2){$\fl_\bR^\tss = \fsu(2,1)$}
\put(3.5,0.65){$\sS' = \{\s_1,\s_2\}$}
\put(3.5,-0.65){$\sZ = 2\,\ttS^1 + 2\, \ttS^2$}
\put(3.5,-2){$\tcodim\,\cO = 3$}
\end{picture}
\end{footnotesize}
\end{center}

\begin{remark} \label{R:egA2B}
   (a)
Observe that the first two Hodge diamonds in Example \ref{eg:A2B} are identical; in particular, they fail to distinguish the two distinct $G_\bR$--conjugacy classes of horizontal $\tSL(2)$s.
   
   (b)
Moreover, while the two nilpotent elements $N \in \tNilp(\fg_\bR)$ of these examples lie in the same $\tAd(G_\bC)$--orbit (the minimal orbit $\cN_\tmin$), they lie in distinct $\tAd(G_\bR)$--orbits.  This may be seen by computing the invariants $(\gamma(\sZ);\a'(\sZ))$ of \S\ref{S:rtno_R}, and observing that they differ.  For this, we work with the simple roots $\tilde\sS = (1)\sS = \{ -\s_1 , \s_1+\s_2 \}$.  Then $\tilde\sS_\fk = \{ \s_1+\s_2\}$ is a set of simple roots for $\fk_\bC = \fgl_2\bC$, and the noncompact root is $\a' = -\s_1$.  In both cases the compact characteristic vector satisfies
\[
  \gamma(\sZ) \ = \ \left( (\s_1+\s_2)(\sZ) \right) \ = \ (1) \,;
\]
however, in the first example we have $\a'(\sZ) = -2$, while in the second we have $\a'(\sZ) = 1$.
\end{remark}
\end{example}

\begin{example}[$G_\bR = G_2 \subset \tSO(3,4)$ and $\bh = (1,2,1,2,1)$] \label{eg:G2P1}
We have $\ttE_\varphi = \ttS^1$ and $G_\bR \subset \tSO(3,4)$, and consider the Mumford--Tate domain $D \subset \check D$ for the Hodge representation on $V_\bR = \bR^7$ with Hodge numbers $\bh = (1,2,1,2,1)$.  Kerr and Pearlstein have shown that $\tbd(D)$ contains three $G_\bR$--orbits, only one of which is polarized \cite[\S6.1.3]{KP2013}.

Here $\sW^0 = \{ \one , (2) \}$.  Applying Theorem \ref{T:cSL2}, we find that $\Upsilon_D$ consists of a single element:   
\begin{center}
\begin{footnotesize}
\setlength{\unitlength}{10pt}
\begin{picture}(15,5.5)
\put(0,0){\vector(1,0){5}} \put(0,0){\vector(0,1){5}}  
\put(0,3){\circle*{0.3}} \put(1,1){\circle*{0.3}} \put(1,4){\circle*{0.3}} \put(2,2){\circle*{0.3}} \put(3,0){\circle*{0.3}} \put(3,3){\circle*{0.3}} \put(4,1){\circle*{0.3}}
\put(6,4.1){$\fl_\bR^\tss = \fsu(1,1)$}
\put(6,2.9){$\sS' = \{\s_1\}$}
\put(6,1.7){$\sZ = 2\,\ttS^1 - 3\,\ttS^2$}
\put(6,0.5){$\tcodim\,\cO = 1$}
\end{picture}
\end{footnotesize}
\end{center}
\end{example}

\begin{example}[$G_\bR = G_2 \subset \tSO(3,4)$ and $\bh = (2,3,2)$] \label{eg:G2P2}
We have $\ttE_\varphi = \ttS^2$ and $G_\bR \subset \tSO(3,4)$, and consider the Mumford--Tate domain $D \subset \check D$ for the Hodge representation on $V_\bR = \bR^7$ with Hodge numbers $\bh = (2,3,2)$.  Kerr and Pearlstein have shown that $\tbd(D)$ contains three $G_\bR$--orbits, all of which are polarized \cite[\S6.1.3]{KP2013}.

Here $\sW^0 = \{ \one , (1) \}$.  Applying Theorem \ref{T:cSL2}, we find that $\Upsilon_D$ consists of three elements: 
\begin{center}
\begin{footnotesize}
\setlength{\unitlength}{10pt}
\begin{picture}(11,4.5)
\put(0,0){\vector(1,0){3.5}} \put(0,0){\vector(0,1){3.5}}  
\put(0,1){\circle*{0.3}} \put(0,2){\circle*{0.3}} \put(1,0){\circle*{0.3}} \put(1,1){\circle*{0.3}} \put(1,2){\circle*{0.3}} \put(2,0){\circle*{0.3}} \put(2,1){\circle*{0.3}}
\put(4,3.6){$\fl_\bR^\tss = \fsu(1,1)$}
\put(4,2.4){$\sS' = \{\s_2\}$}
\put(4,1.2){$\sZ = -\ttS^1 + 2\,\ttS^2$}
\put(4,0){$\tcodim\,\cO = 1$}
\end{picture}
\hsp{30pt}
\begin{picture}(11,4.5)
\put(0,0){\vector(1,0){3.5}} \put(0,0){\vector(0,1){3.5}}  
\put(0,0){\circle*{0.3}} \put(0,1){\circle*{0.3}} \put(1,0){\circle*{0.3}} \put(1,1){\circle*{0.3}} \put(1,2){\circle*{0.3}} \put(2,1){\circle*{0.3}} \put(2,2){\circle*{0.3}}
\put(4,3.6){$\fl_\bR^\tss = \fsu(1,1)$}
\put(4,2.4){$\sS' = \{\s_1+\s_2\}$}
\put(4,1.2){$\sZ = -\ttS^1 + 3\,\ttS^2$}
\put(4,0){$\tcodim\,\cO = 3$}
\end{picture}
\hsp{30pt}
\begin{picture}(11,4.5)
\put(0,0){\vector(1,0){3.5}} \put(0,0){\vector(0,1){3.5}}  
\put(0,0){\circle*{0.3}} \put(1,1){\circle*{0.3}} \put(2,2){\circle*{0.3}}
\put(4,3.6){$\fl_\bR^\tss = \fg_\bR$}
\put(4,2.4){$\sS' = \{\s_1,\s_2\}$}
\put(4,1.2){$\sZ = 2\,\ttS^2$}
\put(4,0){$\tcodim\,\cO = 5$}
\end{picture}
\end{footnotesize}
\end{center}
\end{example}

\begin{example}[$G_\bR = G_2 \subset \tSO(3,4)$ and $\bh = (1,1,1,1,1,1,1)$] \label{eg:G2B}
We have $\ttE_\varphi = \ttS^1+\ttS^2$ and $G_\bR \subset \tSO(3,4)$, and consider the Mumford--Tate domain $D \subset \check D$ for the Hodge representation on $V_\bR = \bR^7$ with Hodge numbers $\bh = (1,1,1,1,1,1,1)$.  Kerr and Pearlstein have shown that $\tbd(D)$ contains seven $G_\bR$--orbits, three of which are polarizable \cite[\S6.1.3]{KP2013}.

Here $\sW^0 = \{\one\}$.  Applying Theorem \ref{T:cSL2}, we find that $\Upsilon_D$ consists of three elements; they are given by: 
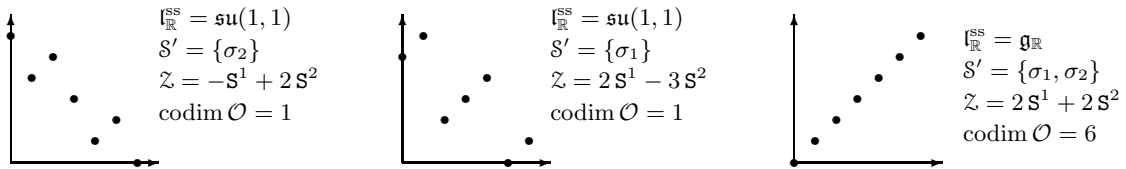
\begin{center}
\begin{footnotesize}
\setlength{\unitlength}{8pt}
\begin{picture}(14,8)
\put(0,0){\vector(1,0){7}} \put(0,0){\vector(0,1){7}}  
\put(0,6){\circle*{0.4}} \put(1,4){\circle*{0.4}} \put(2,5){\circle*{0.4}} \put(3,3){\circle*{0.4}} \put(4,1){\circle*{0.4}} \put(5,2){\circle*{0.4}} \put(6,0){\circle*{0.4}}
\put(7,6.5){$\fl_\bR^\tss = \fsu(1,1)$}
\put(7,5.0){$\sS' = \{\s_2\}$}
\put(7,3.5){$\sZ = -\ttS^1 + 2\,\ttS^2$}
\put(7,2){$\tcodim\,\cO = 1$}
\end{picture}
\hspace{30pt}
\begin{picture}(14,8)
\put(0,0){\vector(1,0){7}} \put(0,0){\vector(0,1){7}}  
\put(0,5){\circle*{0.4}} \put(1,6){\circle*{0.4}} \put(2,2){\circle*{0.4}} \put(3,3){\circle*{0.4}} \put(4,4){\circle*{0.4}} \put(5,0){\circle*{0.4}} \put(6,1){\circle*{0.4}}
\put(7,6.5){$\fl_\bR^\tss = \fsu(1,1)$}
\put(7,5.0){$\sS' = \{\s_1\}$}
\put(7,3.5){$\sZ = 2\,\ttS^1 - 3\,\ttS^2$}
\put(7,2.0){$\tcodim\,\cO = 1$}
\end{picture}
\hspace{30pt}
\begin{picture}(15,8)
\put(0,0){\vector(1,0){7}} \put(0,0){\vector(0,1){7}}  \put(0,0){\circle*{0.4}} \put(1,1){\circle*{0.4}} \put(2,2){\circle*{0.4}} \put(3,3){\circle*{0.4}} \put(4,4){\circle*{0.4}} \put(5,5){\circle*{0.4}} \put(6,6){\circle*{0.4}}
\put(8,5.5){$\fl_\bR^\tss = \fg_\bR$}
\put(8,4.0){$\sS' = \{\s_1,\s_2\}$}
\put(8,2.5){$\sZ = 2\,\ttS^1 + 2\,\ttS^2$}
\put(8,1){$\tcodim\,\cO = 6$}
\end{picture}
\end{footnotesize}
\end{center}
\end{example}

\begin{example}[Period domain for $\bh = (1,2,2,1)$] \label{eg:C3P13}
We have $G_\bR = \tSp(3,\bR)$ and $\ttE_\varphi = \ttS^1 + \ttS^3$.  In this case $\sW^0 = \{\one , (2) \}$. Applying Theorem \ref{T:cSL2} we find that $\Upsilon_D$ contains seven elements.  The corresponding data is:
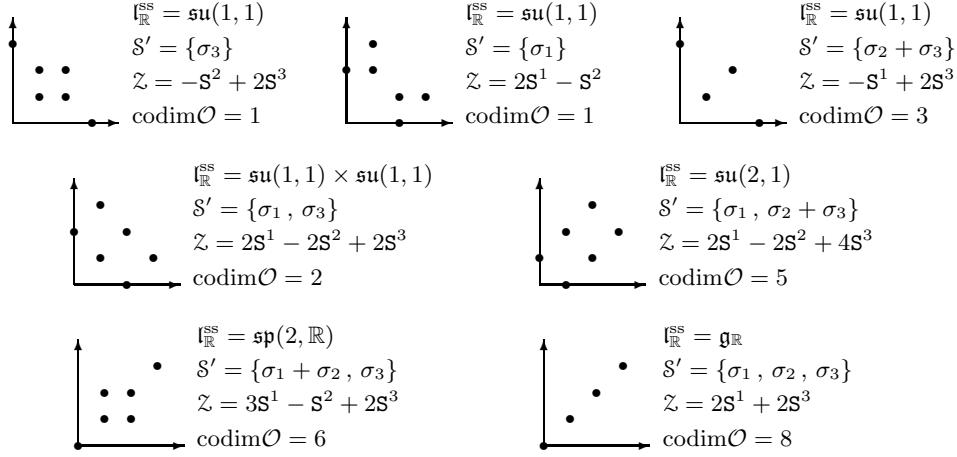
\begin{center}
\begin{footnotesize}
\setlength{\unitlength}{10pt}
\begin{picture}(11,5)
\put(0,0){\vector(1,0){4}} \put(0,0){\vector(0,1){4}}  
\put(0,3){\circle*{0.3}} \put(1,1){\circle*{0.3}} \put(1,2){\circle*{0.3}} \put(2,1){\circle*{0.3}} \put(2,2){\circle*{0.3}} \put(3,0){\circle*{0.3}}
\put(4.5,3.9){$\fl_\bR^\tss = \fsu(1,1)$}
\put(4.5,2.6){$\sS' = \{ \s_3 \}$}
\put(4.5,1.3){$\sZ = -\ttS^2 + 2\ttS^3$}
\put(4.5,0){$\tcodim\cO = 1$}
\end{picture}
\hsp{10pt}
\begin{picture}(11,5)
\put(0,0){\vector(1,0){4}} \put(0,0){\vector(0,1){4}}  
\put(0,2){\circle*{0.3}} \put(1,2){\circle*{0.3}} \put(1,3){\circle*{0.3}} \put(2,0){\circle*{0.3}} \put(2,1){\circle*{0.3}} \put(3,1){\circle*{0.3}}
\put(4.5,3.9){$\fl_\bR^\tss = \fsu(1,1)$}
\put(4.5,2.6){$\sS' = \{ \s_1 \}$}
\put(4.5,1.3){$\sZ = 2\ttS^1 - \ttS^2$}
\put(4.5,0){$\tcodim\cO = 1$}
\end{picture}
\hsp{10pt}
\begin{picture}(11,5)
\put(0,0){\vector(1,0){4}} \put(0,0){\vector(0,1){4}}  
\put(0,3){\circle*{0.3}} \put(1,1){\circle*{0.3}} \put(2,2){\circle*{0.3}} \put(3,0){\circle*{0.3}}
\put(4.5,3.9){$\fl_\bR^\tss = \fsu(1,1)$}
\put(4.5,2.6){$\sS' = \{ \s_2 + \s_3 \}$}
\put(4.5,1.3){$\sZ = -\ttS^1 + 2\ttS^3$}
\put(4.5,0){$\tcodim\cO = 3$}
\end{picture}
\vspace{10pt}\\
\begin{picture}(14,5)
\put(0,0){\vector(1,0){4}} \put(0,0){\vector(0,1){4}}  
\put(0,2){\circle*{0.3}} \put(1,1){\circle*{0.3}} \put(1,3){\circle*{0.3}} \put(2,0){\circle*{0.3}} \put(2,2){\circle*{0.3}} \put(3,1){\circle*{0.3}}
\put(4.5,3.9){$\fl_\bR^\tss = \fsu(1,1) \times \fsu(1,1)$}
\put(4.5,2.6){$\sS' = \{ \s_1 \,,\, \s_3 \}$}
\put(4.5,1.3){$\sZ = 2\ttS^1-2\ttS^2+2\ttS^3$}
\put(4.5,0){$\tcodim\cO = 2$}
\end{picture}
\hspace{30pt}
\begin{picture}(14,5)
\put(0,0){\vector(1,0){4}} \put(0,0){\vector(0,1){4}}  
\put(0,1){\circle*{0.3}} \put(1,0){\circle*{0.3}} \put(1,2){\circle*{0.3}} \put(2,1){\circle*{0.3}} \put(2,3){\circle*{0.3}} \put(3,2){\circle*{0.3}}
\put(4.5,3.9){$\fl_\bR^\tss = \fsu(2,1)$}
\put(4.5,2.6){$\sS' = \{ \s_1\,,\, \s_2+\s_3 \}$}
\put(4.5,1.3){$\sZ = 2\ttS^1-2\ttS^2+4\ttS^3$}
\put(4.5,0){$\tcodim\cO = 5$}
\end{picture}
\vspace{10pt}\\
\begin{picture}(14,5)
\put(0,0){\vector(1,0){4}} \put(0,0){\vector(0,1){4}}  
\put(0,0){\circle*{0.3}} \put(1,1){\circle*{0.3}} \put(1,2){\circle*{0.3}} \put(2,1){\circle*{0.3}} \put(2,2){\circle*{0.3}} \put(3,3){\circle*{0.3}}
\put(4.5,3.9){$\fl_\bR^\tss = \fsp(2,\bR)$}
\put(4.5,2.6){$\sS' = \{ \s_1+\s_2\,,\,\s_3 \}$}
\put(4.5,1.3){$\sZ = 3\ttS^1-\ttS^2+2\ttS^3$}
\put(4.5,0){$\tcodim\cO = 6$}
\end{picture}
\hspace{30pt}
\begin{picture}(14,5)
\put(0,0){\vector(1,0){4}} \put(0,0){\vector(0,1){4}}  
\multiput(0,0)(1,1){4}{\circle*{0.3}}
\put(4.5,3.9){$\fl_\bR^\tss = \fg_\bR$}
\put(4.5,2.6){$\sS' = \{ \s_1\,,\,\s_2\,,\,\s_3 \}$}
\put(4.5,1.3){$\sZ = 2\ttS^1+2\ttS^3$}
\put(4.5,0){$\tcodim\cO = 8$}
\end{picture}
\end{footnotesize}
\end{center}
\end{example}

\begin{example}[Period domain for $\bh = (2,1,1,2)$] \label{eg:C3P23}
We have $G_\bR = \tSp(3,\bR)$ and $\ttE_\varphi = \ttS^2 + \ttS^3$.  In this case $\sW^0 = \{\one , (1) \}$. Applying Theorem \ref{T:cPMHS} we find that there are three (conjugacy classes of) horizontal $\tSL(2)$s on the period domain $D$.  The corresponding data is:
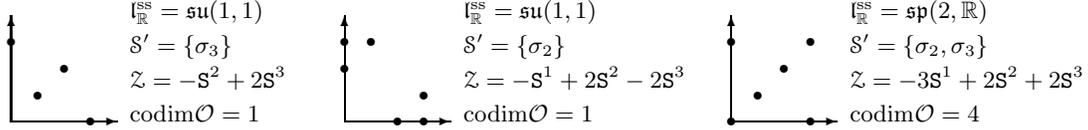
\begin{center}
\begin{footnotesize}
\setlength{\unitlength}{10pt}
\begin{picture}(11,4.5)
\put(0,0){\vector(1,0){4}} \put(0,0){\vector(0,1){4}}  
\put(0,3){\circle*{0.3}} \put(1,1){\circle*{0.3}} \put(2,2){\circle*{0.3}} \put(3,0){\circle*{0.3}}
\put(4.5,3.9){$\fl_\bR^\tss = \fsu(1,1)$}
\put(4.5,2.6){$\sS' = \{ \s_3 \}$}
\put(4.5,1.3){$\sZ = -\ttS^2+2\ttS^3$}
\put(4.5,0){$\tcodim\cO = 1$}
\end{picture}
\hsp{10pt}
\begin{picture}(13,4.5)
\put(0,0){\vector(1,0){4}} \put(0,0){\vector(0,1){4}}  
\put(0,2){\circle*{0.3}} \put(0,3){\circle*{0.3}} \put(1,3){\circle*{0.3}} 
\put(2,0){\circle*{0.3}} \put(3,0){\circle*{0.3}} \put(3,1){\circle*{0.3}}
\put(4.5,3.9){$\fl_\bR^\tss = \fsu(1,1)$}
\put(4.5,2.6){$\sS' = \{ \s_2 \}$}
\put(4.5,1.3){$\sZ = -\ttS^1+2\ttS^2-2\ttS^3$}
\put(4.5,0){$\tcodim\cO = 1$}
\end{picture}
\hsp{10pt}
\begin{picture}(13,4.5)
\put(0,0){\vector(1,0){4}} \put(0,0){\vector(0,1){4}}  
\put(0,0){\circle*{0.3}} \put(0,3){\circle*{0.3}} \put(1,1){\circle*{0.3}} \put(2,2){\circle*{0.3}} \put(3,0){\circle*{0.3}} \put(3,3){\circle*{0.3}}
\put(4.5,3.9){$\fl_\bR^\tss = \fsp(2,\bR)$}
\put(4.5,2.6){$\sS' = \{ \s_2 , \s_3\}$}
\put(4.5,1.3){$\sZ = -3\ttS^1+2\ttS^2+2\ttS^3$}
\put(4.5,0){$\tcodim\cO = 4$}
\end{picture}
\end{footnotesize}
\end{center}
\end{example}

\begin{example}[Period domain for $\bh = (1,1,1,1,1,1)$] \label{eg:C3B}
We have $G_\bR = \tSp(3,\bR)$ and $\ttE_\varphi = \ttS^1 + \ttS^2 + \ttS^3$.  In this case $\sW^0 = \{\one\}$. Applying Theorem \ref{T:cPMHS} we find that there are seven (conjugacy classes of) horizontal $\tSL(2)$s on the period domain $D$.  The corresponding data is:
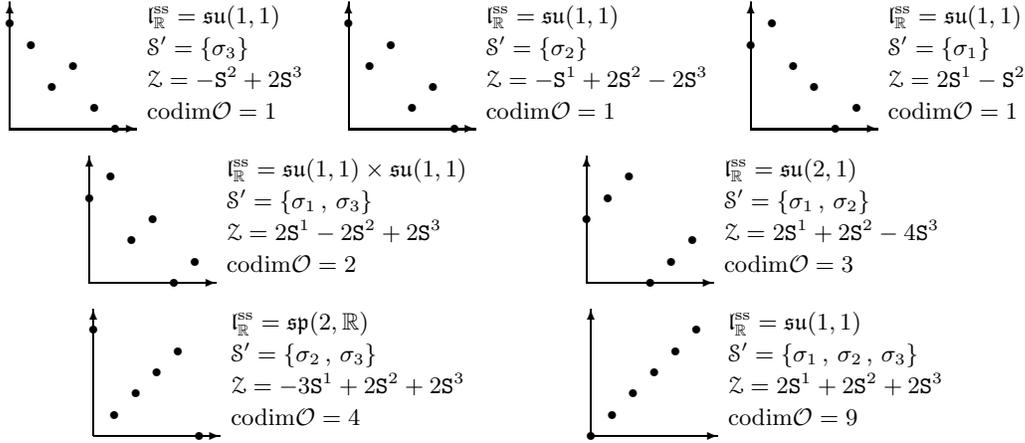
\begin{center}
\begin{footnotesize}
\setlength{\unitlength}{8pt}
\begin{picture}(14,6.5)
\put(0,0){\vector(1,0){6}} \put(0,0){\vector(0,1){6}}
\put(0,5){\circle*{0.35}} \put(1,4){\circle*{0.35}} \put(2,2){\circle*{0.35}} \put(3,3){\circle*{0.35}} \put(4,1){\circle*{0.35}} \put(5,0){\circle*{0.35}}
\put(6.5,5){$\fl_\bR^\tss = \fsu(1,1)$}
\put(6.5,3.5){$\sS' = \{ \s_3 \}$}
\put(6.5,2){$\sZ = -\ttS^2+2\ttS^3$}
\put(6.5,0.5){$\tcodim\cO = 1$}
\end{picture}
\hsp{10pt}
\begin{picture}(17,6.5)
\put(0,0){\vector(1,0){6}} \put(0,0){\vector(0,1){6}}
\put(0,5){\circle*{0.35}} \put(1,3){\circle*{0.35}} \put(2,4){\circle*{0.35}} \put(3,1){\circle*{0.35}} \put(4,2){\circle*{0.35}} \put(5,0){\circle*{0.35}}
\put(6.5,5){$\fl_\bR^\tss = \fsu(1,1)$}
\put(6.5,3.5){$\sS' = \{ \s_2 \}$}
\put(6.5,2){$\sZ = -\ttS^1+2\ttS^2-2\ttS^3$}
\put(6.5,0.5){$\tcodim\cO = 1$}
\end{picture}
\hsp{10pt}
\begin{picture}(15,6.5)
\put(0,0){\vector(1,0){6}} \put(0,0){\vector(0,1){6}}
\put(0,4){\circle*{0.35}} \put(1,5){\circle*{0.35}} \put(2,3){\circle*{0.35}} \put(3,2){\circle*{0.35}} \put(4,0){\circle*{0.35}} \put(5,1){\circle*{0.35}}
\put(6.5,5){$\fl_\bR^\tss = \fsu(1,1)$}
\put(6.5,3.5){$\sS' = \{ \s_1 \}$}
\put(6.5,2){$\sZ = 2\ttS^1-\ttS^2$}
\put(6.5,0.5){$\tcodim\cO = 1$}
\end{picture}
\vspace{5pt}\\
\begin{picture}(19,6.5)
\put(0,0){\vector(1,0){6}} \put(0,0){\vector(0,1){6}}
\put(0,4){\circle*{0.35}} \put(1,5){\circle*{0.35}} \put(2,2){\circle*{0.35}} \put(3,3){\circle*{0.35}} \put(4,0){\circle*{0.35}} \put(5,1){\circle*{0.35}}
\put(6.5,5){$\fl_\bR^\tss = \fsu(1,1)\times\fsu(1,1)$}
\put(6.5,3.5){$\sS' = \{ \s_1\,,\,\s_3 \}$}
\put(6.5,2){$\sZ = 2\ttS^1-2\ttS^2+2\ttS^3$}
\put(6.5,0.5){$\tcodim\cO = 2$}
\end{picture}
\hsp{30pt}
\begin{picture}(19,6.5)
\put(0,0){\vector(1,0){6}} \put(0,0){\vector(0,1){6}}
\put(0,3){\circle*{0.35}} \put(1,4){\circle*{0.35}} \put(2,5){\circle*{0.35}} \put(3,0){\circle*{0.35}} \put(4,1){\circle*{0.35}} \put(5,2){\circle*{0.35}}
\put(6.5,5){$\fl_\bR^\tss = \fsu(2,1)$}
\put(6.5,3.5){$\sS' = \{ \s_1\,,\,\s_2 \}$}
\put(6.5,2){$\sZ = 2\ttS^1+2\ttS^2-4\ttS^3$}
\put(6.5,0.5){$\tcodim\cO = 3$}
\end{picture}
\vspace{5pt}\\
\begin{picture}(19,6.5)
\put(0,0){\vector(1,0){6}} \put(0,0){\vector(0,1){6}}
\put(0,5){\circle*{0.35}} \put(1,1){\circle*{0.35}} \put(2,2){\circle*{0.35}} \put(3,3){\circle*{0.35}} \put(4,4){\circle*{0.35}} \put(5,0){\circle*{0.35}}
\put(6.5,5){$\fl_\bR^\tss = \fsp(2,\bR)$}
\put(6.5,3.5){$\sS' = \{ \s_2\,,\,\s_3 \}$}
\put(6.5,2){$\sZ = -3\ttS^1+2\ttS^2+2\ttS^3$}
\put(6.5,0.5){$\tcodim\cO = 4$}
\end{picture}
\hsp{30pt}
\begin{picture}(19,6.5)
\put(0,0){\vector(1,0){6}} \put(0,0){\vector(0,1){6}}
\multiput(0,0)(1,1){6}{\circle*{0.35}}
\put(6.5,5){$\fl_\bR^\tss = \fsu(1,1)$}
\put(6.5,3.5){$\sS' = \{ \s_1\,,\,\s_2\,,\,\s_3 \}$}
\put(6.5,2){$\sZ = 2\ttS^1+2\ttS^2+2\ttS^3$}
\put(6.5,0.5){$\tcodim\cO = 9$}
\end{picture}
\end{footnotesize}
\end{center}
\end{example}

\begin{example}[Period domain for $\bh = (2,4,2)$] \label{eg:D4P2}
We have $G_\bR = \tSO(4,4)^\circ$, $\ttE_\varphi = \ttS^2$ and $\sW^0 = \{ (1) , (3) , (4) \}$.  There are six ($G_\bR$--conjugacy classes of) horizontal $\tSL(2)$s.
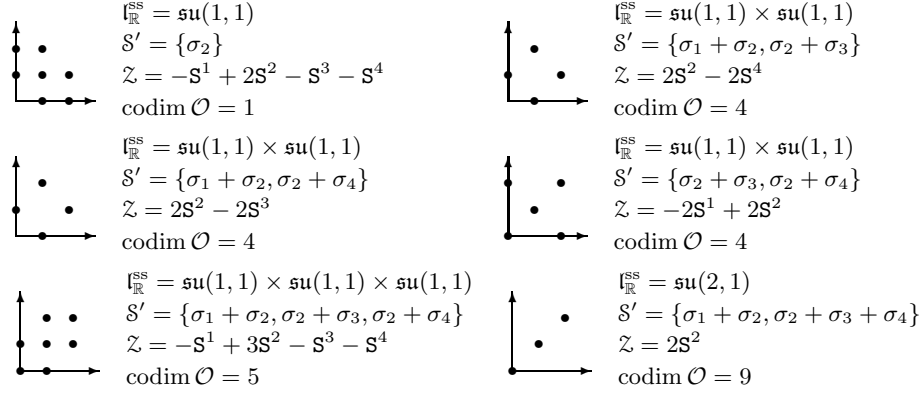
\begin{center}
\begin{footnotesize}
\setlength{\unitlength}{10pt}
\begin{picture}(17,4)
\put(0,0){\vector(1,0){3}} \put(0,0){\vector(0,1){3}}  
\put(0,1){\circle*{0.3}} \put(0,2){\circle*{0.3}} \put(1,0){\circle*{0.3}} \put(1,1){\circle*{0.3}} \put(1,2){\circle*{0.3}} \put(2,0){\circle*{0.3}} \put(2,1){\circle*{0.3}}
\put(4,3.1){$\fl_\bR^\tss = \fsu(1,1)$}
\put(4,1.9){$\sS' = \{\s_2\}$}
\put(4,0.7){$\sZ = -\ttS^1+2\ttS^2-\ttS^3-\ttS^4$}
\put(4,-0.5){$\tcodim\,\cO = 1$}
\end{picture}
\hsp{10pt}
\begin{picture}(17,4)
\put(0,0){\vector(1,0){3}} \put(0,0){\vector(0,1){3}}  
\put(0,1){\circle*{0.3}} \put(1,0){\circle*{0.3}} \put(1,2){\circle*{0.3}} \put(2,1){\circle*{0.3}}
\put(4,3.1){$\fl_\bR^\tss = \fsu(1,1)\times\fsu(1,1)$}
\put(4,1.9){$\sS' = \{\s_1+\s_2 , \s_2+\s_3\}$}
\put(4,0.7){$\sZ = 2\ttS^2-2\ttS^4$}
\put(4,-0.5){$\tcodim\,\cO = 4$}
\end{picture}
\vspace{10pt}\\
\begin{picture}(17,4)
\put(0,0){\vector(1,0){3}} \put(0,0){\vector(0,1){3}}  
\put(0,1){\circle*{0.3}} \put(1,0){\circle*{0.3}} \put(1,2){\circle*{0.3}} \put(2,1){\circle*{0.3}}
\put(4,3.1){$\fl_\bR^\tss = \fsu(1,1)\times\fsu(1,1)$}
\put(4,1.9){$\sS' = \{\s_1+\s_2 , \s_2+\s_4\}$}
\put(4,0.7){$\sZ = 2\ttS^2-2\ttS^3$}
\put(4,-0.5){$\tcodim\,\cO = 4$}
\end{picture}
\hsp{10pt}
\begin{picture}(17,4)
\put(0,0){\vector(1,0){3}} \put(0,0){\vector(0,1){3}}  
\put(0,0){\circle*{0.3}} \put(0,2){\circle*{0.3}} \put(1,1){\circle*{0.3}} \put(2,0){\circle*{0.3}} \put(2,2){\circle*{0.3}}
\put(4,3.1){$\fl_\bR^\tss = \fsu(1,1)\times\fsu(1,1)$}
\put(4,1.9){$\sS' = \{\s_2+\s_3 , \s_2+\s_4\}$}
\put(4,0.7){$\sZ = -2\ttS^1+2\ttS^2$}
\put(4,-0.5){$\tcodim\,\cO = 4$}
\end{picture}
\vspace{10pt}\\
\begin{picture}(17,4)
\put(0,0){\vector(1,0){3}} \put(0,0){\vector(0,1){3}}  
\put(0,0){\circle*{0.3}} \put(0,1){\circle*{0.3}} \put(1,0){\circle*{0.3}} \put(1,1){\circle*{0.3}} \put(1,2){\circle*{0.3}} \put(2,1){\circle*{0.3}} \put(2,2){\circle*{0.3}}
\put(4,3.1){$\fl_\bR^\tss = \fsu(1,1)\times\fsu(1,1)\times\fsu(1,1)$}
\put(4,1.9){$\sS' = \{\s_1+\s_2 , \s_2 + \s_3 , \s_2+\s_4\}$}
\put(4,0.7){$\sZ = -\ttS^1+3\ttS^2-\ttS^3-\ttS^4$}
\put(4,-0.5){$\tcodim\,\cO = 5$}
\end{picture}
\hsp{10pt}
\begin{picture}(17,4)
\put(0,0){\vector(1,0){3}} \put(0,0){\vector(0,1){3}}  
\put(0,0){\circle*{0.3}} \put(1,1){\circle*{0.3}} \put(2,2){\circle*{0.3}}
\put(4,3.1){$\fl_\bR^\tss = \fsu(2,1)$}
\put(4,1.9){$\sS' = \{\s_1+\s_2 , \s_2+\s_3+\s_4\}$}
\put(4,0.7){$\sZ = 2\ttS^2$}
\put(4,-0.5){$\tcodim\,\cO = 9$}
\end{picture}
\end{footnotesize}
\end{center}

\begin{remark}\label{R:egD4P2}
Note that the second and third ($G_\bR$--conjugacy classes of) horizontal $\tSL(2)$s in (the first row of) Example \ref{eg:D4P2} are \emph{not} distinguished by their Hodge diamonds.
\end{remark}
\end{example}

\begin{example}[Cattani--Kaplan]
In \cite[\S4]{MR496761} Cattani and Kaplan consider the case that $D$ is the period domain for Hodge numbers $\bh = (3,3,3)$, and find that there are five conjugacy classes of horizontal $\tSL(2)$s.  In the notation of Theorems \ref{T:cPMHS} and \ref{T:cSL2} (and the introduction to \S\ref{S:egPMHS}) those conjugacy classes are enumerated as follows.  We have $G_\bR = \tSO(3,6)$.  The grading element is $\ttE_\varphi = \ttS^3$, and $\sW^0$ is generated by the simple reflections $\{ (1) , (2) , (4)\}$.  
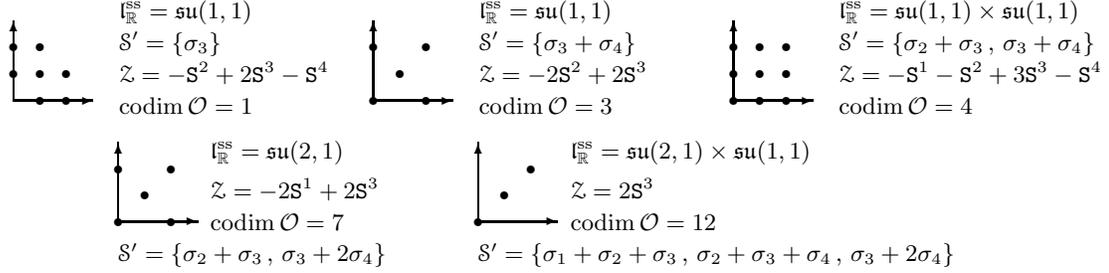
\begin{center}
\begin{footnotesize}
\setlength{\unitlength}{10pt}
\begin{picture}(12,4)
\put(0,0){\vector(1,0){3}} \put(0,0){\vector(0,1){3}}  
\put(0,1){\circle*{0.3}} \put(0,2){\circle*{0.3}} \put(1,0){\circle*{0.3}} \put(1,1){\circle*{0.3}} \put(1,2){\circle*{0.3}} \put(2,0){\circle*{0.3}} \put(2,1){\circle*{0.3}}
\put(4,3.1){$\fl_\bR^\tss = \fsu(1,1)$}
\put(4,1.9){$\sS' = \{ \s_3 \}$}
\put(4,0.7){$\sZ = -\ttS^2 + 2 \ttS^3 - \ttS^4$}
\put(4,-0.5){$\tcodim\,\cO = 1$}
\end{picture}
\hsp{10pt}
\begin{picture}(12,4)
\put(0,0){\vector(1,0){3}} \put(0,0){\vector(0,1){3}}  
\put(0,0){\circle*{0.3}} \put(0,2){\circle*{0.3}} \put(1,1){\circle*{0.3}} \put(2,0){\circle*{0.3}} \put(2,2){\circle*{0.3}}
\put(4,3.1){$\fl_\bR^\tss = \fsu(1,1)$}
\put(4,1.9){$\sS' = \{ \s_3+\s_4\}$}
\put(4,0.7){$\sZ = -2\ttS^2 + 2 \ttS^3$}
\put(4,-0.5){$\tcodim\,\cO = 3$}
\end{picture}
\hsp{10pt}
\begin{picture}(12,4)
\put(0,0){\vector(1,0){3}} \put(0,0){\vector(0,1){3}}  \put(0,0){\circle*{0.3}} \put(0,1){\circle*{0.3}} \put(0,2){\circle*{0.3}} \put(1,0){\circle*{0.3}} \put(1,1){\circle*{0.3}} \put(1,2){\circle*{0.3}} \put(2,0){\circle*{0.3}} \put(2,1){\circle*{0.3}} \put(2,2){\circle*{0.3}}
\put(4,3.1){$\fl_\bR^\tss = \fsu(1,1)\times\fsu(1,1)$}
\put(4,1.9){$\sS' = \{ \s_2+\s_3 \,,\, \s_3+\s_4 \}$}
\put(4,0.7){$\sZ = -\ttS^1-\ttS^2+3\ttS^3-\ttS^4$}
\put(4,-0.5){$\tcodim\,\cO = 4$}
\end{picture}
\vspace{5pt}\\
\begin{picture}(10,5.5)(0,-1.5)
\put(0,0){\vector(1,0){3}} \put(0,0){\vector(0,1){3}}  \put(0,0){\circle*{0.3}} \put(0,2){\circle*{0.3}} \put(1,1){\circle*{0.3}} \put(2,0){\circle*{0.3}} \put(2,2){\circle*{0.3}}
\put(0,-1.5){$\sS' = \{ \s_2+\s_3 \,,\, \s_3 + 2\s_4 \}$}
\put(3.5,2.4){$\fl_\bR^\tss = \fsu(2,1)$}
\put(3.5,0.9){$\sZ = -2\ttS^1 + 2\ttS^3$}
\put(3.5,-0.3){$\tcodim\,\cO = 7$}
\end{picture}
\hsp{30pt}
\begin{picture}(18,5.5)(0,-1.5)
\put(0,0){\vector(1,0){3}} \put(0,0){\vector(0,1){3}}  
\put(0,0){\circle*{0.3}} \put(1,1){\circle*{0.3}} \put(2,2){\circle*{0.3}}
\put(0,-1.5){$\sS' = \{ \s_1+\s_2+\s_3 \,,\, \s_2+\s_3+\s_4 \,,\, \s_3+2\s_4\}$}
\put(3.5,2.4){$\fl_\bR^\tss = \fsu(2,1)\times\fsu(1,1)$}
\put(3.5,0.9){$\sZ = 2\ttS^3$}
\put(3.5,-0.3){$\tcodim\,\cO = 12$}
\end{picture}
\end{footnotesize}
\end{center}
\end{example}

\appendix

\section{Non-compact real forms} \label{S:excp_gR}

The classical non-compact simple real forms $\fg_\bR$ that contain a compact Cartan subalgebra are listed in Table \ref{t:cl_gR} along with their maximal compact subalgebras; there $a,b >0$.  Recall that 
\[
  \fso(2) \ \simeq \ \bR \,,\quad
  \fsp(1) \ \simeq \ \fsu(2) \,,\quad
  \fsp(2) \ \simeq \ \fso(5) \,,\quad
  \fsu(4) \ \simeq \ \fso(6) \,.
\]
\begin{table}[!h]
\caption{The classical real forms}
\begin{center}
\renewcommand{\arraystretch}{1.3}
\begin{tabular}{c|ccccc}
  $\fg_\bR$ & $\fsu(a,b)$ & $\fsp(a,b)$  
            & $\fsp(n,\bR)$ & $\fso(2a,b)$ & $\fso^*(2n)$ \\
  \hline
  $\fk_\bR$ & $\mathfrak{s}(\fu(a)\op\fu(b))$ & $\fsp(a)\op\fsp(b)$ 
            & $\fu(n)$ & $\fso(2a)\op\fso(b)$ & $\fu(n)$
\end{tabular}
\end{center}
\label{t:cl_gR}
\end{table}

Table \ref{t:excp_gR} lists those non-compact real forms $\fg_\bR$ of the exceptional simple complex Lie algebras $\fg$ that contain a compact Cartan subalgebra.  The table also lists the the maximal compact Lie subalgebra $\fk_\bR \subset \fg_\bR$, and the real rank $\trank_\bR \,\fg_\bR$ of $\fg_\bR$.  In the first column we give the two common notations for the real forms; in the case of the second, the notation $X_n(s)$ indicates the complex form $X_n$ of the algebra, and $s = \tdim\,\fk^\perp_\bR - \tdim\,\fk_\bR$.
\begin{table}[!h]
\caption{The exceptional real forms}
\begin{center}
\renewcommand{\arraystretch}{1.3}
\begin{tabular}{|c|c|c|c|}
\hline
  $\fg$ & $\fg_\bR$ & $\fk_\bR$ & $\trank_\bR\,\fg_\bR$ \\ \hline
        & $\mathrm{E\,II} = E_6(2)$ & $\fsu(6) \op \fsu(2)$ & $4$ \\ \cline{2-4}
  \rb{$\fe_6$} & $\mathrm{E\,III} = E_6(-14)$ & $\fso(10)\op\bR$ & 2 \\ \hline
        & $\mathrm{E\,V} = E_7(7)$ & $\fsu(8)$ & $7$ \\ \cline{2-4}
  $\fe_7$ & $\mathrm{E\,VI} = E_7(-5)$ & $\fso(12)\op\fsu(2)$ & $4$ \\ \cline{2-4}
        & $\mathrm{E\,VII} = E_7(-25)$ & $\fe_6\op\bR$ & $3$ \\ \hline
        & $\mathrm{E\,VIII} = E_8(8)$ & $\fso(16)$ & $8$ \\ \cline{2-4}
  \rb{$\fe_8$} & $\mathrm{E\,IX} = E_8(-24)$ & $\fe_7\op\fsu(2)$ & $4$ \\ \hline
        & $\mathrm{F\,I} = F_4(4)$ & $\fsp(3)\op\fsu(2)$ & $4$ \\ \cline{2-4}
  \rb{$\ff_4$} & $\mathrm{F\,II}= F_4(-20)$ & $\fso(9)$ & $1$ \\ \hline
  $\fg_2$ & $\mathrm{G} = G_2(2)$ & $\fsu(2) \op \fsu(2)$ & $2$ \\
\hline
\end{tabular}
\end{center}
\label{t:excp_gR}
\end{table}

\def\cprime{$'$} \def\Dbar{\leavevmode\lower.6ex\hbox to 0pt{\hskip-.23ex
  \accent"16\hss}D}
\providecommand{\bysame}{\leavevmode\hbox to3em{\hrulefill}\thinspace}
\providecommand{\MR}{\relax\ifhmode\unskip\space\fi MR }
\providecommand{\MRhref}[2]{%
  \href{http://www.ams.org/mathscinet-getitem?mr=#1}{#2}
}
\providecommand{\href}[2]{#2}


\begin{thebibliography}{FHW06}

\bibitem[AMN10]{MR2668874}
Andrea Altomani, Costantino Medori, and Mauro Nacinovich, \emph{Orbits of real
  forms in complex flag manifolds}, Ann. Sc. Norm. Super. Pisa Cl. Sci. (5)
  \textbf{9} (2010), no.~1, 69--109.

\bibitem[BC76a]{MR0417306}
P.~Bala and R.~W. Carter, \emph{Classes of unipotent elements in simple
  algebraic groups. {I}}, Math. Proc. Cambridge Philos. Soc. \textbf{79}
  (1976), no.~3, 401--425.

\bibitem[BC76b]{MR0417307}
\bysame, \emph{Classes of unipotent elements in simple algebraic groups. {II}},
  Math. Proc. Cambridge Philos. Soc. \textbf{80} (1976), no.~1, 1--17.

\bibitem[BP13]{MR3133298}
Patrick Brosnan and Gregory Pearlstein, \emph{On the algebraicity of the zero
  locus of an admissible normal function}, Compos. Math. \textbf{149} (2013),
  no.~11, 1913--1962.

\bibitem[BPR15]{BPR}
P.~Brosnan, G.~Pearlstein, and C.~Robles, \emph{Nilpotent cones and their
  representation theory}, In preparation; an expository chapter in a collection
  honoring Steve Zucker's work, 2015.

\bibitem[CDK95]{MR1273413}
Eduardo Cattani, Pierre Deligne, and Aroldo Kaplan, \emph{On the locus of
  {H}odge classes}, J. Amer. Math. Soc. \textbf{8} (1995), no.~2, 483--506.

\bibitem[CK77]{MR0432925}
Eduardo~H. Cattani and Aroldo~G. Kaplan, \emph{Extension of period mappings for
  {H}odge structures of weight two}, Duke Math. J. \textbf{44} (1977), no.~1,
  1--43.

\bibitem[CK78]{MR496761}
\bysame, \emph{Horizontal {${\rm SL}_{2}$}-orbits in flag domains}, Math. Ann.
  \textbf{235} (1978), no.~1, 17--35.

\bibitem[CK82a]{MR664326}
Eduardo Cattani and Aroldo Kaplan, \emph{Polarized mixed {H}odge structures and
  the local monodromy of a variation of {H}odge structure}, Invent. Math.
  \textbf{67} (1982), no.~1, 101--115.

\bibitem[CK82b]{CattaniKaplanIHES}
\bysame, \emph{On the sl(2)-orbits in hodge theory}, IHES pre-pub M/82/58,
  October 1982.

\bibitem[CK85]{MR786917}
\bysame, \emph{Sur la cohomologie {$L_2$} et la cohomologie d'intersection \`a
  coefficients dans une variation de structure de {H}odge}, C. R. Acad. Sci.
  Paris S\'er. I Math. \textbf{300} (1985), no.~11, 351--353.

\bibitem[CK89]{MR1042802}
\bysame, \emph{Degenerating variations of {H}odge structure}, Ast\'erisque
  (1989), no.~179-180, 9, 67--96, Actes du Colloque de Th{\'e}orie de Hodge
  (Luminy, 1987).

\bibitem[CKS86]{MR840721}
Eduardo Cattani, Aroldo Kaplan, and Wilfried Schmid, \emph{Degeneration of
  {H}odge structures}, Ann. of Math. (2) \textbf{123} (1986), no.~3, 457--535.

\bibitem[CKS87]{MR870728}
\bysame, \emph{{$L^2$} and intersection cohomologies for a polarizable
  variation of {H}odge structure}, Invent. Math. \textbf{87} (1987), no.~2,
  217--252.

\bibitem[CM93]{MR1251060}
David~H. Collingwood and William~M. McGovern, \emph{Nilpotent orbits in
  semisimple {L}ie algebras}, Van Nostrand Reinhold Mathematics Series, Van
  Nostrand Reinhold Co., New York, 1993.

\bibitem[{\v{C}}S09]{MR2532439}
Andreas {\v{C}}ap and Jan Slov{\'a}k, \emph{Parabolic geometries. {I}},
  Mathematical Surveys and Monographs, vol. 154, American Mathematical Society,
  Providence, RI, 2009, Background and general theory.

\bibitem[Del]{DeligneApp}
Pierre Deligne, \emph{Structures de hodge mixtes r\'eelles}, appendix to
  \cite{CattaniKaplanIHES}.

\bibitem[Del71]{MR0498551}
\bysame, \emph{Th\'eorie de {H}odge. {II}}, Inst. Hautes \'Etudes Sci. Publ.
  Math. (1971), no.~40, 5--57.

\bibitem[Djo87]{MR891636}
Dragomir~{\v{Z}}. Djokovi{\'c}, \emph{Proof of a conjecture of {K}ostant},
  Trans. Amer. Math. Soc. \textbf{302} (1987), no.~2, 577--585.

\bibitem[Dyn52]{MR0047629}
E.~B. Dynkin, \emph{Semisimple subalgebras of semisimple {L}ie algebras}, Mat.
  Sbornik N.S. \textbf{30(72)} (1952), 349--462 (3 plates).

\bibitem[Dyn57]{MR0047629_trans}
\bysame, \emph{Semisimple subalgebras of semisimple {L}ie algebras}, Amer.
  Math. Soc. Trans. \textbf{6} (1957), 111--244, Translation of
  \cite{MR0047629}.

\bibitem[FHW06]{MR2188135}
Gregor Fels, Alan Huckleberry, and Joseph~A. Wolf, \emph{Cycle spaces of flag
  domains}, Progress in Mathematics, vol. 245, Birkh\"auser Boston Inc.,
  Boston, MA, 2006, A complex geometric viewpoint.

\bibitem[GGK12]{MR2918237}
Mark Green, Phillip Griffiths, and Matt Kerr, \emph{Mumford-{T}ate groups and
  domains: their geometry and arithmetic}, Annals of Mathematics Studies, vol.
  183, Princeton University Press, Princeton, NJ, 2012.

\bibitem[GGK13]{MR3115136}
\bysame, \emph{Hodge theory, complex geometry, and representation theory}, CBMS
  Regional Conference Series in Mathematics, vol. 118, Published for the
  Conference Board of the Mathematical Sciences, Washington, DC, 2013.

\bibitem[GGR14]{GGR}
Mark Green, Phillip Griffiths, and Colleen Robles, \emph{Extremal degenerations
  of polarized {H}odge structures}, Preprint, arXiv:1403.0646, 2014.

\bibitem[HP14]{HayPearl}
T.~Hayama and G.~Pearlstein, \emph{Asymptotics of degenerations of mixed
  {H}odge structures}, arXiv:1403.1971, 2014.

\bibitem[KK85]{MR804358}
Masaki Kashiwara and Takahiro Kawai, \emph{The {P}oincar\'e lemma for a
  variation of polarized {H}odge structure}, Proc. Japan Acad. Ser. A Math.
  Sci. \textbf{61} (1985), no.~6, 164--167.

\bibitem[Kna02]{MR1920389}
Anthony~W. Knapp, \emph{Lie groups beyond an introduction}, second ed.,
  Progress in Mathematics, vol. 140, Birkh\"auser Boston Inc., Boston, MA,
  2002.

\bibitem[Kos59]{MR0114875}
Bertram Kostant, \emph{The principal three-dimensional subgroup and the {B}etti
  numbers of a complex simple {L}ie group}, Amer. J. Math. \textbf{81} (1959),
  973--1032.

\bibitem[KP13]{KP2013}
M.~Kerr and G.~Pearlstein, \emph{Na{\"\i}ve boundary strata and nilpotent
  orbits}, Preprint, 2013.

\bibitem[KR14]{KR1}
M.~Kerr and C.~Robles, \emph{Hodge theory and real orbits in flag varieties},
  arXiv:1407.4507, 2014.

\bibitem[KR15]{KR2}
Matt Kerr and Colleen Robles, \emph{Partial orders and polarized relations on
  limit mixed hodge structures}, In prepration, 2015.

\bibitem[KU09]{MR2465224}
Kazuya Kato and Sampei Usui, \emph{Classifying spaces of degenerating polarized
  {H}odge structures}, Annals of Mathematics Studies, vol. 169, Princeton
  University Press, Princeton, NJ, 2009.

\bibitem[No{\"e}98]{MR1600330}
Alfred~G. No{\"e}l, \emph{Nilpotent orbits and theta-stable parabolic
  subalgebras}, Represent. Theory \textbf{2} (1998), 1--32 (electronic).

\bibitem[Rob14]{MR3217458}
C.~Robles, \emph{Schubert varieties as variations of {H}odge structure},
  Selecta Math. (N.S.) \textbf{20} (2014), no.~3, 719--768.

\bibitem[Sch73]{MR0382272}
Wilfried Schmid, \emph{Variation of {H}odge structure: the singularities of the
  period mapping}, Invent. Math. \textbf{22} (1973), 211--319.

\bibitem[Sek87]{MR867991}
Jir{\=o} Sekiguchi, \emph{Remarks on real nilpotent orbits of a symmetric
  pair}, J. Math. Soc. Japan \textbf{39} (1987), no.~1, 127--138.

\bibitem[Usu93]{MR1245714}
Sampei Usui, \emph{A numerical criterion for admissibility of semi-simple
  elements}, Tohoku Math. J. (2) \textbf{45} (1993), no.~4, 471--484.

\bibitem[Vin75]{MR0506488}
{\`E}.~B. Vinberg, \emph{The classification of nilpotent elements of graded
  {L}ie algebras}, Dokl. Akad. Nauk SSSR \textbf{225} (1975), no.~4, 745--748.

\bibitem[Wol69]{MR0251246}
Joseph~A. Wolf, \emph{The action of a real semisimple group on a complex flag
  manifold. {I}. {O}rbit structure and holomorphic arc components}, Bull. Amer.
  Math. Soc. \textbf{75} (1969), 1121--1237.

\bibitem[Zuc79]{MR534758}
Steven Zucker, \emph{Hodge theory with degenerating coefficients. {$L_{2}$}
  cohomology in the {P}oincar\'e metric}, Ann. of Math. (2) \textbf{109}
  (1979), no.~3, 415--476.

\end{thebibliography}
\end{document}